\documentclass[a4paper,10pt]{amsart}

\usepackage{amsmath,amssymb,latexsym}
\usepackage{amscd}
\usepackage[OT4]{fontenc}
\usepackage{lmodern}
\usepackage{mathrsfs}
\usepackage{amsthm}
\usepackage{amsfonts}
\usepackage{url}
\usepackage[fleqn,tbtags]{mathtools}
\usepackage{mathrsfs}
\usepackage{enumerate}
\usepackage{graphicx}
\usepackage{color}
\usepackage[all]{xy}
\usepackage{tikz}
\usepackage{marginnote}
\usepackage[foot]{amsaddr}
\usepackage{orcidlink}
\usepackage{todonotes}

\usepackage{dsfont}

\usepackage{dirtytalk}
\usepackage{csquotes}

\usepackage{soul}

\usepackage{xspace}

\usepackage{hyperref}
\usepackage{mathtools}
\usepackage{tikz}
\usepackage{pgfplots}
\usepgfplotslibrary{fillbetween}
\pgfplotsset{compat=1.18}

\newtheorem{theorem}{Theorem}
\newtheorem{lemma}[theorem]{Lemma}
\newtheorem{corollary}[theorem]{Corollary}
\newtheorem{proposition}[theorem]{Proposition}
\newtheorem{conjecture}[theorem]{Conjecture}

\newtheorem*{crucial lemma}{Lemma \ref{lem:nice polynomial2}}

\theoremstyle{definition}
\newtheorem{definition}[theorem]{Definition}
\newtheorem{remark}[theorem]{Remark}
\newtheorem{example}[theorem]{Example}

\newcommand{\R}{\mathbb{R}}
\newcommand{\C}{\mathbb{C}}
\newcommand{\N}{\mathbb{N}}

\newcommand{\supp}{\mbox{supp}\,}
\newcommand{\dist}{\mbox{dist}\,}

\newcommand{\Pad}{P^{ad}}
\newcommand{\detP}{\operatorname{det}\!P}
\newcommand{\curl}{\operatorname{curl}}
\newcommand{\charcone}{\operatorname{Char}}

\newcommand{\wniceon}[2]{$#1$-degenerate on~#2}
\newcommand{\niceon}[1]{degenerate on~#1}
\newcommand{\veryniceon}[1]{orthogonally degenerate on~#1}

\allowdisplaybreaks

\title[Zero solutions, solvability, and {R}unge approximation]{From zero solutions with partially bounded supports to solvability and {R}unge approximation for partial differential equations}

\author[T.~Cia\'s]{Tomasz Cia\'s$^1$ \orcidlink{0000-0001-7119-7814}}

\address{$^1$Faculty of Mathematics and Computer Science, Adam Mickiewicz University, Pozna\'n, ul.~Uniwersytetu Pozna\'nskiego 4, 61-614 Pozna\'n, Poland}
\email{tomasz.cias@amu.edu.pl}

\author[T.~Kalmes]{Thomas Kalmes$^2$ \orcidlink{0000-0001-7542-1334}}

\address{$^2$Faculty of Mathematics, Chemnitz University of Technology, 09107 Chemnitz, Germany}
\email{thomas.kalmes@math.tu-chemnitz.de}

\date{}

\begin{document}

\begin{abstract}
	The existence of non-trivial solutions of homogeneous partial differential equations with prescribed support properties plays a fundamental role in the theory of linear partial differential operators. In this article, we establish the existence of smooth zero solutions with partially bounded supports for a broad class of constant coefficient partial differential operators.
    
    As applications, we obtain new geometric results concerning solvability and approximation for partial differential equations. In particular, we derive geometric characterizations of $P$-convexity for supports for a large class of non-elliptic operators, thereby extending classical geometric ideas of H\"ormander. By a theorem of Malgrange, $P$-convexity for supports is equivalent to the surjectivity of the differential operator on spaces of smooth functions and, more generally, on local subspaces of distributions of finite order. As a further consequence, we prove that the kernel of every surjective semi-elliptic operator satisfies condition ($\Omega$) implying parameter dependence results.
    
    We also investigate Runge-type approximation phenomena. We obtain geometric characterizations of Runge pairs for smooth functions, distributions, and spaces of smooth Whitney jets. Finally, we develop a general framework for Runge approximation for square systems of constant coefficient partial differential equations. As applications, we recover by alternative methods known Runge approximation results for Beltrami fields and for the three-dimensional unsteady Stokes system, and we complement them by corresponding approximation theorems in spaces of smooth Whitney jets. In particular, this yields approximation up to the boundary for solutions on suitable domains with H\"older continuous boundary.\\
	
	
\end{abstract}

\maketitle

\section{Introduction}\label{section:introduction}

The existence of non-trivial smooth solutions of homogeneous partial differential equations (i.e.~smooth zero solutions) with prescribed support properties has long played a fundamental role in the theory of linear partial differential operators. Solutions of this type are used to establish necessary conditions for surjectivity of partial differential operators on so-called local subspaces (in the sense of H\"ormander \cite[Chapter 10]{Hor2}) of the space of finite-order distributions, such as spaces of smooth functions and local Sobolev spaces. By the seminal work of Malgrange \cite{Malgrange1955}, a constant coefficient linear partial differential operator $P(D)$ is surjective on such spaces over an open subset $X$ of $\R^d$ precisely when $X$ is $P$-convex for supports (see also \cite[Section 10.6]{Hor2}). Despite the classical nature of this characterization, a geometric description of this condition for given (classes of) partial differential operators is notoriously difficult. Convex open sets are $P$-convex for supports \cite{Malgrange1955} (whenever $P\not\equiv 0$), every open set is $P$-convex for supports whenever $P$ is elliptic, and there is a complete geometric characterization of $P$-convexity for supports in the two dimensional case due to H\"ormander \cite[Theorem 10.8.3]{Hor2}. In arbitrary dimensions, however, the problem of geometrically characterizing open subsets of $\R^d$ which are $P$-convex for supports is far from being solved. For second order operators with principal part equal to the wave operator, there is a characterization \cite[Theorem 10.8.6]{Hor2}, essentially due to Persson \cite{Persson1981} (see also \cite{Persson1993,Tintarev1992,Tintarev1988} for general operators of real principal type). Moreover, Nakane \cite{Nakane1979} gave a geometric characterization of $P$-convexity for supports for polynomials which act along a subspace of $\R^d$ and are elliptic there (see also \cite[Theorem 10.8.5]{Hor2}; the special case of first order operators was studied in \cite{Zachmanoglou69}).

Besides their role in the study of P-convexity for supports, smooth zero solutions with partially bounded supports also play an important role in Runge-type approximation theorems, i.e.~approximability of solutions of homogeneous partial differential equations on a small set by solutions of the same equation on a larger set. Starting from Runge's classical theorem on rational approximation and its far-reaching generalization by Lax and Malgrange to elliptic differential operators, Runge approximation results have become a classical topic in the theory of partial differential equations \cite{Browder1962b,Browder1962}. In recent years, this subject has experienced renewed interest. For variable coefficients parabolic operators of second order, such approximation results have been studied by Enciso, Garc\'{\i}a-Ferrero, Peralta-Salas \cite{EnGaPe19} and have been complemented by Enciso, Peralta-Salas \cite{EnPe21} as well as Shlapunov, Vilkov \cite{ShVi24}. For several classes of non-elliptic constant coefficient partial differential operators, Runge approximation theorems are due to the second author \cite{Kalmes2021} as well as to Debrouwere and the second author \cite{DebKalmes2024}. Beyond scalar equations, Runge approximation for systems has attracted considerable recent interest, not least because of its striking applications in fluid dynamics. Notable examples include the work of Higaki and Sueur on the three-dimensional unsteady Stokes system \cite{HiSu25} and the influential contributions of Enciso and Peralta-Salas on the three-dimensional Euler equations \cite{EnPe12,EnPe15} and, jointly with Lucá, on the Navier--Stokes equations \cite{EnLuPe17}. These developments suggest that approximation results for systems can serve as a useful tool in the study of fluid equations and related PDE models. While the above mentioned results deal with Runge theorems for solutions of homogeneous partial differential equations on open subsets of $\R^d$, approximation results for solutions on closed subsets of $\R^d$, in the context of smooth Whitney jets, have very recently been studied in \cite{CiasKalmes2025}. For non-elliptic differential operators, the existence of zero solutions with partially bounded supports naturally leads to geometric conditions for Runge-type approximation phenomena.

The central aim of this article is twofold. Our first objective is to establish the existence of such solutions for a large class of constant coefficient partial differential operators. This is done in Section \ref{section:zero solutions}. Our second objective is to systematically apply the existence results as described above. More precisely, building on these existence results and earlier work of the authors, Section \ref{sec:p-convexity} provides a geometric characterization of $P$-convexity for supports for certain partial differential operators. In particular, the class of operators covered by our results contains all semi-elliptic operators.

While Section \ref{sec:p-convexity} focuses on solvability, the existence of zero solutions with partially bounded supports also has important consequences for approximation theory. In Section \ref{sec:p-Runge pairs} we use our existence results, again in combination with earlier work of the authors, to derive geometric criteria for Runge-type approximation phenomena, not only in the classical setting of smooth functions and distributions, but also in the setting of smooth Whitney jets.

The final Section \ref{sec:systems} of the paper is devoted to systems of partial differential equations. We develop a general framework for Runge pairs for square systems of constant coefficient linear partial differential equations. This framework is formulated for smooth functions, distributions, and spaces of smooth Whitney jets. As consequences of this general framework, we give alternative proofs of the above mentioned Runge approximation results (on compact subsets for zero solutions on open supersets) by Enciso and Peralta-Salas \cite{EnPe15} and by Higaki and Sueur \cite{HiSu25}, respectively. In addition, our approach yields corresponding Runge-type results in spaces of smooth Whitney jets which establish approximation theorems up to the boundary for smooth functions on domains with H\"older continuous boundary.

The geometric characterization of $P$-convexity for supports obtained in Section \ref{sec:p-convexity} also has further consequences of a more functional analytic nature. We prove that the kernel of every surjective semi-elliptic operator satisfies condition ($\Omega$) of Vogt and Wagner. By the celebrated splitting theorem of Vogt and Wagner, this implies surjectivity of semi-elliptic operators acting on spaces of vector-valued smooth functions \cite{Vogt1983}, thereby giving a positive answer to the parameter dependence problem for solutions of semi-elliptic equations, including distributional and tempered distributional parameter dependence (see e.g.~\cite{BonetDomanski06} and references therein).

In two appendices, we discuss some results on polynomials and on spherical harmonics expansion in the Fr\'echet space of smooth functions, respectively, which are used to establish our results.

Throughout the paper, we freely use standard notions and results from the theory of distributions, partial differential operators, and locally convex spaces. For background material, we refer to \cite{Hor1,Hor2,Treves1966} and \cite{Jarchow1981,MeV}, respectively.

\section{Smooth zero solutions with partially bounded support}\label{section:zero solutions}

In this section, we prove one of the main results of this paper, Theorem \ref{th:zero solution} below. It states the existence of non-trivial smooth functions with partially bounded supports which belong to the kernel of certain partial differential operators.

Let $P\in\C[\xi_1,\ldots,\xi_d]$ be a non-constant polynomial of degree $m$, i.e.~$P(\xi)=\sum_{|\alpha|\leq m} c_\alpha \xi^\alpha$ with $c_\alpha\neq 0$ for some $\alpha$ with $|\alpha|=m$. Here and in what follows, for a multi-index $\alpha\in\N_0^d$ we denote its length by $|\alpha|(=\sum_{j=1}^d\alpha_j$). Although we also denote the euclidean norm of a vector $x$ in $\R^d$ by $|x|$, this should not cause any confusion, since the meaning will be apparent from the context. As usual, for $P\in\C[\xi_1,\ldots,\xi_d]$ of degree $m$ we denote its principal part by $P_m$, i.e.~$P_m(\xi)=\sum_{|\alpha|= m} c_\alpha \xi^\alpha$. The partial differential operator with constant coefficients $P(D)$ corresponding to $P$ is defined by the expression $\sum_{|\alpha|\leq m}c_\alpha D^\alpha=\sum_{|\alpha|\leq m}c_\alpha (-i)^{|\alpha|}\partial^\alpha$.

Additionally, for $X\subset\R^d$ open, we define
\[\mathscr{E}_P(X)=\{f\in\mathscr{E}(X) \colon P(D)f=0\text{ in }X\},\]
where we use Schwartz' notation \cite{Schwartz1966} and denote by $\mathscr{E}(X)$ the space of complex valued, smooth functions equipped with its natural Fr\'echet space topology.

The main ingredient in the proof of Theorem \ref{th:zero solution} is a result on the Cauchy problem, Theorem \ref{rem:cohoon-solution} below, which is of independent interest. In order to formulate and prove it, we will work with suitable (inhomogeneous) Gevrey classes whose definition we recall for the reader's convenience. 

\begin{definition}\label{def:Gevrey}
	Let $s=(s_1,\ldots,s_n)\in [1,\infty)^n$. Let $K\subset\R^n$ be a compact set such that $\overline{\operatorname{int}K}=K$ and let $C>0$. We write $G^{s,C}(K)$ for the Banach space consisting of those $\phi\in C^\infty(K)$ such that
	\[\|\phi\|_{G^{s,C}(K)}:=\sup_{\alpha\in\N_0^n}\sup_{x\in K}\frac{|D^\alpha\phi(x)|}{C^{|\alpha|} \prod_{j=1}^n (\alpha_j!)^{s_j}}<\infty,\]
where 
$$C^\infty(K):=\{f\in \mathscr{E}(\operatorname{int}K)\colon \partial^\alpha f\text{ admits a continuous extension to }K\text{ for all }\alpha\in\N_0^d\}.$$
	For $X\subset\R^n$ open we define the \emph{Gevrey space of class $s$ in $X$} as
	\[G^s(X):=\varprojlim_{K \subset X\mbox{\tiny compact,} \overline{\operatorname{int}K}=K} \varinjlim_{C\rightarrow \infty}G^{s,C}(K)\]
	which we equip with its natural locally convex topology. In case of $s\in(1,\infty)^n$, let us also set $G^s_c(X):=G^s(X)\cap\mathscr{D}(X)$.
\end{definition}

Obviously, for $s,t\in [1,\infty)^n$ with $s_j\leq t_j$ we have $G^{s}(X)\subset G^{t}(X)$ continuously.   For $s=(\sigma,\ldots,\sigma), \sigma\in[1,\infty)$, the space $G^s(X)$ is denoted by $G^\sigma(X)$ and called the (homogeneous) Gevrey space of class $\sigma$. We have $G^{\max_j s_j}(X)\subset G^s(X)$ continuously. Apparently, for $\sigma=1$ the Gevrey class $G^\sigma(X)$ is the space of real analytic functions $\mathscr{A}(X)$ on $X$. By the elementary inequalities $\alpha!\leq |\alpha|!\leq n^{|\alpha|}\alpha!$, $\alpha\in\N_0^n$, it holds
$$G^\sigma(X):=\varprojlim_{K \subset X\mbox{\tiny compact,} \overline{\operatorname{int}K}=K} \varinjlim_{C\rightarrow \infty}\tilde{G}^{\sigma,C}(K),$$
where by $\tilde{G}^{\sigma,C}(K)$ we denote the Banach space consisting of $\phi\in C^\infty(K)$ with
\[\|\phi\|_{\tilde{G}^{\sigma,C}(K)}:=\sup_{\alpha\in\N_0^n}\sup_{x\in K}\frac{|D^\alpha\phi(x)|}{C^{|\alpha|} (|\alpha|!)^\sigma}<\infty.\] 
We refer to \cite{Rodino1993} for any further properties of Gevrey classes we may use.

The proof of our first result is an adaptation of the proof of \cite[Theorem 5]{Kalmes2021}. For $s=1$, it is, of course, a special case of the Cauchy-Kovalevskaya Theorem (see also \cite[Theorem 12.7.5]{Hor2}).

\begin{theorem}\label{cohoon-solution}
	Let $m\in\N$ and let $P\in\C[\xi_1,\ldots,\xi_d]$ be a polynomial of the form $P(\xi',\xi_d)=\xi_d^m+\sum_{k=0}^{m-1}Q_k(\xi') \xi_d^k$, where $Q_k\in\C[\xi_1,\ldots,\xi_{d-1}]$ with $\operatorname{deg}(Q_k)\leq m-1-k$, $k=0,\ldots, m-1$. Moreover, let $Y\subset\R^{d-1}$ be open and let $s\in[1,\infty)^{d-1}$ be such that $\max_k s_k<\frac{m}{m-1}$. For every $h_0, \ldots, h_{m-1}\in G^s(Y)$ there is $u\in G^{(s_1,\ldots,s_{d-1},1)}(Y\times\R)$ such that the following properties hold:
	\begin{enumerate}[{\upshape(i)}]
		\item $u\in\mathscr{E}_P(Y\times\R)$, 
		\item $\partial^{j}_{d}u(x',0)=h_j(x')$ for all $x'\in Y$, $0\leq j\leq m-1$,
		\item $\supp u\subset\left(\cup_{j=0}^{m-1}\supp h_j\right)\times\R$.
	\end{enumerate}
	Especially, for every $x'\in Y$ the function $u(x',\cdot)$ is real analytic on $\R$, and $u(\cdot, x_d)\in G^s(Y)$ for every $x_d\in\R$.
	Consequently , whenever $\sigma>1$, for all $\epsilon>0$ there is $u\in G^{(\sigma,\ldots,\sigma,1)}(\R^{d-1}\times \R)$ satisfying $P(D)u=0$ such that
	\[ \supp u= \overline{B_{d-1}(0,\epsilon)}\times\R\] 
    and $u$ is real analytic on $B_{d-1}(0,\epsilon)\times\R$, where for $n\in\N$ we denote the euclidean ball in $\R^n$ of radius $\rho$ centered at the origin by $B_n(0,\rho)$.
\end{theorem}

\begin{proof}
    For $l\in \N_0$, we define recursively
	\[\mathscr{C}_l:\mathscr{E}(Y)\rightarrow \mathscr{E}(Y), f\mapsto\begin{cases}
		0, &l\in\{0,\ldots,m-2\},\\
		f, &l=m-1,\\
		-\sum_{k=0}^{m-1}Q_k(D)\mathscr{C}_{k+l-m}(f), &l\geq m.
	\end{cases}\]
	Additionally, for $n\in\N$ we define the linear and continuous mappings
	\[L_n:\mathscr{E}(Y)\rightarrow \mathscr{E}(Y\times\R), L_n(h)(x',x_d):=\sum_{l=0}^n\mathscr{C}_{l}(h)(x')\frac{(i\,x_d)^l}{l!}.\]
    Obviously, $\supp L_n(h)\subset\supp h\times\R$.

    Let us assume, for the moment, that $h_0,\ldots,h_{m-1}\in \mathscr{E}(Y)$ are such that for all $0\leq j\leq m-1$ the sequence $(L_n(h_j))_n$ converges in $\mathscr{E}(Y\times\R)$, $L(h_j):=\lim_{n\rightarrow\infty}L_n(h_j)$. Then
    \[u:=\sum_{j=0}^{m-1}\sum_{k=0}^{m-1-j}Q_{j+k+1}(D)D_d^kL(h_j)\]
    belongs to $\mathscr{E}_P(Y\times \R)$ and satisfies $D_d^j u(\cdot,0)=h_j(\cdot), 0\leq s\leq m-1$ (cf.~\cite[Section 4]{Kalmes2021}).
    
    Therefore, we shall prove that $(L_n(h_j))_n$ converges in $\mathscr{E}(Y\times\R)$. To do so, we use the fact that, by \cite[Remark 1]{Kalmes2021},
	\[L_n(h)(x',x_d)=\sum_{l=m-1}^n\sum\limits_{\substack{\beta\in\N_0^m,\\ \sigma(\beta)=l-m+1}}(-1)^{|\beta|}\binom{|\beta|}{\beta_1,\ldots,\beta_m}\prod_{k=1}^m Q_{m-k}^{\beta_k}(D)h(x') \frac{(ix_d)^l}{l!},\]
	where $\sigma(\beta)=\sum_{j=1}^m j\beta_j$.

	By hypothesis, for $0\leq k\leq m-1$ it holds 
	$$Q_k(\xi')=\sum_{\gamma\in\N_0^{d-1}, |\gamma|\leq m-1-k}q_{k,\gamma}\xi'^\gamma$$
	for suitable $q_{k,\gamma}\in\C$. We fix $q\geq 1$ such that $\sum_{\gamma\in\N_0^{d-1}, |\gamma|\leq m-1-k}|q_{k,\gamma}|\leq q$ for every $0\leq k\leq m-1$.
	
	Fix $h_0,\ldots,h_{m-1}\in G^s(Y)$. Moreover, let $K$ be an arbitrary compact subset of $Y$ such that $K=\overline{\operatorname{int}(K)}$ and let $C>1$ be such that
	$$\forall\, 0\leq j\leq m-1:\,\|h_j\|_{G^{s,C}(K)}<\infty.$$
	Additionally, let $R\geq 1$. For $\alpha=(\alpha',\alpha_d)\in\N_0^{d-1}\times\N_0$ it follows for $n\geq \alpha_d+m$, $r\in\N$ and each $(x',x_d)\in K\times [-R,R]$, $0\leq j\leq m-1$,
	\begin{eqnarray}\label{eq:convergence 1}
		&&|D^\alpha L_{n+r}(h_j)(x',x_d)-D^\alpha L_n(h_j)(x',x_d)|\nonumber\\
		&\leq& \sum_{l=\max\{n+1,\alpha_d\}}^{\max\{n+r+1,\alpha_d\}-1}\sum\limits_{\substack{\beta\in\N_0^m,\\ \sigma(\beta)=l-m+1}}\binom{|\beta|}{\beta_1,\ldots,\beta_m}\left|\prod_{k=1}^m Q_{m-k}^{\beta_k}(D)D^{\alpha'}h_j(x')\right| \times\nonumber\\
		&&\times\frac{R^{l-\alpha_d}}{l!}\binom{l}{\alpha_d}\alpha_d!\nonumber\\
		&\leq&\sum_{l=\max\{n+1,\alpha_d\}}^{\max\{n+r+1,\alpha_d\}-1}\sum\limits_{\substack{\beta\in\N_0^m,\\ \sigma(\beta)=l-m+1}}\binom{|\beta|}{\beta_1,\ldots,\beta_m}\left|\prod_{k=1}^m Q_{m-k}^{\beta_k}(D)D^{\alpha'}h_j(x')\right|\times\nonumber\\
		&&\times \frac{(2R)^l}{l!}\alpha_d!
	\end{eqnarray}
	Because
	\begin{eqnarray*}
		\operatorname{deg}\left(\prod_{k=1}^m Q_{m-k}^{\beta_k}\right)&=&\sum_{k=1}^{m}\beta_k\operatorname{deg}(Q_{m-k})\leq \sum_{k=1}^{m}\beta_k(k-1)=\sum_{k=1}^{m}\beta_k k\frac{k-1}{k}\\
		&\leq&\sigma(\beta)\frac{m-1}{m},
	\end{eqnarray*}
	it holds for $x'\in K$ and $\beta\in\N_0^m$, due to Stirling's formula,
	\begin{eqnarray*}
		&&\left|\prod_{k=1}^m Q_{m-k}^{\beta_k}(D)D^{\alpha'}h_j(x')\right| \leq q^{|\beta|}\max_{\substack{\gamma\in\N_0^{d-1},\\ |\gamma|\leq \frac{m-1}{m}\sigma(\beta)}}\max_{x'\in K}\left|D^{\gamma+\alpha'}h_j(x')\right|\\
		&\leq&\|h_j\|_{G^{s,C}(K)}  q^{\sigma(\beta)} \max_{\substack{\gamma\in\N_0^{d-1},\\ |\gamma|\leq \frac{m-1}{m}\sigma(\beta)}}C^{|\gamma|+|\alpha'|}\prod_{k=1}^{d-1}\left((\gamma_k+\alpha'_k)!\right)^{s_k}\\
		&\leq& \|h_j\|_{G^{s,C}(K)}  q^{\sigma(\beta)} (2^{\max_k s_k} C)^{\frac{m-1}{m}\sigma(\beta)}(2^{\max_k s_k} C)^{|\alpha'|}\prod_{k=1}^{d-1}\left(\alpha_k'!\right)^{s_k}\times\\
		&&\times 2^{\frac{m-1}{m}\sigma(\beta)\max_k s_k}\max_{\substack{\gamma\in\N_0^{d-1},\\ |\gamma|\leq \frac{m-1}{m}\sigma(\beta)}}\prod_{k=1}^{d-1}(\gamma_k!)^{s_k}\\
		&\leq& \|h_j\|_{G^{s,C}(K)}  q^{\sigma(\beta)} (2^{\max_k s_k} C)^{\frac{m-1}{m}\sigma(\beta)}(2^{\max_k s_k} C)^{|\alpha'|}\prod_{k=1}^{d-1}\left(\alpha_k'!\right)^{s_k}\times\\
		&&\times 2^{\frac{m-1}{m}\sigma(\beta)\max_k s_k}\,3^{s_1+\cdots+s_{d-1}}\max_{\substack{\gamma\in\N_0^{d-1},\\ |\gamma|\leq \frac{m-1}{m}\sigma(\beta)}}\prod_{k=1}^{d-1}\gamma_k^{s_k\gamma_k}\\
		&\leq& \|h_j\|_{G^{s,C}(K)}  q^{\sigma(\beta)} (2^{\max_k s_k} C)^{\frac{m-1}{m}\sigma(\beta)}(2^{\max_k s_k} C)^{|\alpha'|}\prod_{k=1}^{d-1}\left(\alpha_k'!\right)^{s_k}\times\\
		&&\times 2^{\frac{m-1}{m}\sigma(\beta)\max_k s_k}\,3^{s_1+\cdots+s_{d-1}}\left(\frac{m-1}{m}\sigma(\beta)\right)^{\frac{m-1}{m}\sigma(\beta)\max_k s_k}.
	\end{eqnarray*}
	Combining the previous inequality with inequality \eqref{eq:convergence 1}, we obtain
	\begin{eqnarray}\label{eq:convergence 2}
		&&|D^\alpha L_{n+r}(h_j)(x',x_d)-D^\alpha L_n(h_j)(x',x_d)|\nonumber\\
		&\leq& \|h_j\|_{G^{s,C}(K)} 3^{s_1+\cdots +s_{d-1}}\left(2^{\max_k s_k}C\right)^{|\alpha'|}\prod_{k=1}^{d-1}(\alpha'_k!)^{s_k}\cdot\alpha_d!\times\nonumber\\
		&&\times\sum_{l=n+1}^\infty \frac{(2R)^l}{l!}\sum\limits_{\substack{\beta\in\N_0^m,\\ \sigma(\beta)=l-m+1}}\binom{|\beta|}{\beta_1,\ldots,\beta_m} \nonumber \left(4^{\max_k s_k} C\right)^{\frac{m-1}{m}\sigma(\beta)}\\
		&&\times q^{\sigma(\beta)}\left(\frac{m-1}{m}\sigma(\beta)\right)^{\frac{m-1}{m}\sigma(\beta)\max_k s_k}\nonumber\\ 
		&\leq& \|h_j\|_{G^{s,C}(K)} 3^{s_1+\cdots +s_{d-1}}\left(2^{\max_k s_k}C\right)^{|\alpha'|}\prod_{k=1}^{d-1}(\alpha'_k!)^{s_k}\cdot\alpha_d!\times\nonumber\\
		&&\times\sum_{l=n+1}^\infty \frac{(2R)^l}{l!}\sum\limits_{\substack{\beta\in\N_0^m,\\ \sigma(\beta)=l-m+1}}\binom{|\beta|}{\beta_1,\ldots,\beta_m} \left(4^{\max_k s_k} C\right)^{\frac{m-1}{m}l}\nonumber\\
		&&\times q^l\left(\frac{m-1}{m} l\right)^{l \frac{m-1}{m} \max_k s_k}\nonumber\\ 
		&\leq&  \|h_j\|_{G^{s,C}(K)} 3^{s_1+\cdots +s_{d-1}}\left(2^{\max_k s_k}C\right)^{|\alpha'|}\prod_{k=1}^{d-1}(\alpha'_k!)^{s_k}\cdot\alpha_d!\times\nonumber\\
		&&\times\sum_{l=n+1}^\infty  \left(2R\,4^{\max_k s_k} C q \right)^l\frac{l^{l\frac{m-1}{m}\max_k s_k}}{l!} \times\nonumber\\
		&&\times \sum\limits_{\substack{\beta\in\N_0^m,\\ |\beta|=l-m+1}}\binom{|\beta|}{\beta_1,\ldots,\beta_m} \nonumber\\
		&\leq&  \|h_j\|_{G^{s,C}(K)} 3^{s_1+\cdots +s_{d-1}}\left(2^{\max k s_k}C\right)^{|\alpha'|}\prod_{k=1}^{d-1}(\alpha'_k!)^{s_k}\cdot\alpha_d!\times\nonumber\\
		&&\times\sum_{l=n+1}^\infty  \left(2mR\,4^{\max_k s_k} C q \right)^l\frac{l^{l\frac{m-1}{m}\max_k s_k}}{(l/e)^l}\nonumber\\
		&=& \|h_j\|_{G^{s,C}(K)} 3^{s_1+\cdots +s_{d-1}}\left(2^{\max_k s_k}C\right)^{|\alpha|}\prod_{k=1}^{d-1}(\alpha'_k)^{s_k}\cdot\alpha_d!\times\nonumber\\
		&&\times\sum_{l=n+1}^\infty  \left(\frac{2mR\,4^{\max_k s_k} C q}{e} \right)^l l^{l(\frac{m-1}{m}\max_k s_k-1)}
	\end{eqnarray}
	where we have used the Multinomial Theorem as well as Stirling's formula, again.
	
	By hypothesis, $\frac{m-1}{m}\max_k s_k-1<0$ so that by \eqref{eq:convergence 2}
	\begin{eqnarray*}
		&&\|L_{n+r}(h_j)(\cdot, x_d)-L_n(h_j)(\cdot,x_d)\|_{G^{(s_1,\ldots,s_{d-1},1), 2^{\max_k s_k}C}(K\times[-R,R])}\\
		&\leq &  \|h_j\|_{G^{s,C}(K)}3^{s_1+\cdots+s_{d-1}} \sum_{l=n+1}^\infty\left(\frac{2mR\,4^{\max_k s_k} C q}{e} \right)^l l^{l(\frac{m-1}{m}\max_k s_k-1)}\rightarrow_{n\rightarrow\infty}0.
	\end{eqnarray*}
	Thus, for each $j=1,\ldots,m-1$ the sequences $(L_n(h_j)))_{n\in\N}$ are Cauchy sequences in $G^{(s_1,\ldots,s_{d-1},1),2^{\max_k s_k}C}(K\times[-R,R])$ and therefore convergent.
	
	Herefrom it follows easily that $(L_n(h_j)))_{n\in\N}$ converges in $G^{(s_1,\ldots,s_{d-1},1)}(Y\times\R)$. In particular,
	\[L(h_j)(x',x_d)=\sum_{l=m-1}^\infty \sum\limits_{\substack{\beta\in\N_0^m,\\ \sigma(\beta)=l-m+1}}(-1)^{|\beta|}\binom{|\beta|}{\beta_1,\ldots,\beta_m}\prod_{k=1}^m Q_{m-k}^{\beta_k}(D)h_j(x') \frac{(ix_d)^l}{l!},\]
	converges in $\mathscr{E}(Y\times\R)$ so that
	\[u:=\sum_{j=0}^{m-1}\sum_{k=0}^{m-1-j}Q_{j+k+1}(D)D_d^kL(h_j)\]
	has all the claimed properties.
	
	Finally, for $\sigma>1$ and $\epsilon>0$ let $h_0\in G^\sigma(\R^{d-1})$ be such that
	\[\supp h_0= \overline{B_{d-1}(0,\epsilon)},\quad h_0> 0\text{ in }B_{d-1}(0,\epsilon),\]
	and $h_0$ is real analytic in $B_{d-1}(0,\epsilon)$. Then, the above solution $u$ for $h_0$ and $h_1=\ldots=h_{m-1}=0$ belongs to $G^\sigma(\R^{d-1}\times\R)$, and 
	\[\supp u\subset \overline{B_{d-1}(0,\epsilon)}\times \R.\]
    Moreover, applying the first part of the theorem to the open set $Y=B_{d-1}(0,\epsilon)$, $s=(1,\ldots,1)\in[1,\infty)^{d-1}$ and analytic functions ${h_0}_{\mid{B_{d-1}(0,\epsilon)}}$, $h_1=\ldots=h_{m-1}=0$, we conclude that $u_{\mid{B_{d-1}(0,\epsilon)\times\R}}\in G^1(B_{d-1}(0,\epsilon)\times\R)$, i.e. $u$ is real analytic in $B_{d-1}(0,\epsilon)\times\R$.
	Because $u(x',0)=h_0(x')>0$, the identity principle for real analytic functions entails
	\[B_{d-1}(0,\epsilon)\times\R\subset\supp u\]
    which completes the proof of the theorem.
\end{proof}

\begin{remark}\label{rem:cohoon-solution}
	The proof of Theorem \ref{cohoon-solution} can be adapted to give an analogous result for ultradifferentiable classes of Roumieu type $\mathscr{E}^{\{M\}}(Y)$ instead of Gevrey classes, as long as the defining weight sequence $M=(M_p)_{p\in\N}$ satisfies $$\lim_{p\rightarrow\infty}M_p/(p!)^{m/m-1}=0.$$
	For certain other asymptotic behavior of the weight sequence $M$, see, e.g.,~\cite[Sectiona 4]{DebKalmes2023}.
\end{remark}

For a function $f\colon\R^d\to\C$ and a point $x\in\R^d$, we denote as usual $\tau_x f(y):=f(y-x)$ and $\check f(y):=f(-y)$.
\begin{lemma}\label{lemma:convolution f ast g}
	Let $f,g\in \mathscr{E}(\R^d)$ with $\supp f\subset\prod_{k=1}^d I_k$ and $\supp g\subset\prod_{k=1}^d J_k$ for some open (not necessarily bounded) intervals $I_k,J_k\subset\R$. Let us assume that, for all $1\leq k\leq d$, $I_k$ or $J_k$ is bounded. Then, the following assertions hold.
	\begin{enumerate}[{\upshape(i)}]
		\item The convolution $f\ast g$ is well-defined.
		\item $f\ast g\in \mathscr{E}(\R^d)$ and $D^{\alpha}(f\ast g)=(D^{\alpha}f)\ast g=f\ast (D^{\alpha}g)$ for all $\alpha\in\N_0^d$.
	\end{enumerate}
\end{lemma}

\begin{proof}
	In order to prove (i), let us note that
	\begin{align*}
		\supp(f\cdot\tau_x\check g)&=\supp f\cap\supp \tau_x\check g
		\subset\prod_{k=1}^d \overline{I_k}\cap\left(\prod_{k=1}^d \overline{J_k}+x\right)\\
		&=\prod_{k=1}^d \left[\overline{I_k}\cap (\overline{J_k}+x_k)\right].
	\end{align*}
	By hypothesis imposed on the intervals $I_k, J_k$, the set $\supp(f\cdot\tau_x\check g)$ is compact. Therefore, the convolution $f\ast g$,
	\[(f\ast g)(x)
	:=\int_{\R^d}f(y)g(x-y)\mathrm{d}y
	=\int_{\supp(f\cdot\tau_x\check g)}(f\cdot \tau_x\check g)(y)\mathrm{d}y,\]
	is well-defined.
	
	Statement (ii) can be proved in a similar way as \cite[Theorem 6.30(b)]{Rud-FA}.  
\end{proof}

A first version of the existence result of non-trivial smooth functions with partially bounded supports for kernel of a certain class of partial differential operators is the next one. This is an extension of Cohoon's result \cite[Theorem 1]{Cohoon1970}, where the case $l=d-1$ is treated. For this special case, the form of the principal part of $P$ is also necessary for the existence of a non-trivial zero solution with such a partially bounded support (see also Remark \ref{rem:partially bounded support}). 

\begin{theorem}\label{th:zero solution standard form}
	Let $m\in\N, 1\leq l\leq d-1$, and let $P\in\C[\xi_1,\ldots,\xi_d]$ be a polynomial of degree $m$ with principal part $P_m$ of the form 
    \[P_m(\xi)=\xi_d^m+\sum_{j=l+1}^{d-1}p_j(\xi)\xi_j\]
    for some polynomials $p_j\in\C[\xi_1,\ldots,\xi_d]$.
    Then, for every $0<\delta<\epsilon$ there is $v\in\mathscr{E}_P(\R^d)$ such that
	\begin {enumerate}[{\upshape(i)}]
	\item $\overline{B_{l}(0,\delta)}\times\R^{d-l}\subset\operatorname{supp} v\subset \overline{B_{l}(0,\epsilon)}\times\R^{d-l}$,
	\item $v$ is real analytic on $B_{l}(0,\delta)\times\R^{d-l}$.
	\end{enumerate} 
    As before, for $n\in\N$ we denote the euclidean ball in $\R^n$ of radius $\rho>0$ centered at the origin by $B_n(0,\rho)$.
\end{theorem}

\begin{proof}
Let $P(\xi)=P_m(\xi)+\sum_{|\alpha|<m}c_\alpha\xi^\alpha$. Set $\epsilon_1=\frac{\epsilon-\delta}{2}$ and $\epsilon_2=\frac{\epsilon+\delta}{2}$. Next, let us fix a nonnegative function $\lambda\in \mathscr{E}(\R^{d-1})$ such that $\lambda$ depends only on the first $l$ coordinates,
\[\supp\lambda=\overline{B_l(0,\epsilon_2)}\times\R^{d-l-1},\] and such that $\lambda$ is even symmetric, and $\lambda$ is real analytic on $B_l(0,\epsilon_2)\times\R^{d-l-1}$. 
	
	Let $1<s<\frac{m}{m-1}$ and let us fix a nonzero and nonnegative function $\phi\in G^s(\R^{d-1})$ with $\supp\phi=\overline{B_{d-1}(0,\epsilon_1)}$ and $\phi>0$ in $B_{d-1}(0,\epsilon_1)$.
	By Theorem \ref{cohoon-solution} applied for $h_0=\phi, h_1=\ldots,h_{m-1}=0$, there is a solution $u\in \mathscr{E}(\R^d)$ of the partial differential equation
	\[\left(D_d^m+\sum_{|\alpha|<m} c_\alpha D^\alpha\right)u=0\]
	such that (see also the proof of the last part of Theorem \ref{cohoon-solution})
	$$\supp{u}=\supp\phi\times\R=\overline{B_{d-1}(0,\epsilon_1)}\times\R\subset[-\epsilon_1,\epsilon_1]^{d-1}\times\R,$$
	$u(x',0)=\phi(x')$ for all $x'\in\R^{d-1}$, and $u(x',\cdot)$ is real analytic on $\R$. In particular, it can be extended to a function $u\colon\R^{d-1}\times\C\to\C$ in such a way that $u(x',\cdot)$ is an entire function.
	
	By a standard application of Lebesgue's Dominated Convergence Theorem to differentiability of parameter intergrals, the function
	\[f\colon\C\to\C, f(z):=\int_{\left[-\epsilon_1,\epsilon_1\right]^{d-1}}u(x',z)\lambda(x')\mathrm{d}x',\]
	is entire and since
	\begin{align*}
		f(0)=\int_{\left[-\epsilon_1,\epsilon_1\right]^{d-1}}u(x',0)\lambda(x')\mathrm{d}x'=\int_{\left[-\epsilon_1,\epsilon_1\right]^{d-1}}\phi(x')\lambda(x')\mathrm{d}x'>0,
	\end{align*}
	$f$ is nonzero. 
	
	Let $\mu\in \mathscr{E}(\R)$ be a nonzero, nonnegative function such that $\supp\mu\subset\left[-\epsilon_2,\epsilon_2\right]$ and $\mu$ is real analytic on $\left(-\epsilon_2,\epsilon_2\right)$. We set $g\colon\R\to\C$, $g(t):=\overline{f(-t)}\mu(t)$. Then $g\in \mathscr{E}(\R)$, $\supp g\subset\left[-\epsilon_2,\epsilon_2\right]$ and $g$ is real analytic on $\left(-\epsilon_2,\epsilon_2\right)$. With these, let us also define
	\[\psi\colon\R^{d-1}\times\R\to\R, \psi(x',x_d):=\lambda(x')g(x_d).\]
	Since $\lambda$ depends only on the first $l$ coordinates, $D_j\psi=0$ for $l+1\leq j\leq d-1$,
	\[\supp{\psi}\subset\overline{B_{l}(0,\epsilon_2)}\times\R^{d-l-1}\times\left[-\epsilon_2,\epsilon_2\right]\] and $\psi$ is real analytic on
    \[B_{l}(0,\epsilon_2)\times\R^{d-l-1}\times\left(-\epsilon_2,\epsilon_2\right).\]
	
	Finally, we set $v:=u\ast\psi$. By Lemma \ref{lemma:convolution f ast g}, $v$ is well-defined and $v\in \mathscr{E}(\R^d)$. Moreover,
    \begin{align*}
		\supp{v}&\subset\supp u+\supp\psi\\
		&\subset\overline{B_{d-1}(0,\epsilon_1)}\times\R
		+\overline{B_{l}(0,\epsilon_2)}\times\R^{d-l-1}\times\left[-\epsilon_2,\epsilon_2\right]\\
		&\subset\overline{B_{l}(0,\epsilon)}\times\R^{d-l},
	\end{align*}
	and
	\begin{align*}
		v(0)&=\int_{\R^d}u(y)\psi(-y)\mathrm{d}y\\
		&=\int_{[-\epsilon_2,\epsilon_2]}\left(\int_{[-\epsilon_1,\epsilon_1]^{d-1}}u(y',y_d)\lambda(y')\mathrm{d}y'\right)g(-y_d)\mathrm{d}y_d\\
		&=\int_{[-\epsilon_2,\epsilon_2]}f(y_d)g(-y_d)\mathrm{d}y_d=\int_{[-\epsilon_2,\epsilon_2]}f(y_d)\overline{f(y_d)}\mu(y_d)\mathrm{d}y_d\\
		&=\int_{[-\epsilon_2,\epsilon_2]}|f(y_d)|^2\mu(y_d)\mathrm{d}y_d>0
	\end{align*}
	as well as
	\begin{align*}
		P(D)v&=\left(D_d^m+\sum_{j=l+1}^{d-1}p_j(D)D_j+\sum_{|\alpha|<m}c_\alpha D^\alpha\right)v\\
		&=\left(D_d^m+\sum_{|\alpha|<m}c_\alpha D^\alpha\right)v
		+\sum_{j=l+1}^{d-1}p_j(D)D_jv\\
		&=\left[\left(D_d^m+\sum_{|\alpha|<m}c_\alpha D^\alpha\right)u\right]\ast\psi+u\ast\sum_{j=l+1}^{d-1}p_j(D)D_j\psi=0.
	\end{align*}
	
	It remains to show that $v$ is real analytic on $B_l(0,\delta)\times\R^{d-l}$ which by $v(0)>0$ and the identity principle for real analytic functions will also entail
	$$\overline{B_{l}(0,\delta)}\times\R^{d-l}\subset \supp v. $$
	Let $K$ be a compact subset of $B_l(0,\delta)\times\R^{d-l}$. Denoting by $\pi_d\colon \R^d\rightarrow\R$ the projection onto the $d$-th axis, let $R>0$ be such that $\pi_d(K)\subset(-R,R)$. Because
	\[\supp u= \overline{B_{d-1}(0,\epsilon_1)}\times\R\]
	and
	\[\supp\psi(x-\cdot)\subset x-\overline{B_l(0,\epsilon_2)}\times\R^{d-l-1}\times\left[-\epsilon_2,\epsilon_2\right],\]
	for $x\in K$ and $\alpha,\beta\in\N_0^d$ it holds
    \begin{align*}
        \supp \partial^\alpha u\cap\supp\partial^\beta\psi(x-\cdot)& \subset\overline{B_{d-1}(0,\epsilon_2)}\times [-R-\epsilon_2,R+\epsilon_2]\\
        &\subset [-\epsilon_1,\epsilon_1]^{d-1}\times [-R-\epsilon_2,R+\epsilon_2].
    \end{align*}
	Hence, for $x\in K$ and $\alpha\in\N_0^d$,
	\begin{eqnarray}\label{eq:real analyticity 1}
		|\partial^\alpha v(x)|
		&=&\left|\left(\partial^{(0,\ldots,0,\alpha_d)}u\ast\partial^{(\alpha_1,\ldots,\alpha_{d-1},0)}\psi\right)(x)\right|\nonumber\\
		&\leq&\int_{[-\epsilon_1,\epsilon_1]^{d-1}\times[-R-\epsilon_2,R+\epsilon_2]}|\partial^{(0,\ldots,0,\alpha_d)}u(y)|\cdot\left|\partial^{(\alpha_1,\ldots,\alpha_{d-1},0)}\psi(x-y)\right|\mathrm{d}y\nonumber\\
		&\leq& (2\epsilon_1)^{d-1}2(R+\epsilon_2) \max_{y\in [-\epsilon_1,\epsilon_1]^{d-1}\times[-R-\epsilon_2,R+\epsilon_2]}\left|\partial^{(0,\ldots,0,\alpha_d)}u(y)\right|\times\nonumber\\
		&&\times \max_{y\in [-\epsilon_1,\epsilon_1]^{d-1}\times[-R-\epsilon_2,R+\epsilon_2]}\left|\partial^{(\alpha_1,\ldots,\alpha_{d-1},0)}\psi(x-y)\right|.
	\end{eqnarray}
	By the Cauchy integral formula, for each $y=(y',y_d)\in \overline{B_{d-1}(0,\epsilon_1)}\times[-R-\epsilon_2,R+\epsilon_2]$ we have
	\begin{eqnarray}\label{eq:real analyticity 2}
		\left|\partial^{(0,\ldots,0,\alpha_d)}u(y)\right|
		&=&\frac{1}{2\pi}\alpha_d!\left|\int_{|\xi|=R+\epsilon}\frac{u(y',\xi)}{(\xi-y_d)^{\alpha_d+1}}\mathrm{d}\xi\right|\nonumber\\
		&\leq& \left(\max\left\{\frac{1}{\epsilon_1},1\right\}\right)^{\alpha_d+1}(R+\epsilon)\\
        &\times&\max\left\{|u(y',\xi)|\colon y'\in\left[-\epsilon_1,\epsilon_1\right]^{d-1},|\xi|=R+\epsilon\right\}\alpha_d!.\nonumber
	\end{eqnarray}
	
	Because for $x\in K$ and $y\in \overline{B_{d-1}(0,\epsilon_1)}\times[-R-\epsilon_2,R+\epsilon_2]$ it holds
    \[x-y\in B_{l}(0,\delta+\epsilon_1)\times \R^{d-l}\]
	and since $\lambda$ is real analytic in $B_{l}(0,\epsilon_2)\times\R^{d-l-1}$, there is $C_3>0$ such that for every $x=(x',x_d)\in K$ and $y=(y',y_d)\in \overline{B_{d-1}(0,\epsilon_1)}\times[-R-\epsilon_2,R+\epsilon_2]$ we have
	$$|\partial^{\alpha'}\lambda(x'-y')|\leq C_3^{|\alpha'|}\alpha'!$$
	for each $\alpha'\in\N_0^{d-1}$. Consequently,
	\begin{eqnarray}\label{eq:real analyticity 3}
		&&\max_{x\in K, y\in \overline{B_{d-1}(0,\epsilon_1)}\times[-R-\epsilon_2,R+\epsilon_2]}\left|\partial^{(\alpha_1,\ldots,\alpha_{d-1},0)}\psi(x-y)\right|\nonumber\\
		&=&\max_{x\in K, y\in \overline{B_{d-1}(0,\epsilon_1)}\times[-R-\epsilon_2,R+\epsilon_2]}\left|\partial^{(\alpha_1,\ldots,\alpha_{d-1})}\lambda(x'-y')\right|\left|g(x_d-y_d)\right|\nonumber\\
		&\leq&C_3^{\alpha_1+\cdots\alpha_{d-1}}\alpha_1!\cdots\alpha_{d-1}!\max_{t\in [-\epsilon_2,\epsilon_2]} \left|g(t)\right|.
	\end{eqnarray}
	Combining \eqref{eq:real analyticity 1}, \eqref{eq:real analyticity 2}, and \eqref{eq:real analyticity 3}, we obtain $M, C>0$ such that for every $x\in K$ and $\alpha\in\N_0^d$
	\[|\partial^\alpha v(x)|\leq C^{|\alpha|}\alpha! M\]
	which finally shows that $v$ is real analytic on $B_{l}(0,\delta)\times\R^{d-l}$. 
\end{proof}

Of course, the above theorem can be applied to polynomials $P$ (resp.~to partial differential operators $P(D)$) which, after an orthogonal change of variables, are of the form mentioned in the previous result. In order to characterize these polynomials, we introduce the following notions.

\begin{definition}
    Let $P\in\C[\xi_1,\ldots,\xi_d]$ be a polynomial of degree $m$ with principal part $P_m$.
    \begin{enumerate}
        \item[(i)] As usual, the \emph{characteristic cone} of $P$ is defined as
        \[\charcone(P):=\{\xi\in\R^d \colon P_m(\xi)=0\}.\]
        Clearly, $\charcone(P)=\{0\}$ if and only if $P$ is elliptic, while $\charcone(P)=\R^d$ precisely when $P=0$. We point out a slight abuse of the usual notation, as for a differential operator with smooth coefficients of order $m$ in an open set $X$ of $\R^d$, $P(x,D)=\sum_{|\alpha|\leq m}c_\alpha(x)D^\alpha$, $\charcone (P)$ usually denotes its characteristic set
        \[\{(x,\xi)\in X\times\left(\R^d\backslash\{0\}\right) \colon P_m(x,\xi)=0\},\]
        where $P_m(x,\xi)=\sum_{|\alpha|= m}c_\alpha(x) \xi^\alpha$ (cf.~\cite[Theorem 8.3.1]{Hor1}).

        We say that the characteristic cone of the complex coefficient polynomial $P$ is \emph{linear} if $\charcone(P)$ is a linear subspace of $\R^d$.

        \item[(ii)] Let $Z$ be a linear subspace of $\R^d$ different from $\R^d$ and let $w\in \R^d\setminus Z$. We say that the polynomial $P$ (and the differential operator $P(D)$) is \emph{\wniceon{w}{$Z$}} if $P_m(w)\neq 0$ and $P_m$ together with the directional derivatives $\partial^k_w P_m$, $1\leq k\leq m-1$, vanish on $Z$. $P$ is called \emph{\niceon{$Z$}} if it is $w$-degenerate for some $w\in \R^d\setminus Z$. Finally, $P$ is called \emph{\veryniceon{$Z$}} if it is $w$-degenerate for some $w\in Z^\perp\setminus\{0\}$.   
    \end{enumerate} 
\end{definition}

We postpone the proof of the following lemma to Appendix \ref{appendix: characterization of some polynomials}.

\begin{lemma}\label{lem:nice polynomial2}
	Let $P\in\C[\xi_1,\ldots,\xi_d]$ be a polynomial of degree $m\geq1$ with principal part $P_m$. Moreover, let $Z$ be a subspace of $\R^d$ of dimension $1\leq l\leq d-1$ and let $w\in \R^d\backslash Z$. Then the following statements are equivalent:
    \begin{enumerate}
        \item[\upshape{(i)}] 
	    $P$ is \wniceon{w}{$Z$};
	   \item[\upshape{(ii)}] there is a linear isomorphism $A:\R^d\rightarrow\R^d$,  $c\in\C\backslash\{0\}$, and polynomials $p_{l+1},\ldots,p_{d-1}\in\C[\xi_1,\ldots,\xi_d]$ such that $A^{-1}w=e_d$,
	\[A^{-1}(Z)=\{\xi\in\R^d\colon \xi_{l+1}=\ldots=\xi_d=0\},\]
    and the principal part $(P\circ A)_m$ of $P\circ A$ is given by
	\[(P\circ A)_m(\xi)=c\xi_d^m + \sum_{j=l+1}^{d-1}p_j(\xi)\xi_j.\] 
    \end{enumerate}
    Additionally, for $w\in Z^\perp\setminus\{0\}$, the above equivalence is true with an orthogonal linear map $A$.
\end{lemma}

\begin{remark}\label{rem:nice polynomial}
Let $X\subset\R^d$ be open or closed. Then, as follows from the chain rule, for every $P\in\C[\xi_1,\ldots,\xi_d]$ and each linear isomorphism $A\colon\R^d\rightarrow\R^d$, the mapping
		\[\Phi_A:\mathscr{E}(X)\rightarrow\mathscr{E}(A^T(X)),f\mapsto f\circ (A^T)^{-1}\]
is a linear topological isomorphism for which the identity
		\[\Phi_A\left(P(D)f\right)=(P\circ A)(D)\left(\Phi_A(f)\right), f\in\mathscr{E}(X),\] 
holds. In particular, $f\in\mathscr{E}_P(X)$ if and only $\Phi_A(f)\in\mathscr{E}_{P\circ A}(A^T(X))$. Here, for a closed set $X$, $\mathscr{E}(X)$ denotes the space of smooth Whitney jets and $P(D)\colon\mathscr{E}(X)\to\mathscr{E}(X)$ is a partial differential operator, as explained in the introduction of Section \ref{sec:p-Runge pairs}).
\end{remark}

We have now everything at hand to prove the main theorem of this section which in particular proves the existence of non-trivial $v\in\mathscr{E}_P(\R^d)$ for which its support is bounded in the directions of a non-trivial subspace $Z$, i.e.~there is $C>0$ with $|\langle x,z\rangle|\leq C$ for every $x\in\supp v$ and $z\in Z$, $|z|=1$.

\begin{theorem}\label{th:zero solution}
	Let $P\in\C[\xi_1,\ldots,\xi_d]$ be \veryniceon{a linear subspace $Z\neq\{0\}$ of $\R^d$}. Then, for every $0<\delta<\epsilon$ there is $v\in\mathscr{E}_P(\R^d)$ such that
	\begin {enumerate}[{\upshape(i)}]
	\item $Z^\perp+\overline{B(0,\delta)}\subset\operatorname{supp} v\subset Z^\perp+\overline{B(0,\epsilon)}$,
	\item $v$ is real analytic on $Z^\perp+B(0,\delta)$.
	\end{enumerate} 
\end{theorem}

\begin{proof}
We take $w\in W:=Z^\perp\setminus\{0\}$ such that $P$ is \wniceon{w}{$Z$}. By Lemma \ref{lem:nice polynomial2}, 
there is an orthogonal map $A:\R^d\rightarrow\R^d$, $c\in\C\backslash\{0\}$, and polynomials $p_{l+1},\ldots,p_{d-1}\in\C[\xi_1,\ldots,\xi_d]$ such that $A^{-1}w=e_d$,
	\[A^{-1}(Z)=\{\xi\in\R^d\colon \xi_{l+1}=\ldots=\xi_d=0\}=\R^l\times\{0\}^{d-l},\]
    and the principal part $(P\circ A)_m$ of $P\circ A$ is given by
	\[(P\circ A)_m(\xi)=c\xi_d^m + \sum_{j=l+1}^{d-1}p_j(\xi)\xi_j.\]
Hence, dividing $P\circ A$ by $c$, we may assume without loss of generality that,
\[(P\circ A)_m(\xi)=\xi_d^m + \sum_{j=l+1}^{d-1}p_j(\xi)\xi_j.\]

Let us fix $0<\delta<\epsilon$. By Theorem \ref{th:zero solution standard form} applied to the polynomial $P\circ A$, there is $v_0\in\mathscr{E}_{P\circ A}(\R^d)$ such that 
\[\overline{B_{l}(0,\delta)}\times\R^{d-l}\subset\supp v_0\subset \overline{B_{l}(0,\epsilon)}\times\R^{d-l}\]
and $v_0$ is real analytic on $B_{l}(0,\delta)\times\R^{d-l}$.	
Let 
\[\Phi_A\colon \mathscr{E}(\R^d)\to \mathscr{E}(\R^d),\quad \Phi_Af:=f\circ (A^T)^{-1}.\] Then $\Phi_A^{-1}f=\Phi_{A^T}f:=f\circ A^{-1}=f\circ A^T$ because $A$ is orthogonal. We define $v\in \mathscr{E}(\R^d)$ by
\[v:=v_0\circ A^T\]
so that $\supp v=A\left(\supp v_0\right)$. By Remark \ref{rem:nice polynomial},
\[P(D)v=P(D)(\Phi_{A^T}v_0)=\Phi_{A^T}(P\circ A)(D)v_0)=0.\]
Moreover, since $A$ is orthogonal, $A(\overline{B(0,\rho)})=\overline{B(0,\rho)}$ and 
\begin{align*}
A^T\left(W+\overline{B(0,\rho)}\right)&=
A^T\left(W+A(\overline{B(0,\rho)})\right)
=A^{-1}(W)+\overline{B(0,\rho)}\\
&=\{0\}^l\times\R^{d-l}+\overline{B(0,\rho)}\\
&=\overline{B_l(0,\rho)}\times \R^{d-l}
\end{align*}    
for all $\rho>0$. Consequently,
\begin{align*}
W+\overline{B(0,\delta)}&=A\left(\overline{B_l(0,\delta)}\times \R^{d-l}\right)=\supp v\\
&\subset A\left(\overline{B_l(0,\epsilon)}\times \R^{d-l}\right)=W+\overline{B(0,\epsilon)}.
\end{align*}
Moreover, $v$ is real analytic on $W+B(0,\delta)$ as a composition of two real analytic functions, which completes the proof.
\end{proof}

For elliptic $P$ we have $\charcone(P)=\{0\}$ so that $\charcone(P)^\perp+B(0,\delta)=\R^d$ etc.. Additionally, for $\zeta\in\C^d$ with $P(\zeta)=0$, the real analytic function
\[\R^d\rightarrow\C, x\mapsto\exp\left(i\sum_{j=1}^dx_j\zeta_j\right)\]
belongs to $\mathscr{E}_P(\R^d)$. Hence, for elliptic $P$, the following corollary is trivially true. For other polynomials which are \veryniceon{their linear characteristic cone}, it is an immediate consequence of Theorem \ref{th:zero solution} applied to $Z=\charcone(P)$.
\begin{corollary}\label{cor:zero solution}
Assume that a polynomial $P\in\C[\xi_1,\ldots,\xi_d]$ is \veryniceon{its linear characteristic cone}. Then, for every $0<\delta<\epsilon$ there is $v\in\mathscr{E}_P(\R^d)$ such that
	\begin {enumerate}[{\upshape(i)}]
	\item $\charcone(P)^\perp+\overline{B(0,\delta)}\subset\supp v\subset \charcone(P)^\perp+\overline{B(0,\epsilon)}$,
	\item $v$ is real analytic on $\charcone(P)^\perp+B(0,\delta)$.
	\end{enumerate} 
\end{corollary}

We close this section with some examples of polynomials which are orthogonally degenerate on non-trivial subspaces. 

\begin{example}\label{example}
    (i) Let $P(\xi)=\sum_{j=1}^d a_j\xi_j+a_0$ be a polynomial of degree one, $a_j\in\C$, $0\leq j\leq d$. Then, with $a=(a_1,\ldots,a_d)\in\C^d$,
    \[\charcone(P)=\operatorname{span}\{\operatorname{Re}(a),\operatorname{Im}(a)\}^\perp,\]
    and $P$ is \veryniceon{its linear characteristic cone $\charcone(P)$}. Note that $P$ is elliptic if and only if $d=2$ and $\operatorname{Re}(a), \operatorname{Im}(a)$ are linearly independent over $\R$.
    
    (ii) Recall that a polynomial $P\in\C[\xi_1,\ldots,\xi_d]$ of degree $m\geq1$ is called \textit{semi-elliptic} if there is  ${\bf{n}} \in\N^d$ such that $P$ can be written as 
    $$
    P(\xi)=\sum_{|\alpha: {\bf{n}}|\leq 1}c_\alpha\xi^\alpha
    $$ 
    and 
    $$
    P^0(\xi):=\sum_{|\alpha: {\bf{n}}|=1}c_\alpha\xi^\alpha\neq 0
    $$
    for all $\xi\in\R^d\backslash\{0\}$. Here, for $\alpha\in\N_0^d$ and ${\bf{n}}  \in\N^d$ we write $|\alpha:{\bf{n}}|:=\sum_{j=1}^d\frac{\alpha_j}{n_j}$.

    If $P$ is semi-elliptic, the above multi-index ${\bf{n}}$ is unique. More precisely, $n_j=\operatorname{deg}_{\xi_j}P$, $j= 1, \ldots, d$, where $\operatorname{deg}_{\xi_j}P$ denotes  the degree of $P$ with respect to  $\xi_j$; see \cite[Lemma 7.14]{Treves1966}. Furthermore,  $\deg P = \max\{\operatorname{deg}_{\xi_j}P\colon 1\leq j\leq d\}$ by \cite[Proposition 2]{FrerickKalmes10}. It is well-known that semi-elliptic polynomials are hypoelliptic; see \cite[Theorem 11.1.11]{Hor2} and \cite[Theorem 7.7]{Treves1966}. Note that when $n_j=m$ for every $j$, the semi-elliptic operators are just the elliptic operators of order $m$. If $n_1=1$ and $n_j=2$ for $j>1$, we find that the heat operator is semi-elliptic. Moreover, the $k$-parabolic operators of Petrowsky are semi-elliptic; see \cite[Definition 7.11]{Treves1966}. 

    For a semi-elliptic polynomial $P$ of degree $m$ with principal $P_m$, its characteristic cone $\charcone(P)$ is a subspace of $\R^d$. More precisely,
    \begin{equation}\label{eq:charactersitic semi-elliptic}
        \charcone(P)=\{\xi\in\R^d\colon \xi_j=0\text{ for all }j\text{ with }n_j=m\},
    \end{equation}
    see \cite[Proposition 2 (ii)]{FrerickKalmes10}. Additionally, by the proof of \cite[Proposition 2 (ii)]{FrerickKalmes10},
    \[P_m(\xi)=\sum_{|\alpha|=m}c_{\alpha}\xi^\alpha=\sum_{j\colon n_j=m}c_{me_j}\xi_j^m,\]
    where $e_j$ denotes the $j$th standard basis vector of $\R^d$, i.e.~$e_j=(\delta_{j,k})_{1\leq k\leq d}$ (Kron\-ecker's $\delta$). Hence, for every $j$ with $n_j=m$, $P$ is \wniceon{e_j}{$\charcone(P)$}. In particular, $P$ is \veryniceon{its linear characteristic cone}.
    
    (iii) We consider second order partial differential operators whose principal part is given by the wave operator. As is usual, instead of $x=(x_1,\ldots,x_d)\in\R^d$ we write $(x_0,x_1,\ldots,x_d)\in\R^{d+1}$ and we denote $x_0$ by $t$ and $(x_1,\ldots,x_d)$ by $x$. Moreover, we abbreviate $\Delta_x=\sum_{j=1}^d\frac{\partial^2}{\partial x_j^2}$ as well as $\partial^2_t=\frac{\partial^2}{\partial t^2}(=\frac{\partial^2}{\partial x_0^2})$. Let $P\in\C[\xi_0,\xi_1,\ldots,\xi_d]$ be of degree 2 such that the principal part of $P(D)$ is \[P_2(D)=\Delta_x-\partial_t^2\]
    and
    \[\charcone(P)=\left\{\xi=(\xi_0,\ldots,\xi_d)\in\R^{d+1}\colon \xi_0^2=\sum_{j=1}^d\xi_j^2\right\}.\]
    In case $d\geq 2$, $P$ is orthogonally \wniceon{w}{the one dimensional subspace $\operatorname{span}\{\xi\}$} for all $\xi\in \charcone(P)\backslash\{0\}$ and all non-zero $w=(0,w_1,\ldots,w_d)\in\R^{d+1}$ orthogonal to $\xi$.

    Indeed, let $\xi\in \charcone(P)\backslash\{0\}$ and $w\in\R^{d+1}\backslash\{0\}$ be orthogonal to $\xi$ with $w_0=0$. Then,
    \[0\neq -\sum_{j=1}^d w_j^2=P_2(w)\]
    and since $\nabla P_2(x)=2(x_0,-x_1,\ldots,-x_d)$, for $\lambda\in\R$ we have
    \[\nabla_w P_2(\lambda\xi)=\langle w,\nabla P_2(\lambda\xi)\rangle=-2\lambda\sum_{j=0}^dw_j\xi_j=0.\]
    We point out that the hypothesis $d\geq 2$ is needed to ensure the existence of a non-zero $w=(0,w_1,\ldots,w_d)\in\R^{d+1}$ orthogonal to $\xi\in \charcone(P)\backslash\{0\}$. Given such $\xi$, there is $1\leq k\leq d$ with $\xi_k\neq 0$. If $k<d$, the vector $w$ defined by
    \[
    w_j:=
    \begin{cases}
    -\zeta_{k+1}, \text{ for } j=k\\ 
    \zeta_{k}, \text{ for } j=k+1\\
    0, \text{ otherwise},
    \end{cases}
    \]
    is as desired. If $k=d$, we define $w$ as
    \[
    w_j:=
    \begin{cases}
    -\zeta_{d}, \text{ for } j=d-1\\ 
    \zeta_{d-1}, \text{ for } j=d\\
    0, \text{ otherwise}.
    \end{cases}
    \]
    As a consequence of Theorem \ref{th:zero solution}, we obtain that, whenever $d\geq 2$, for each $0<\delta<\epsilon$ and every $\xi\in \charcone(P)\backslash\{0\}$, there is $v_\xi\in\mathscr{E}_P(\R^{d+1})$ such that $v_\xi$ is real analytic in $\operatorname{span}\{\xi\}^\perp +B(0,\delta)$ and
    \[\operatorname{span}\{\xi\}^\perp +\overline{B(0,\delta)}\subset\supp v_\xi\subset \operatorname{span}\{\xi\}^\perp +\overline{B(0,\epsilon)}.\]
    We point out that the orthogonal complements above are precisely the characteristic hyperplanes for $P$ which contain the origin. For $d=1$, the existence of such zero solutions for $P_2(D)$ can be easily proved via the factorization of $P_2(D)$ into first-order operators.
\end{example}

\begin{remark}\label{rem:partially bounded support}
    By Holmgren's Uniqueness Theorem (cf.~\cite[Theorem 8.6.5]{Hor1}), whenever the support of $u\in\mathscr{E}_P(\R^d)$ is contained in some half space $\{x\in\R^d\colon \langle x,N\rangle\leq C\}$ with non-characteristic boundary, then $u\equiv 0$. Hence, the directions into which the support of a non-trivial $u\in\mathscr{E}_P(\R^d)$ is bounded (in the sense explained before Theorem \ref{th:zero solution}), have to be characteristic ones. Now assume that $u\in\mathscr{E}_P(\R^d)\backslash\{0\}$ satisfies
    \[\supp u\subset\bigcap_{j=1}^l\{x\in\R^d\colon \left|\langle x,N_j\rangle\right|\leq C\},\]
    for some $C>0$ and some linearly independent $N_1,\ldots,N_l\in\charcone(P)$. Then, for every $\lambda_j\in\R$, we also have
    \[\supp u\subset \left\{x\in\R^d\colon \left|\left\langle x,\sum_{j=1}^l \lambda_j N_j\right\rangle\right|\leq \sum_{j=1}^l|\lambda_j| C\right\}.\]
    Since $u\neq 0$, by the above observation, $\sum_{j=1}^l \lambda_j N_j\in\charcone(P)$. Therefore, whenever there are linearly independent, characteristic directions $N_1,\ldots, N_l$ and a non-trivial $u\in\mathscr{E}_P(\R^d)$ with bounded support in each of these characteristic directions, then $\charcone(P)$ contains the linear span of $N_1,\ldots, N_l$. Thus, the seemingly strong hypothesis in Theorem \ref{th:zero solution} that $\charcone(P)$ contains a linear subspace, is necessary for the existence of non-trivial zero solutions with partially bounded supports in several different directions.
    
    Finally, we state explicitly, that a zero solution with partially bounded support in several distinct directions does not exist for the wave operator (compare Example \ref{example} (iii)).
\end{remark}

\section{A characterization of \texorpdfstring{$P$}{P}-convexity for supports for partial differential operators which are \veryniceon{their linear characteristic cone}}\label{sec:p-convexity}

By the seminal work of Malgrange \cite{Malgrange1955}, for a polynomial $P\in\C[\xi_1,\ldots,\xi_d]$ and for an open set $X\subset\R^d$ the constant coefficient differential operator $P(D)$ is surjective on $\mathscr{E}(X)$ if and only if $X$ is $P$-convex for supports, that is, if and only if for every compactly supported distribution in $X$, $\mu\in\mathscr{E}'(X)$, the distances between $X^c$ and $\supp\mu$ on the one hand and between $X^c$ and $\supp\check{P}(D)\mu$ on the other hand coincide, where $\check{P}(\xi)=P(-\xi)$. More general than surjectivity of $P(D)$ on $\mathscr{E}(X)$, Malgrange proved that $P$-convexity of supports of $X$ characterizes surjectivity of $P(D)$ on local subspaces of the space of distributions of finite order over $X$, like e.g.~local Sobolev spaces (see e.g.~\cite[Section 10.6]{Hor2}).

Despite the classical nature of Malgrange's results, there are only very few differential operators $P(D)$ for which a satisfactory geometric evaluation of $P$-convexity for supports is available. Convex open sets are $P$-convex for supports, whenever $P\not\equiv 0$, and every open set is $P$-convex for supports whenever $P$ is elliptic. While there is a complete geometric characterization of $P$-convexity for supports in the two dimensional case due to H\"ormander \cite[Theorem 10.8.3]{Hor2}, in arbitrary dimensions, however, the problem of characterizing open subsets of $\R^d$ which are $P$-convex for supports is far from being solved. Geometric characterizations of $P$-convexity for supports are only available for operators with principal part equal to the wave operator \cite{Persson1981} (see also \cite[Theorem 10.8.6]{Hor2}), for general operators of real principal type under additional regularity conditions on the boundary of $X$ (see \cite{Persson1993, Tintarev1988, Tintarev1992}), as well as for first order operators \cite{Zachmanoglou69}, and, more generally, for the special case of polynomials $P$ acting along a subspace and being elliptic there (see \cite{Nakane1979} and \cite[Theorem 10.8.5]{Hor2}).

A general necessary condition due to H\"ormander for an open subset $X\subset\R^d$ to be $P$-convex for supports is that the boundary distance of $X$ satisfies the minimum principle in every characteristic hyperplane (cf.\ \cite[Theorem 10.8.1]{Hor2}). Recall that a real-valued function $f$ on $X\subset\R^d$ is said to satisfy the minimum principle in a closed subset $F$ of $\R^d$ if for every compact set $K\subset F\cap X$ it holds
\[\min_{x\in K}f(x)=\min_{x\in\partial_F K}f(x),\]
where $\partial_F K$ denotes the boundary of $K$ in $F$. We set the boundary distance $d_X$ of $X$ to be the mapping
\[d_X:X\rightarrow\R,x\mapsto\dist(x,X^c),\]
where the distance is taken with respect to the Euclidean norm $|x|$ in $\R^d$.

\begin{remark}\label{rem:normal form for P-convexity for operators with flat characteristic cone}
	Let $A:\R^d\rightarrow\R^d$ be a linear isomorphism. Then, by Remark \ref{rem:nice polynomial}, it follows easily that an open subset $X$ of $\R^d$ is $P$-convex for supports if and only if $A^T(X)$ is $P\circ A$-convex for supports. Moreover, in case $A$ is orthogonal, $d_X$ satisfies the minimum principle in a closed subset $F$ of $\R^d$ if and only if $d_{A^T(X)}$ satisfies the minimum principle in $A^T(F)$. 
\end{remark}

For polynomials which are \veryniceon{a linear subspace $Z$}, the next theorem gives a necessary condition for $P$-convexity for supports in terms of $d_X$. If, additionally, $P$ is \veryniceon{its linear characteristic cone}, this necessary condition is also sufficient for the $P$-convexity for supports, as has been proved in \cite[Theorem 1]{Kalmes19}.

\begin{theorem}\label{th:necessary P-convexity}
	Let $P\in\C[\xi_1,\ldots,\xi_d]$ be \veryniceon{a linear subspace $Z$ of $\R^d$}. If $X$ is an open subset of $\R^d$ which is $P$-convex for supports, then $d_X$ satisfies the minimum principle in each affine subspace parallel to $Z^\perp$.
\end{theorem}

\begin{proof}
    We set $W:=Z^\perp$. For $Z=\{0\}$ we have $x+W=\R^d$, $x\in \R^d$, and $d_X$ trivially satisfies the minimum principle in $x+W$. For $Z\neq \{0\}$, we set $l:=\operatorname{dim}W^\perp$. Then, $l\in \{1,\ldots,d-1\}$. After a suitable orthogonal change of variables, by Lemma \ref{lem:nice polynomial2} and Remark \ref{rem:normal form for P-convexity for operators with flat characteristic cone}, we can assume without loss of generality that
	\[Z=\{x\in\R^d\colon x_{l+1}=\ldots=x_d=0\}\text{ and }W=\{x\in\R^d\colon x_1=\ldots=x_l=0\}.\]
	We assume that there are $x^0\in X$ such $d_X$ does not satisfy the minimum principle in $x^0+W$. Then, there is a compact subset $K$ of $(x^0+W)\cap X$ with
	\[\epsilon:=\min_{x\in K}d_X(x)<\min_{x\in\partial_{x^0+W}K}d_X(x).\]
	Let us choose $y^0\in K$ and $\xi\in\partial X$ so that $\epsilon=|y^0-\xi|$. Let $y\in W$ be such that $y^0=x^0+y$. By Theorem \ref{th:zero solution} applied to $\check{P}$ instead of $P$, for every $t>0$ there is $v_t\in\mathscr{E}_{\check{P}}(\R^d)$ with
	\[W+\overline{B\left(0,\frac{t}{1+t}\epsilon\right)}\subset \supp v_t\subset W+\overline{B(0,t\epsilon)}.\]
	Then, $u_t(x):=v_t(x-y^0), x\in \R^d, t>0,$ satisfies $u_t\in \mathscr{E}_{\check{P}}(\R^d)$ and, because $y^0+W=x^0+y+W=x^0+W$,
	\[\left(x^0+W\right)+\overline{B\left(0,\frac{t}{1+t}\epsilon\right)}\subset \supp u_t\subset \left(x^0+W\right)+\overline{B(0,t\epsilon)}.\]
	Now, let $\pi_{d-l}\colon\R^d\rightarrow\R^{d-l},(x_1,\ldots,x_d)\mapsto (x_{l+1},\ldots,x_d)$. Then, the restriction of $\pi_{d-l}$ to $x^0+W$ is bijective. For $t>0$, let $\tilde{\psi}_t\in \mathscr{D}(\R^{d-l})$ be such that
	\[\supp\tilde{\psi}_t\subset\pi_{d-l}\left(K+B(0,t\epsilon)\right)\]
	and $\tilde{\psi}_t=1$ in a neighborhood of $\pi_{d-l}(K)$. We define $\psi_t:=\tilde{\psi}_t\circ\pi_{d-l}$ and observe that $\psi_t=1$ on $K$,
	\[\supp\psi_t\subset \R^l\times\left(\pi_{d-l}(K)+[-t\epsilon,t\epsilon]^{d-l}\right)\]
	as well as
	\[\supp D^\beta\psi_t\subset\R^l\times\left(\pi_{d-l}\left(\partial_{x^0+W}K\right)+[-t\epsilon,t\epsilon]^{d-l}\right),\quad \beta\in\N_0^d\backslash\{0\}.\]
	We define $\varphi_t:=\psi_t u_t$ so that $\varphi_t\in\mathscr{D}(\R^d)$, $t>0$, and $\varphi_t\in \mathscr{D}(X)$ for $t\in (0,1)$. Additionally, $y^0\in \supp\varphi_t$, $t>0$, so that
	\[\dist(\supp\varphi_t,X^c)\leq d_X(y^0)=\epsilon, t\in (0,1).\]
	On the other hand, because $u_t\in \mathscr{E}_{\check{P}}(\R^d)$ and 
	\[\check{P}(D)\varphi_t=\psi_t\check{P}(D)u_t+\sum_{\beta\in\N_0^d\backslash\{0\}}\frac{1}{\beta!}\left(D^\beta\psi_t\right)(\partial^\beta\check{P})(D)u_t,\]
	it holds
	\begin{eqnarray*}
		\supp\check{P}(D)\varphi_t&\subset &\supp u_t\cap \left(\R^l\times\left(\pi_{d-l}\left(\partial_{x^0+W}K\right)+[-t\epsilon,t\epsilon]^{d-l}\right)\right)\\
		&\subset& \left(\left(x^0+W\right)+\overline{B(0,t\epsilon)}\right)\cap\left(\R^l\times\left(\pi_{d-l}\left(\partial_{x^0+W}K\right)+[-t\epsilon,t\epsilon]^{d-l}\right)\right)\\
		&\subset&\prod_{j=1}^l \left[x^0_j-t\epsilon,x^0_j+t\epsilon\right]\times\left(\pi_{d-l}\left(\partial_{x^0+W}K\right)+[-t\epsilon,t\epsilon]^{d-l}\right)\\
		&=&\partial_{x^0+W}K+[-t\epsilon,t\epsilon]^d.
	\end{eqnarray*}
	Thus, for sufficiently small $t$, $\dist(\supp\check{P}(D)\varphi_t,X^c)$ comes arbitrarily close to $\dist(\partial_{x^0+W}K,X^c)=\min_{x\in\partial_x^0+W}d_X(x)$. In particular, for sufficiently small $t\in(0,1)$ we have $\varphi_t\in\mathscr{D}(X)$ and
	\[\dist(\supp\check{P}(D)\varphi_t,X^c)>\epsilon\geq \dist(\supp\varphi_t,X^c)\]
	which gives the desired contradiction to the $P$-convexity for supports of $X$ (cf.~\cite[Theorem 10.6.3]{Hor2}).
\end{proof}

Our next objective is to give some geometric conditions for $X$ which are equivalent to $d_X$ satisfying the minimum principle in every affine subspace parallel to a subspace $W$. The following definition is inspired by \cite{Zachmanoglou69}.
\begin{definition}\label{def:critical points}
	Let $W\neq\{0\}$ be a linear subspace of $\R^d$ and let $X$ be an open subset of $\R^d$. We say that a point $x_0\in\partial X$ is a \emph{critical point of $X$ with respect to $W$} if there exists a sequence $(x_n)_{n\in\N}$ in $X$ converging to $x_0$ and bounded, connected neighborhoods $U_n$ of $x_n$ in $(x_n+W)\cap X$ such that $(x_j+W)\cap(x_k+W)=\emptyset$ for $j\neq k$ and the set
		\[\overline{\bigcup_{n\in\N}\partial_{x_n+W} U_n}\]
        is a compact subset of $X$, where $\partial_{x+W}C$ denotes the boundary of the set $C\subset x+W$ in $x+W$.
\end{definition}

The next notion generalizes \cite[Definition 3.9]{MTV1990} where the special case of $W=\operatorname{span}\{N\}^\perp$ for $N\in\R^d\backslash\{0\}$ is considered.

\begin{definition}\label{def: inner support}
	Let $W$ be a non-trivial linear subspace of $\R^d$ and let $X$ be an open subset of $\R^d$. We say that $x_0\in\partial X$ is a \emph{point of inner support with respect to $W$} if there is $N\in\R^d\setminus W$ and a compact neighborhood $U$ of $x_0$ in $x_0+(W\oplus\operatorname{span}\{N\})$  such that
    \begin{eqnarray*}
        U\cap\{x_0+w+tN\colon w\in W, t<0\}&\subset& X \quad\text{and}\\
		\partial_{x_0+(W\oplus\operatorname{span}\{N\})} U\cap\{x_0+w+tN\colon w\in W, t\leq 0\}&\subset& X.
    \end{eqnarray*}
\end{definition}

\begin{proposition}\label{prop:equivalence to minimum principle}
	For a non-trivial linear subspace $W$ of $\R^d$ and an open subset $X$ of $\R^d$ the following properties are equivalent.
	\begin{enumerate}[{\upshape(i)}]
		\item $d_X$ satisfies the minimum principle in every affine subspace parallel to $W$.
		\item $X$ does not admit a critical point with respect to $W$.
		\item $X$ does not admit a point of inner support with respect to $W$.
	\end{enumerate}
\end{proposition}
\begin{proof}
	(i)$\Rightarrow$(ii): Suppose that $X$ admits a critical point with respect to $W$.
	Then there is a sequence $(x_n)_{n\in\N}$ in $X$ converging to $x_0$ and bounded open neighborhoods $U_n$ of $x_n$ in $(x_n+W)\cap X$ such that 
	\[C:=\overline{\bigcup_{n\in\N}\partial_{x_n+W} U_n}\]
	is a compact subset of $X$. Let $\epsilon:=\dist(C,X^c)$ and let us choose $n\in\N$ so that $|x_n-x_0|<\epsilon$. Then for the compact set $K:=\overline U_n$ we have
	\[\min_{x\in K}d_X(x)\leq|x_n-x_0|<\epsilon=\dist(C,X^c)<\min_{x\in\partial_{x_n+W}K}d_X(x),\]
	which means that $d_X$ does not satisfy the minimum principle in $x_n+W$.
    
	(ii)$\Rightarrow$(iii): Let us assume that $x_0\in\partial X$ is a point of inner support with respect to $W$. Then there are $N\in\R^d\setminus W$ and a compact neighborhood $U$ of $x_0$ in $x_0+(W\oplus\operatorname{span}\{N\})$ such that
    \begin{eqnarray*}
        U\cap\{x_0+w+tN\colon w\in W, t<0\}&\subset& X,\\
		\partial U\cap\{x_0+w+tN\colon w\in W, t\leq 0\}&\subset& X.
    \end{eqnarray*}
	It follows easily that $x_0$ is a critical point of $X$ with respect to $W$, since the conditions in Definition \ref{def:critical points} are satisfied by a suitable sequence $(x_n)_{n\in\N}$ in $U\cap\{x_0+w+ tN\colon w\in W, t<0\}$ and the connected components $U_n$ of $U\cap (x_n+W)$ containing $x_n$, $n\in\N$.

    (iii)$\Rightarrow$(i): We argue again by contradiction. Assume that $d_X$ does not satisfy the minimum principle in every affine subspace parallel to $W$. Let $x\in\R^d$ and $K$ be a compact subset of $(x+W)\cap X$ with
    \begin{equation}\label{eq:violation of mp 1}
        \epsilon_1:=\min_{y\in K}d_X(y)<\min_{y\in\partial_{x+W}K}d_X(y)=:\epsilon_2.
    \end{equation}
    Choose $y_0\in K,x_0\in\partial X$ with $\epsilon_1=|x_0-y_0|$. Then, $y_0$ is in the interior, with respect to $x+W$, of $K$, $\operatorname{int}_{x+W}(K)$. Let $V$ be the connected component of $\operatorname{int}_{x+W}(K)$ which contains $y_0$. Since $x+W$ is path connected, the same holds for $V$. Replacing $K$ by $\overline{V}$ if necessary, we may assume without loss of generality that
    \[K=\overline{\operatorname{int}_{x+W}(K)},\quad \operatorname{int}_{x+W}(K)\quad\text{is path connected.}\]
    Note that this will at most increase $\epsilon_2$ in \eqref{eq:violation of mp 1}.

    Because $y_0\in x+W$, we have
    \[x+W=y_0+(x-y_0)+W=y_0+W.\]
    Additionally, it is an easy exercise to show that
    \[x_0-y_0=(x_0-x)+(x-y_0)\in W^\perp.\]
    Setting $N:=x_0-y_0$,
    since
    \[W\times\operatorname{span}\{N\}\rightarrow W\oplus\operatorname{span}\{N\},(w,sN)\mapsto w+sN\]
    is a homeomorphism and because $-y_0+K\subset W$,
    \[U:=K+\{sN\colon s\in [0,2]\}=(x_0-y_0)+K+\{tN\colon t\in[-1,1]\}\]
    is a compact neighborhood of $x_0$ in $x_0+(W\oplus\operatorname{span}\{N\})$. Using $-y_0+K\subset W$ and
    \begin{eqnarray*}
        &&\partial_{x_0+(W\oplus\operatorname{span}(\{N\})}U\\
        &=&x_0+\partial_{W\oplus\operatorname{span}(\{N\})}\left(-y_0+K+\{tN\colon t\in [-1,1]\}\right)\\
        &=&x_0+\left(\left(\partial_W\left(-y_0+K\right)+\{tN\colon t\in[-1,1]\}\right)\cup\left(-y_0+K+\{\sigma N\colon\sigma\in\{-1,1\}\}\right)\right)
    \end{eqnarray*}
    it follows
    \begin{eqnarray*}
        &&\partial_{x_0+(W\oplus\operatorname{span}\{N\})}U\cap\{x_0+w+tN\colon w\in W, t\leq 0\}\\
        &=&x_0+\left(\left(\partial_W\left(-y_0+K\right)+\{tN\colon t\in[-1,0)\}\right)\cup\left(-y_0+K+\{\sigma N\colon\sigma\in\{-1\}\}\right)\right)\\
        &=&\left((x_0-y_0)+\partial_{y_0+W}K+\{tN\colon t\in[-1,0)\}\right)\cup K\\
        &=&\left(\partial_{x+W}K+\{tN\colon t\in [0,1)\}\right)\cup K.
    \end{eqnarray*}
    Thus, by \eqref{eq:violation of mp 1} and $\epsilon_1=|x_0-y_0|=|N|$, we conclude
    \[\partial_{x_0+(W\oplus\operatorname{span}\{N\})}U\cap\{x_0+w+tN\colon w\in W, t\leq 0\}\subset X.\]
    Since
    \begin{eqnarray*}
        U\cap\{x_0+w+tN\colon w\in W, t<0\}&=&(x_0-y_0)+K+\{tN\colon t\in [-1,0)\}\\
        &=&K+\{tN\colon t\in[0,1)\}\subset X,
    \end{eqnarray*}
    $x_0$ is a point of inner support with respect to $W$.
\end{proof}

We also recall the following definition from \cite[Introduction]{Zachmanoglou69}.
\begin{definition}\label{def:W-hull}
Let $W$ be a non-zero linear subspace of $\R^d$, $X$ be an open subset of $\R^d$ and let $K$ be a compact subset of $X$. The \emph{$W$-hull of $K$ in $X$}, denoted by $\widehat{K}(W,X)$, is the set $K\cup \bigcup_{x\in\R^d}C_x$, where $C_x$ denotes the sum of all bounded connected components of $(\R^d\setminus K)\cap (x+W)$ contained in $X$.
\end{definition}

\begin{remark}\label{rem:W-hull bounded}
For sets $W,X,K$ as above, the set $\widehat{K}(W,X)$ is contained in the convex hull of the compact set $K$ so it is a bounded subset of $X$.     
\end{remark}

Our next result is a reformulation of \cite[Theorem 2]{Zachmanoglou69}.
\begin{theorem}\label{theorem: geometric sufficient condition for surjectivity}
	Let $P\in\C[\xi_1,\ldots,\xi_d]$ be such that $\charcone(P)$ is contained in a non-trivial subspace of $\R^d$, and let $X$ be an open subset of $\R^d$.
     If $X$ does not admit a critical point with respect to $\charcone(P)^\perp$, then $\widehat{K}(\charcone(P)^\perp,X)$ is compact for each compact subset $K$ of $X$. Consequently, $X$ is $P$-convex for supports.
\end{theorem}
\begin{proof}
Let us assume that $X$ has no critical point with respect to the subspace $W:=\charcone(P)^\perp$. Let $K$ be a compact subset of $X$. We will show the set $\widehat{K}(W,X)$ is compact.  

The statement trivially holds in the case when $\widehat{K}(W,X)=K$. Let us assume that $K\subsetneq\widehat{K}(W,X)$ and suppose that $\widehat{K}(W,X)$ is not compact. By \cite[Lemma 3]{Zachmanoglou69}, there is a sequence $(y_n)_{n\in\N}$ in $\R^d$ and bounded connected components $C_n$ of $(\R^d\setminus K)\cap(y_n+W)$ such that $\overline{C_n}\subset X$ and the distance from $C_n$ to $\partial X$ tends to zero as $n\to\infty$. Hence, there is
$x_0\in\partial X$ and a sequence $(x_n)_{n\in\N}$ converging to $x_0$ such that $x_n\in\overline{C_n}$ for $n\in\N$. Since arbitrarily chosen sequence $(x_n')_{n\in\N}$ with $x_n'\in B(x_n,1/n)\cap C_n$ converges to $x_0$, without loss of generality, we may assume that  $x_n\in C_n$ for $n\in\N$. 
Moreover, there is a subsequence $(x_{n_j})_{j\in\N}$ such that $(x_{n_j}+W)\cap(x_{n_k}+W)=\emptyset$ for $j\neq k$. Indeed, otherwise $(x_n)_{n\in\N}$ has a subsequence contained in some set $C_x$ (see Definition \ref{def:W-hull}) which has to converge to $x_0$. Since $(K\cap(x+W))\cup C_x$ is a compact subset of $X$, $x_0\in \overline{C_x}\subset X$, a contradiction. Hence, without loss of generality, we may also assume that $(x_j+W)\cap(x_k+W)=\emptyset$ for $j\neq k$. Moreover, $\partial_{x_n+W} C_n\subset K$ for all $n\in\N$, so 
$\overline{\bigcup_{n\in\N}\partial_{x_n+W} C_n}$ is a compact subset of $K\subset X$. Hence, $x_0$ is a critical point of $X$ with respect to $W$, which contradicts our assumption.
This shows that $\widehat{K}(W,X)$ is a compact subset of $X$, and \cite[Theorem 3]{Zachmanoglou69} implies that $X$ is $P$-convex for supports.
\end{proof}

We now summarize the geometric properties characterizing $P$-convexity for supports for a polynomial $P$ which is \veryniceon{its linear characteristic cone}. By Example \ref{example}, this characterization is in particular applicable to arbitrary semi-elliptic operators as well as to first order operators. For the latter operators, see also \cite[Theorem 4]{Zachmanoglou69}.

\begin{corollary}\label{cor:characterization P-convexity}
	Let $P\in\C[\xi_1,\ldots,\xi_d]$ be \veryniceon{its linear characteristic cone $\charcone(P)$} and let $X$ be an open subset of $\R^d$. Then, the following assertions are equivalent.
	\begin{enumerate}[{\upshape(i)}]
		\item $X$ is $P$-convex for supports.
		\item $d_X$ satisfies the minimum principle in any affine subspace parallel to $\charcone(P)^\perp$.
		\item For every compact subset $K$ of $X$ and each
		\[x\in \left\{y\in X\colon d_X(y)<\operatorname{dist}(K,X^c)\right\}\]
		there is a continuous, piecewise differentiable curve $\gamma\colon[0,\infty)\rightarrow X\backslash K$ such that $\gamma(0)=x$, $\gamma'(t)\in \charcone(P)^\perp$ for every $t$ where $\gamma$ is differentiable, and $\liminf_{t\rightarrow\infty}\operatorname{dist}(\gamma(t),\partial_\infty X)=0$, where $\partial_\infty X$ denotes the boundary of $X$ in the one-point compactification $\R^d\cup\{\infty\}$ of $\R^d$.
    \end{enumerate}
    If additionally $\charcone(P)\neq \{0\}$, which is equivalent to $P$ being non-elliptic, the above conditions are also equivalent to the following ones.
    \begin{enumerate}
        \item[{\upshape(iv)}] $X$ does not admit a critical point with respect to $\charcone(P)^\perp$.
        \item[{\upshape(v)}] $X$ does not admit a point of inner support with respect to $\charcone(P)^\perp$.
        \item[{\upshape(vi)}] $\widehat{K}(\charcone(P)^\perp,X)$ is compact for each compact subset $K$ of $X$.
	\end{enumerate}
\end{corollary}

\begin{proof}
    By Theorem \ref{th:necessary P-convexity}, (i) implies (ii). Conversely, that (ii) implies (i) is \cite[Theorem 1]{Kalmes19}. The equivalence of (ii) and (iii) is \cite[Lemma 4]{Kalmes19}. For non-elliptic operators, by Proposition \ref{prop:equivalence to minimum principle}, (ii), (iv) and (v) are equivalent. Finally, by Theorem \ref{theorem: geometric sufficient condition for surjectivity}, (iv) implies (vi), and by \cite[Theorem 3]{Zachmanoglou69}, (vi) implies (i).
\end{proof}

In the remainder of this section, we briefly consider surjectivity of partial differential operators on spaces of vector-valued smooth functions. Let $E$ be a locally convex space and let $X\subset\R^d$ be $P$-convex for supports. We are interested in the question whether $P(D)$ remains surjective on $\mathscr{E}(X;E)$, the space of $E$-valued smooth functions on $X$. If $E$ is a space of functions or distributions, this question is equivalent to the problem of parameter dependence of solutions of differential equations which has a long and rich history, see \cite{BonetDomanski06,BonetDomanski08,Domanski10,Mantlik1990,Mantlik1991,Treves1962a,Treves1962b}. By a result of Grothendieck \cite{Grothendieck1955}, this question has a positive answer whenever $E$ is a Fr\'echet space. However, if $E$ is a $(\operatorname{DF})$-space this is no longer true in general, as was shown by Vogt \cite{Vogt1983}. By the splitting theory for Fr\'echet spaces \cite{Vogt1987}, for $E=s'$, where $s$ is the space of rapidly decreasing sequences, the answer to the surjectivity question  is affirmative if and only if $\mathscr{E}_P(X)$, the kernel of $P(D)$ in $\mathscr{E}(X)$, satisfies condition $(\Omega)$ (cf.~\cite{Vogt1983}). We assume that the reader is familiar with the conditions $(\Omega)$ and $(\operatorname{DN})$ for Fr\'echet spaces, and we refer to \cite{MeV} for more information on these conditions. Moreover, if $\mathscr{E}_P(X)$ satisfies $(\Omega)$, then surjectivity of $P(D)$ on $\mathscr{E}(X)$ implies surjectivity on $\mathscr{E}(X;E)$ for every $E$ isomorphic to a product of Montel $(\operatorname{DF})$-spaces whose strong dual spaces satisfy $(\operatorname{DN})$. In view of these results, it is natural to study when the space $\mathscr{E}_P(X)$ satisfies condition $(\Omega)$. A result of Vogt \cite{Vogt1983} states that for elliptic operators $P(D)$, the space $\mathscr{E}_P(X)$ satisfies condition $(\Omega)$ for every open set $X$; recall that arbitrary open sets $X$ are $P$-convex for supports for each elliptic $P$ (cf.~\cite[Corollary 10.8.2]{Hor2}). Additionally, Petzsche \cite{Petzsche1980} showed that for convex open sets $X$ the space $\mathscr{E}_P(X)$ satisfies $(\Omega)$ for every hypoelliptic operator $P(D)$; recall that convex open sets are $P$-convex for every $P$ (cf.~\cite[Theorem 10.6.2]{Hor2}). On the negative side, in \cite{Kalmes12} the second author constructed for $d\geq 3$ a concrete hypoelliptic polynomial $P$ and a $P$-convex open set $X$ for which $\mathscr{E}_P(X)$ does not have $(\Omega)$. For a recent account on this topic, we refer the reader to \cite{DebKalmes2023b} (and references therein), where it is also proved, among others, that $\mathscr{E}_P(X)$ satisfies $(\Omega)$ whenever $d=2$ and $X$ is $P$-convex for supports.

The next theorem generalizes the above mention result of Vogt for elliptic operators as well as \cite[Theorem 18]{Kalmes19} which covers the special case of semi-elliptic operators for which its linear characteristic cone is one dimensional.

\begin{theorem}\label{theorem:Omega for semi-elliptic}
    Let $X\subset\R^d$ be $P$-convex for supports for a semi-elliptic polynomial $P$. Then $\mathscr{E}_P(X)$ satisfies condition $(\Omega)$.
\end{theorem}

\begin{proof}
    As is well known, if $T:E\rightarrow E$ is a surjective continuous linear operator on a Fr\'echet space $E$ and $F$ is a Fr\'echet space, then
    \[T\otimes \operatorname{id}\colon E\hat{\otimes}_\pi F \rightarrow E\hat{\otimes}_\pi F\]
    is surjective again. By nuclearity of $\mathscr{E}(X)$ it holds $\mathscr{E}(X)\hat{\otimes}_\pi\mathscr{E}(\R)\cong\mathscr{E}(X\times \R)$, the isomorphism being given by the linearization and extension of the bilinear map
    \[\mathscr{E}(X)\times\mathscr{E}(\R)\rightarrow\mathscr{E}(X\times\R),(f,g)\mapsto\left((x,t)\mapsto f(x)g(t)\right)\]
    (cf.~\cite[Section 21.6]{Jarchow1981}). As is easily seen, via this isomorphism, $P(D)\otimes\operatorname{id}$ is conjugate to the so-called augmented differential operator $P^+(D)$ on $\mathscr{E}(X\times\R)$, where the polynomial $P^+\in\C[\xi_1,\ldots,\xi_d,\xi_{d+1}]$ is given by
    \[P^+(\xi_1,\ldots,\xi_d,\xi_{d+1}):=P(\xi_1,\ldots,\xi_d).\]
    Hence, $X\times\R$ is $P^+$-convex for supports.

    Additionally, by Corollary \ref{cor:characterization P-convexity}, $d_X$ satisfies the minimum principle in every affine subspace parallel to $\charcone(P)^\perp$. Combining this with \cite[Corollary 10.7.10]{Hor2}, \cite[Theorem 15]{Kalmes19}, and the $P^+$-convexity for supports of $X\times\R$, $P^+(D)$ is surjective on $\mathscr{D}'(X\times\R)$, too.

    Since semi-elliptic polynomials are hypoelliptic (cf.~\cite[Theorem 11.1.11]{Hor2}), combining \cite[Theorem 52.1]{Treves1967b} with \cite[Propositions 8.3 and 5.3 (b)]{BonetDomanski06} finally yields that $\mathscr{E}_P(X)$ satisfies condition $(\Omega)$.
\end{proof}

\section{\texorpdfstring{$P$}{P}-Runge pairs for operators which are \veryniceon{their linear characteristic cone}}\label{sec:p-Runge pairs}

As a second application of Theorem \ref{th:zero solution}, in this section, we study $P$-Runge pairs for smooth Whitney jets on closed sets as well as for smooth functions and distributions on open sets. For the reader's convenience, we first recall some notions.

For a compact subset $K$ of $\R^d$ we denote by $\mathscr{E}(K)$ the space of smooth Whitney jets on $K$, that is, families $f=(f^{(\alpha)})_{\alpha\in\N_0^d}$ of continuous (real or complex valued) functions $f^{(\alpha)}$ on $K$ for which, with the formal Taylor polynomials of arbitrary order $n$
$$T^n_xf(y)=\sum_{|\alpha|\leq n}\frac{f^{(\alpha)}(x)}{\alpha!}(y-x)^\alpha,$$
the quantities
$$q_{K,n,t}(f):=\sup\left\{\frac{|f^{(\alpha)}(y)-\partial^\alpha T_x^nf(y)|}{|x-y|^{n-|\alpha|}}\colon |\alpha|\leq n, x,y\in K, 0<|x-y|<t\right\}$$
satisfy $\lim_{t\rightarrow 0}q_{K,n,t}(f)=0$ for all $n\in\N_0$. (Apparently, smooth Whitney jets are not usual functions, despite the fact that they are sometimes called Whitney functions in the literature.) $\mathscr{E}(K)$ equipped with the sequence of seminorms $(\|\cdot\|_{n,K})_{n\in\N_0}$ defined by
$$\forall\,f\in\mathscr{E}(K)\colon \|f\|_{n,K}:=\sup\left\{|f^{(\alpha)}(x)|\colon x\in K, |\alpha|\leq n\right\}+\sup\left\{q_{K,n,t}(f)\colon t>0\right\}$$
is a Fr\'echet space. For a closed subset $F$ of $\R^d$ with compact exhaustion $(K_l)_{l\in\N}$ and
$$\rho_{l+1}^l:\mathscr{E}(K_{l+1})\rightarrow\mathscr{E}(K_l), f=(f^{(\alpha)})_{\alpha\in\N_0}\mapsto f_{\mid K_l}:=(f^{(\alpha)}_{\mid K_l})_{\alpha\in\N_0^d}$$
$\mathscr{E}(F)$ is defined as the locally convex projective limit of $(\mathscr{E}(K_l))_{l\in\N}$, i.e.\ 
$$\mathscr{E}(F):=\operatorname{proj}_{\leftarrow l}\mathscr{E}(K_l)=\left\{(f_l)_{l\in\N}\in\prod_{l\in\N}\mathscr{E}(K_l)\colon \rho_{l+1}^l(f_{l+1})=f_l,\quad l\in\N\right\}$$
equipped with the subspace topology of the topological product $\prod_{l\in\N}\mathscr{E}(K_l)$. Then $\mathscr{E}(F)$ is independent of the particular choice of $(K_l)_{l\in\N}$ and $\mathscr{E}(F)$ is a Fr\'echet space.

Taylor's Theorem implies that for $g\in\mathscr{E}(\R^d)$ it holds $(\partial^\alpha g_{\mid F})_{\alpha\in\N_0^d}$ (or, more precisely, $\left(\left(\partial^{\alpha}g_{\mid K_l}\right)_{\alpha\in\N_0^d}\right)_{l\in\N}$) belongs to $\mathscr{E}(F)$. By a result of Whitney \cite{Whitney34} (see also \cite{Malgrange1967}) every $f\in\mathscr{E}(F)$ has such an extension $g$. Therefore, by the Closed Graph Theorem, the restriction map $$\rho_F\colon\mathscr{E}(\R^d)\rightarrow\mathscr{E}(F), g\mapsto (\partial^\alpha g_{\mid F})_{\alpha\in\N_0^d}$$
is continuous. It is well known that the sequence of seminorms $(\|\cdot\|_n)_{n\in\N}$ on $\mathscr{E}(F)$ defined as
$$\forall\,f\in\mathscr{E}(F)\colon\,||f||_n:=\inf\left\{|g|_{n}\colon g\in\mathscr{E}(\R^d)  \text{ and } \partial^\alpha g_{\mid{F}}=f^{(\alpha)}, \alpha\in\N_0^d\right\}$$
also defines the Fr\'echet space topology on $\mathscr{E}(F)$. 

With
\[D_j:\mathscr{E}(F)\rightarrow\mathscr{E}(F),(f^{(\alpha)})_{\alpha\in\N_0^d}\mapsto-i(f^{(\alpha+e_j)})_{\alpha\in\N_0^d}, \quad 1\leq j\leq d,\]
it holds $D_j(\rho_F(f))=\rho_F(D_j f)$, $f\in\mathscr{E}(\R^d)$. More generally, for $P\in\C[\xi_1,\ldots,\xi_d]$ the linear partial differential operator $P(D)$ on $\mathscr{E}(F)$ is defined analogously, which is then a continuous linear operator on $\mathscr{E}(F)$. Its kernel, equipped with the subspace topology is denoted by $\mathscr{E}_P(F)$ and thus a Fr\'echet space.

\begin{definition}\label{def:P-Runge for Whitney}
	For closed subsets $F_1\subset F_2$ of $\R^d$ we define the (continuous) restriction map
	\[r\colon\mathscr{E}(F_2)\rightarrow\mathscr{E}(F_1), f\mapsto \left(f^{(\alpha)}_{\mid F_1}\right)_{\alpha\in\N_0^d}.\]
	  It holds $r\left(\mathscr{E}_P(F_2)\right)\subset \mathscr{E}_P(F_1)$ and  $F_1, F_2$ is called a \emph{$P$-Runge pair for smooth Whitney jets} if  $r\left(\mathscr{E}_P(F_2)\right)$ is dense in $\mathscr{E}_P(F_1)$.
\end{definition}

For an open subset $X$ of $\R^d$ we equip $\mathscr{E}_P(X)$ with the subspace topology of the Fr\'echet space $\mathscr{E}(X)$ so that $\mathscr{E}_P(X)$ is a Fr\'echet space. Moreover, we equip $\mathscr{D}'(X)$ with the strong dual topology, i.e.~with the locally convex topology induced by uniform convergence on bounded subsets of $\mathscr{D}(X)$. Then $P(D)$ is continuous on $\mathscr{D}'(X)$ so that its kernel $\mathscr{D}'_P(X):=\{u\in\mathscr{D}'(X)\colon P(D)u=0\}$ is a closed subspace which we endow with the subspace topology.

\begin{definition}
	For open subsets $X_1\subset X_2$ of $\R^d$ and the restriction maps
	\[r_\mathscr{E}\colon\mathscr{E}(X_2)\rightarrow\mathscr{E}(X_1), f\mapsto f_{\mid X_1}\]
	and 
	\[r_{\mathscr{D}'}\colon\mathscr{D}'(X_2)\rightarrow\mathscr{D}'(X_1), u\mapsto u_{\mid X_1}\]
	it holds $r_\mathscr{E}\left(\mathscr{E}_P(X_2)\right)\subset\mathscr{E}(X_1)$ and $r_{\mathscr{D}'}\left(\mathscr{D}'_P(X_2)\right)\subset\mathscr{D}'(X_1)$, respectively. $X_1, X_2$ is called a \emph{$P$-Runge pair}, respectively, a \emph{$P$-Runge pair for distributions}, if $r_\mathscr{E}\left(\mathscr{E}_P(X_2)\right)$ is dense in $\mathscr{E}(X_1)$, respectively, $r_{\mathscr{D}'}\left(\mathscr{D}'_P(X_2)\right)$ is dense in $\mathscr{D}'(X_1)$.
\end{definition}

\begin{remark}\label{rem:normal form for Runge pairs for operators with flat charactersistic cone}
	Let $A:\R^d\rightarrow\R^d$ be a linear isomorphism. Then, by Remark \ref{rem:nice polynomial}, it follows easily for a pair of open subsets $X_1\subset X_2$ of $\R^d$, respectively, for a pair of closed subsets $F_1\subset F_2$ of $\R^d$ that they are a $P$-Runge pair for a constant coefficient partial differential operator $P(D)$ if and only if $A^T(X_1), A^T(X_2)$, respectively $A^T(F_1), A^T(F_2)$, are $(P\circ A)$-Runge pairs.  
\end{remark}

Using Theorem \ref{th:zero solution} instead of \cite[Theorem 5]{Kalmes2021}, the proof of the next result uses similar reasoning as the proof of \cite[Theorem 19]{CiasKalmes2025}. We include it here for the reader's convenience.

\begin{theorem}\label{thm: necessary condition for P-Runge for partially elliptic}
	Let $P\in\C[\xi_1,\ldots,\xi_d]$ be \veryniceon{a subspace $Z\neq\{0\}$ of $\R^d$}. Moreover, let $F_1\subset F_2$ be closed subsets of $\R^d$.	If $F_1,F_2$ is a $P$-Runge pair for smooth Whitney jets then there are no $x\in\R^d$ and $\epsilon>0$ for which $F_2$ contains a bounded connected component of $(\R^d\setminus F_1)\cap (B(x,\epsilon)+Z^\perp)$. 
\end{theorem}

\begin{proof}
	Because $F_1, F_2$ is a $P$-Runge pair, by \cite[Theorem 5]{CiasKalmes2025} 
	\begin{equation}\label{eq:P-Runge by supports}
		\forall\,\mu\in\mathscr{E}'(\R^d)\colon\left(\supp\mu\subset F_2, \supp\check{P}(D)\mu\subset F_1\Rightarrow \supp\mu\subset F_1\right).
	\end{equation}
	As in the proof of Theorem \ref{th:necessary P-convexity}, we abbreviate $W:=Z^\perp$ and we set $l:=\operatorname{dim}Z$ so that $l\in \{1,\ldots,d-1\}$. After a suitable orthogonal change of variables, by Remarks \ref{rem:nice polynomial} and \ref{rem:normal form for Runge pairs for operators with flat charactersistic cone}, we can assume without loss of generality that
	\[Z=\{x\in\R^d\colon x_{l+1}=\ldots=x_d=0\}\text{ and }W=\{x\in\R^d\colon x_1=\ldots=x_l=0\}.\]
	In order to prove the claim, we argue by contradiction. Thus, we assume that there are $x_0\in\R^d$ and $\epsilon>0$ such that $F_2$ contains a bounded connected component $\tilde{C}$ of $(\R^d\setminus F_1)\cap (B(x_0,\epsilon)+W)$. Then, there is $y\in \tilde{C}$ and $\delta>0$ such that the connected component $C$ of $(\R^d\backslash F_1)\cap (B(y,\delta)+W)$ which contains $y$ satisfies $C\subset\tilde{C}$. In particular, $C$ is contained in $F_2$. By Theorem \ref{th:zero solution} there is $v\in\mathscr{E}_{\check{P}}(\R^d)$ such that
	\[W+\overline{B(0,\delta/2)}\subset\supp v\subset W+\overline{B(0,\delta)}\]
	and $v$ is real analytic in $W+B(0,\delta/2)$. Then, $u(x):=v(x-y)$, $x\in\R^d$, satisfies
	\[W+\overline{B(y,\delta/2)}\subset\supp u\subset W+\overline{B(y,\delta)}\]
	and $u\in\mathscr{E}_{\check{P}}(\R^d)$ is real analytic in $W+B(y,\delta/2)$. 
	We define $w\in\mathscr{D}'(\R^d)$ by
	\[w(\varphi):=\int_C u(x)\varphi(x)\mathrm{d}x\]
    for all $\varphi\in\mathscr{D}(\R^d)$.
	Since $F_2$ is closed, we have $\supp w\subset \overline{C}\subset F_2$, so that $w\in\mathscr{E}'(\R^d)$ and
	\[\emptyset\neq C\cap \left(W+\overline{B(y,\delta/2)}\right)\subset \supp w\]
	which implies $\supp w\not\subset F_1$.
	
	Next, we show $\supp\check{P}(D)w\subset F_1$ which by \eqref{eq:P-Runge by supports} will give the desired contradiction. Let $\varphi\in\mathscr{D}(C)$ be arbitrary. Then, taking into account
	\[W+B(y,\delta)=B_l((y_1,\ldots,y_l),\delta)\times\R^{d-l},\]
	integrating by parts yields
	\begin{eqnarray*}
		\left(\check{P}(D)w\right)(\varphi)&=&w\left(P(D)\varphi\right)=\int_Cu(x)P(D)\varphi(x)\mathrm{d}x\\
		&=&\int_{W+B(y,\delta)}\check{P}(D)u(x)\varphi(x)\mathrm{d}x=0,
	\end{eqnarray*}
	so that $\check{P}(D)w$ vanishes on the open set $C$. Additionally, for 
	\[a\in\partial C\backslash F_1\subset\overline{C}\cap\left(\partial B(y,\delta)+W\right)\]
	there is $\rho>0$ with $B(a,\rho)\subset \R^d\backslash F_1$. Because $u$ vanishes outside
	\[B_l((y_1,\ldots,y_l),\delta)\times\R^{d-l}= B(y,\delta)+W,\]
	for $\varphi\in\mathscr{D}(B(a,\rho))$ it holds
	\begin{eqnarray*}
		\left(\check{P}(D)w\right)(\varphi)&=&\int_C u(x)P(D)\varphi(x)\mathrm{d}x=\int_{C\cap B(a,\delta)}u(x)P(D)\varphi(x)\mathrm{d}x\\
		&=&\int_{B(a,\delta)}u(x)P(D)\varphi(x)\mathrm{d}x=\left(\check{P}(D)u\right)(\varphi)=0,
	\end{eqnarray*}
	i.e.~$\check{P}(D)w$ vanishes on $B(a,\rho)$. Since $a\in\partial C\backslash F_1$ is arbitrary, $\check{P}(D)w$ vanishes on
	\[\left(\R^d\backslash\overline{C}\right)\cup C\cup \left(\partial C\backslash F_1\right)=\R^d\backslash\left(\partial C\cap F_1\right)\]
	so that $\supp\check{P}(D)w\subset \partial C\cap F_1\subset F_1$.

    In summary, $w\in\mathscr{E}'(\R^d)$ satisfies $\supp w\subset F_2$, $\supp w\not\subset F_1$, and $\supp\check{P}(D)w\subset F_1$ which contradicts \eqref{eq:P-Runge by supports}.	
\end{proof}

Under suitable additional hypotheses, the necessary condition from Theorem \ref{thm: necessary condition for P-Runge for partially elliptic} for $P$-Runge pairs for smooth Whitney jets is also sufficient.

\begin{proposition}\label{prop:P-Runge R2}
Let $P(D)$ be a hypoelliptic partial differential operator on $\R^2$ which is \veryniceon{its linear characteristic cone} $\charcone(P)\neq\{0\}$. Moreover, let $F_1$ be a closed subset of $\R^2$ such that for all $\epsilon>0$ and $x\in\R^2$ the set $F_1\cap (B(x,\epsilon)+\charcone(P)^\perp)$ has at most countably many connected components. Then $F_1,\R^2$ is a $P$-Runge pair for smooth Whitney jets if and only if for all $x\in\R^2$ and $\epsilon>0$ the set $(\R^d\setminus F_1)\cap (B(x,\epsilon)+\charcone(P)^\perp)$ does not have a bounded connected component.

Moreover, assume that for all $\epsilon>0$, $x\in\R^2$, $y\in B(x,\epsilon)$ and all bounded connected components $B$ of $(\R^d\setminus F_1)\cap (B(x,\epsilon)+\charcone(P)^\perp)$ the set  $B\cap (y+\charcone(P)^\perp)$ has finitely many connected components and let $F_2$ be a closed subset of $\R^2$ containing $F_1$. Then $F_1,F_2$ is a $P$-Runge pair if and only if for all $x\in\R^2$ and $\epsilon>0$ the set $F_2$ does not contain a bounded connected component of $(\R^d\setminus F_1)\cap (B(x,\epsilon)+\charcone(P)^\perp)$.
\end{proposition}
\begin{proof}
For arbitrary closed set $F_2\subset\R^2$, the necessity follows from Theorem \ref{thm: necessary condition for P-Runge for partially elliptic}.

We shall prove sufficiency. Clearly, $\charcone(P)=\operatorname{span}\{v\}$ for some $v\in\R^2\backslash\{0\}$ and without loss of generality we may assume that $\charcone(P)=\operatorname{span}\{e_1\}=\R\times\{0\}$. Then $\charcone(P)^\perp=\{0\}\times\R$. Moreover, the first and second assumption imposed on $F_1$ mean, respectively, that:
\begin{enumerate}
    \item[($\ast$)] for all open bounded intervals $I$ the set $F_1\cap (I\times\R)$ has at most countably many connected components;
    \item[($\ast\ast$)] for all open bounded intervals $I$, $t\in I$ and all bounded connected components $B$ of $(\R^2\setminus F_1)\cap (I\times\R)$ the set $B\cap (\{t\}\times\R)$ has finitely many connected components.
\end{enumerate}
It remains to prove that if for all open bounded intervals $I$ the set $F_2$ does not contain a bounded connected component of $(\R^2\setminus F_1)\cap (I\times\R)$, then $F_1,F_2$ is a $P$-Runge pair.

In view of \cite[(iii)$\Rightarrow$(i) in Theorem 16]{CiasKalmes2025}, it is enough to show that for all $y\in\R^2$ and $\epsilon>0$ with $B(y,\epsilon)\subset\R^2\setminus F_1$, there is $x\in\R^2$ with $B(y,\epsilon)\cap \{(x_1,t)\colon t\in\R\}\neq\emptyset$ such that the connected component of $(\R^2\setminus F_1)\cap\{(x_1,t)\colon t\in\R\}$ which contains $B(y,\epsilon)\cap \{(x_1,t)\colon t\in\R\}$ is unbounded or it is not contained in $F_2$. We proceed by contraposition. Let us suppose that the last statement is not true, i.e.~we assume that
there are $y\in\R^2$ and $\epsilon>0$ with $B(y,\epsilon)\subset \R^2\setminus F_1$ such that for all $x\in\R^2$ with $B(y,\epsilon)\cap \{(x_1,t)\colon t\in\R\}\neq\emptyset$ the connected component of $(\R^2\setminus F_1)\cap\{(x_1,t)\colon t\in\R\}$ which contains $B(y,\epsilon)\cap \{(x_1,t)\colon t\in\R\}$ is bounded and is contained in $F_2$; this condition will be referred to as ($\ast\ast\ast$). 
Our aim is to show that there is an open interval $I$ for which the set $F_2$ contains a bounded connected component of $(\R^2\setminus F_1)\cap (I\times\R)$.

Let us assume that $F_1$ satisfies ($\ast$). By ($\ast\ast\ast$), there are bounded, open non-empty intervals $I_0',J_1\subset \R$ with $I_0'\times J_1\subset \R^2\setminus F_1$, and functions 
\[f\colon I_0'\to(-\infty,\inf J_1],\quad g\colon I_0'\to[\sup J_1,\infty)\] whose graphs are contained in $F_1$ and such that
\[\{(\xi,\eta)\in\R^2\colon \xi\in I_0' \text{ and } f(\xi)<\eta<g(\xi)\}\subset F_2.\]
Since $F_1\cap(I_0'\times \R)$ has at most countably many components, there are $s_0<t_0$ from the interval $I_0'$ such that the points $(s_0,f(s_0))$ and $(t_0,f(t_0))$ belong to the same component $C_0$ of $F_1\cap(I_0'\times \R)$. The same argument applied to the interval $I_0:=(s_0,t_0)$ and the function $g_{\mid{I_0}}$ leads to the conclusion that there are $s_1<t_1$ from $I_0$ such that the points $(s_1,g(s_1))$ and $(t_1,g(t_1))$ belong to the same component $C_1$ of $F_1\cap(I_0\times\R)$. In consequence, the connected component $B_1$ of the set $(\R^2\setminus F_1)\cap(I_1\times \R)$ containing $I_1\times J_1$ is bounded. Indeed, otherwise there is a continuous curve 
\[\gamma=(\gamma_1,\gamma_2)\colon[0,\infty)\to (\R^2\setminus F_1)\cap (I_1\times\R)\]
such that $\gamma(0)\in I_1\times J_1$ and $\sup_{t\geq0}|\gamma(t)|=\infty$. Then $\inf_{t\geq0}\gamma_2(t)=-\infty$ or $\sup_{t\geq0}\gamma_2(t)=\infty$ which contradicts the connectedness of $C_0$ or $C_1$, respectively.
This shows sufficiency in the case $F_2=\R^2$. 

Now, let us also assume that $F_1$ satisfies condition ($\ast\ast$).
For an open interval $I\subset I_1$, let $B_I$ be the (bounded) connected component of $(\R^2\setminus F_1)\cap (I\times\R)$ containing $I\times J_1$. We also define
\begin{align*}
U_I&:=\{(\xi_1,\xi_2)\in B_I\colon \xi_2>g(\xi_1)\},\\
L_I&:=\{(\xi_1,\xi_2)\in B_I\colon \xi_2<f(\xi_1)\}.
\end{align*}
We claim that there is an open interval $I\subset I_1$ such that the sets $U_I$ and $L_I$ both have empty interior. For such an interval $I$ the set
\[B_I\setminus(U_I\cup L_I)=\{(\xi_1,\xi_2)\in B_I\colon f(\xi_1)\leq \xi_2\leq g(\xi_1)\}\subset F_2\]
is dense in $B_I$, and consequently $B_I\subset F_2$, which is our final goal.

Suppose that the above claim is not true. Then for all open intervals $I\subset I_1$, the set $U_I$ or $L_I$ has a non-empty interior. 
We will recursively construct sequences $(I_k')_{k\geq2}$, $(I_k)_{k\geq2}$, $(J_k)_{k\geq2}$ of non-empty, open intervals and connected components $C_k$ of $F_1\cap(I_k'\times \R)$ such that for each $k\geq 2$:
\begin{enumerate}
    \item $I_{k}\subset \overline{I_{k}'}\subset I_{k-1}\subset \ldots\subset I_2\subset \overline{I_2'}\subset I_1$;
    \item the intervals $J_1,\ldots J_k$ are pairwise disjoint;
    \item either $I_k\times J_k\subset U_{I_k}$ or $I_k\times J_k\subset L_{I_k}$;
    \item if $I_k\times J_k\subset L_{I_k}$ then the set $C_k$ contains the points $(\inf I_k,f(\inf I_k))$ and $(\sup I_k,f(\sup I_k))$, and if $I_k\times J_k\subset U_{I_k}$ then the set $C_k$ contains the points $(\inf I_k,g(\inf I_k))$ and $(\sup I_k,g(\sup I_k))$;
    \item the connected components $X_1^{(k)},X_2^{(k)},\ldots X_k^{(k)}$ of the set $(\R^2\setminus F_1)\cap(I_k\times\R)$ containing $I_k\times J_1, I_k\times J_2,\ldots, I_k\times J_k$, respectively, are pairwise disjoint. 
\end{enumerate}

We start with the case $k=2$. By assumption, one of the set $U_{I_1}$ or $L_{I_1}$ has a non-empty interior. Without loss of generality, we may assume that $L_{I_1}$ contains an open rectangle $I_2'\times J_2$, where $\overline{I_2'}\subset I_1$. Since $F_1\cap(I_2'\times \R)$ has at most countable many components, there are $s_2<t_2$ from the interval $I_2'$ such that the points $(s_2,f(s_2))$ and $(t_2,f(t_2))$ belong to the same component $C_2$ of $F_1\cap(I_2'\times \R)$. We set $I_2:=(s_2,t_2)$. Then the connected components $X_1^{(2)},X_2^{(2)}$ of $(\R^2\setminus F_1)\cap(I_2\times\R)$ containing $I_2\times J_1, I_2\times J_2$, respectively, are disjoint. Otherwise, one could find a continuous curve 
\[\gamma\colon[0,1]\to (\R^2\setminus F_1)\cap (I_2\times\R)\]
such that $\gamma(0)\in I_2\times J_1$ and $\gamma(1)\in I_2\times J_2$ destroying the connectedness of $C_2$. Note that this also implies that $J_1$ and $J_2$ are disjoint.

Let us assume that we have already constructed  open intervals 
\[I_2',\ldots,I_k',\quad I_2,\ldots,I_k, \quad J_2,\ldots,J_k\] 
and connected components $C_2,\ldots, C_k$ of 
\[F_1\cap(I_2'\times \R),\ldots,F_1\cap(I_k'\times \R),\]
respectively, that satisfy the properties (1)--(5) above.
Then at least one of the set $U_{I_k}$ or $L_{I_k}$ has a non-empty interior. Without loss of generality, we may assume that $L_{I_k}$ contains an open rectangle $I_{k+1}'\times J_{k+1}$, where , where $\overline{I_{k+1}'}\subset I_k$. Since $F_1\cap(I_{k+1}'\times \R)$ has at most countably many components, there are $s_{k+1}<t_{k+1}$ from the interval $I_{k+1}'$ such that the points $(s_{k+1},f(s_{k+1}))$ and $(t_{k+1},f(t_{k+1}))$ belong to the same component $C_{k+1}$ of $F_1\cap(I_{k+1}'\times \R)$. We set $I_{k+1}:=(s_{k+1},t_{k+1})$. Then the connected components $X_1^{(k+1)},\ldots, X_{k}^{(k+1)}$ of $(\R^2\setminus F_1)\cap(I_{k+1}\times\R)$ containing $I_{k+1}\times J_1,\ldots,I_{k+1}\times J_{k}$ do not intersect the connected component $X_{k+1}^{(k+1)}$ of $(\R^2\setminus F_1)\cap(I_{k+1}\times\R)$ containing $I_{k+1}\times J_{k+1}$. Otherwise, one finds a continuous curve 
\[\gamma\colon[0,1]\to (\R^2\setminus F_1)\cap (I_{k+1}\times\R)\]
such that $\gamma(0)\in I_{k+1}\times J_m$ and $\gamma(1)\in I_{k+1}\times J_{k+1}$ for some $m=1,\ldots, k$ which contradicts the connectedness of the set $C_{k+1}$. This also shows the disjointness of $J_j$ and $J_{k+1}$, $j=1,\ldots,k$. Since $J_1,\ldots,J_{k}$ are pairwise disjoint, (2) holds for $k+1$, which, by the above, implies that (5) holds also for $k+1$.

From property (1) and Cantor's intersection theorem it follows that 
\[\bigcap_{k\geq 1}I_k\supset\bigcap_{k\geq 1}\overline{I_{k+1}}\neq\emptyset.\]
Let us fix $t\in \bigcap_{k\geq 1} I_k$ and let $M\in\N$ be the number of connected components of $B_1\cap (\{t\}\times\R)$. Notice that the sets
\[(I_{M+1}\times J_1)\cap (\{t\}\times\R),\ldots,(I_{M+1}\times J_{M+1})\cap (\{t\}\times\R) \]
are not empty and, by property (5), belong to different connected components $X_1^{(M+1)},\ldots, X_{M+1}^{(M+1)}$ of $(\R^2\setminus F_1)\cap (I_{M+1}\times\R)$. Next, by property (3), 
\[I_k\times J_k\subset L_{I_k}\cup U_{I_k}\subset B_{I_k}\subset B_{I_1}=B_1\]
for all $k\geq 2$ and clearly $I_1\times J_1\subset B_1$, whence
\[I_{M+1}\times J_k\subset B_1\]
for all $k\geq 1$. Therefore,  
the sets $X_1^{(M+1)},\ldots, X_{M+1}^{(M+1)}$ are contained in $B_1$, so these sets are also connected components of $B_1\cap (I_{M+1}\times\R)$.
Let $\mathcal{X}$ be the class of all connected components of $B_1\cap (I_{M+1}\times\R)$ with a non-empty intersection with the line $\{t\}\times\R$ and let $\mathcal{X}_t$ be the class of all connected components of $B_1\cap (\{t\}\times\R)$.
We consider the map $\Phi\colon\mathcal{X}\to\mathcal{X}_t$ that assigns to each set $X\in\mathcal{X}$ the unique set $Y\in\mathcal{X}_t$ such that $Y\subset X$. We have $|\mathcal{X}|> M=|\mathcal{X}_t|$ so, by the pigeonhole principle, there are $X_1\neq X_2$ from the class $\mathcal{X}$ such that $Y:=\Phi(X_1)=\Phi(X_2)$. But then $Y\subset X_1$, $Y\subset X_2$, and consequently $X_1\cap X_2$  contains a non-empty set $Y$, which contradicts $X_1\cap X_2=\emptyset$. This final contradiction completes the proof.
\end{proof}

Proposition \ref{prop:P-Runge R2} can be applied to the class of subanalytic subsets of $\R^2$; see \cite[Definition 3.1]{BierstoneMilman1988}.
\begin{corollary}
Let $P(D)$ be a hypoelliptic partial differential operator on $\R^2$ which is \veryniceon{its linear characteristic cone} $\charcone(P)\neq \{0\}$. Moreover, let $F_1\subset F_2$ be closed subsets of $\R^2$, where $F_1$ is a subanalytic set.  Then $F_1,F_2$ is a $P$-Runge pair for smooth Whitney jets if and only if for all $x\in\R^2$ and $\epsilon>0$ the set $F_2$ does not contain a bounded connected component of $(\R^d\setminus F_1)\cap (B(x,\epsilon)+\charcone(P)^\perp)$.
\end{corollary}

\begin{proof}
We shall use basic properties of subanalytic sets; see \cite[Remark after Definition  3.1, Gabrielov’s Complement Theorem 3.10 and Theorem 3.14]{BierstoneMilman1988}. Without loss of generality, we may assume that $\charcone(P)=\operatorname{span}\{e_1\}=\R\times\{0\}$.
If $F_1\subset\R^2$ is subanalytic, then so are the sets of the form $F_1\cap ((a,b)\times\R)$ and $(\R^2\setminus F_1)\cap((a,b)\times\R)$,  where $a<b$; here, we use the fact that the intersection of subanalytic sets as well as the complement of a subanalytic set are subanalytic. Next, the collection of connected components of $F_1\cap ((a,b)\times\R)$ is locally finite, so  $F_1\cap ((a,b)\times\R)$ has at most countable many components. Moreover, every connected component of a subanalytic set is subanalytic.  
Therefore, each bounded connected component $B$ of $(\R^2\setminus F_1)\cap ((a,b)\times\R)$ is a relatively compact subanalytic set, and consequently the sets of the form $B\cap(\{t\}\times\R)$ have only finitely many components (here we use \cite[Theorem 3.14]{BierstoneMilman1988}). Hence, by Proposition \ref{prop:P-Runge R2}, $F_1,F_2$ is a $P$-Runge pair.
\end{proof}

For specific non-hypoelliptic operators, the conclusion of Theorem \ref{thm: necessary condition for P-Runge for partially elliptic} can be strengthened, as is shown in the following example.  
\begin{example}\label{example:P-Runge non-hypoelliptic}
(i) Consider the non-hypoelliptic operator $P(D)=\partial_d$, with the corresponding polynomial $P\in\C[\xi_1,\ldots,\xi_d]$, $P(\xi_1,\ldots,\xi_d)=i\xi_d$. We will show that closed sets $F_1, F_2\subset \R^d$ form a $P$-Runge pair for smooth Whitney jets if and only if for all $x\in\R^d$ the set $F_2$ does not contain a bounded connected component of $(\R^d\setminus F_1)\cap (x+\charcone(P)^\perp)$. 

For necessity, let us assume that there is $x\in\R^d$ such that $F_2$ contains a bounded connected component $C$ of 
\[(\R^d\setminus F_1)\cap (x+\charcone(P)^\perp)=(\R^d\setminus F_1)\cap\{(x',t)\colon t\in\R\},\]
where $x'=(x_1,\ldots,x_{d-1})$.
Then $\pi_d(C)=(a,b)$ for some $a<b$ with $\pi_d\colon\R^d\to\R$ defined by $\pi_d(x):=x_d$.
Let $u\in\mathscr{E}'(\R^d)$, 
\[u:=\delta_{x'}\otimes\mathds{1} _{[a,b]},\]
where
$\delta_{x'}\in\mathscr{E'}(\R^{d-1})$ is the Dirac delta distribution at the point $x'$ and $\mathds{1} _{[a,b]}\in\mathscr{E}'(\R)$ is the characteristic function of the interval $[a,b]$ (see \cite[Section 5.1]{Hor1} and \cite[Section 1.5]{OrtnerWagner2015} for definition and basic properties of tensor products of distributions). 
Because
\begin{equation*}
\begin{gathered}
    \partial_du=\delta_{x'}\otimes\partial_d\mathds{1} _{[a,b]},\\
    \supp P(-D)u=\supp\partial_du=\supp\delta_{x'}\times\supp\partial_d\mathds{1} _{[a,b]}=\{(x',a),(x',b)\}\subset F_1,
\end{gathered}
\end{equation*}
and
\[\supp u=\supp \delta_{x'}\times \supp \mathds{1}_{[a,b]}=\{x'\}\times[a,b]= \overline C\subset F_2,\]
by \cite[Theorem 5]{CiasKalmes2025}, $F_1, F_2$ is not a $P$-Runge pair. Sufficiency follows from \cite[Theorem 16]{CiasKalmes2025}.

(ii) Let $F_1=\{(0,0)\}\cup\{(0,1)\}$, $F_2=\R^2$ and $P(\xi_1,\xi_2)=i\xi_2$. Clearly, the pair $F_1,F_2$ satisfies the condition: for all $x\in\R^2$ and $\epsilon>0$ the set
\[(\R^2\setminus F_1)\cap (B(x,\epsilon)+\charcone(P)^\perp)
=(\R^2\setminus F_1)\cap (x_1-\epsilon,x_1+\epsilon)\times\R\]
does not admit a bounded connected component. But $C:=\{0\}\times(0,1)$ is a bounded connected component of the set $(\R^2\setminus F_1)\cap \{0\}\times\R$ and, proceeding as in (i), we conclude that $F_1,F_2$ is not a $P$-Runge pair. 
\end{example}

Having in mind \cite[Theorem 16, implication (iv)$\Rightarrow$(i)]{CiasKalmes2025}, the previous Proposition \ref{prop:P-Runge R2}, and Example \ref{example:P-Runge non-hypoelliptic}, we state the following conjecture.

\begin{conjecture}\label{conjecture:characterization of P-Runge}
	Let $P(D)$ be a hypoelliptic partial differential operator on $\R^d$ which is \veryniceon{its linear characteristic cone} $\charcone(P)\neq \{0\}$. Moreover, let $F_1\subset F_2$ be closed subsets of $\R^d$.
	$F_1,F_2$ is a $P$-Runge pair for smooth Whitney jets if and only if all $x\in\R^d$ and $\epsilon>0$ the set $F_2$ does not contain a bounded connected component of $(\R^d\setminus F_1)\cap (B(x,\epsilon)+\charcone(P)^\perp)$. 
\end{conjecture}

In the remainder of this section, for non-elliptic $P$ which are \veryniceon{their linear characteristic cone}, we study $P$-Runge pairs of open subsets of $\R^d$. In order to do this, we need the following lemma.

\begin{lemma}\label{lemma:preparation P-Runge pairs for open sets}
	Let $P\in\C[\xi_1,\ldots,\xi_d]$ be \veryniceon{a linear subspace $Z\neq\{0\}$ of $\R^d$}. Moreover, let $X_1, X_2$ be open subsets of $\R^d$ with $X_1\subset X_2$. Consider the following assertions.
	\begin{enumerate}
		\item[\upshape{(i)}] For every $\varphi\in\mathscr{D}(X_2)$ with $\supp\check{P}(D)\varphi\subset X_1$ it holds $\supp\varphi\subset X_1$.
		\item[\upshape{(ii)}] There is no $x\in\R^d$ such that $X_2$ contains a bounded connected component of $\left(\R^d\backslash X_1\right)\cap\left(x+Z^\perp\right)$.
	\end{enumerate}
	Then {\upshape(i)} implies {\upshape(ii)}. Additionally, if $X_1$ is $P$-convex for supports and $P$ is \veryniceon{its linear characteristic cone}, then {\upshape(ii)} also implies {\upshape(i)}.
\end{lemma}

\begin{proof}
	We abbreviate again $W:=Z^\perp$ and $l:=\operatorname{dim}(Z)\in\{1,\ldots,d-1\}$. To show that (i) implies (ii), after the usual orthogonal change of variables, we can assume without loss of generality that
	\[W=\{x\in\R^d\colon x_1=\ldots=x_l=0\}\]
	and we argue by contradiction. Thus, we assume that there is $x_0\in\R^d$ such that $\left(\R^d\backslash X_1\right)\cap (x_0+W)$ has a bounded, hence compact, connected component $C$ which is contained in $X_2$. Fix $y\in C$. Then $x_0+W=y+W$. By compactness of $C$ and
	\[y+W=\{(y_1,\ldots,y_l)\}\times\R^{d-l}\]
	there are open and bounded subsets $V\subset U\subset\R^{d-l}$ such that
	\begin{eqnarray}\label{eq:auxiliary1}
		C&\subset& \{(y_1,\ldots,y_l)\}\times V\subset \{(y_1,\ldots,y_l)\}\times \overline{V} \subset \{(y_1,\ldots,y_l)\}\times U\\
		&\subset& \{(y_1,\ldots,y_l)\}\times \overline{U}\subset (X_1\cup C)\cap\left(\{(y_1,\ldots,y_l)\}\times\R^{d-l}\right).\nonumber
	\end{eqnarray}
	Because $\{(y_1,\ldots,y_l)\}\times \overline{U}\backslash V$ is a compact subset of $X_1\cap (y+W)$ there is $\epsilon>0$ with
	\begin{enumerate}
		\item[(a)] $\overline{B_l((y_1,\ldots,y_l),\epsilon)}\times(\overline{U}\backslash V)\subset X_1$,
		\item[(b)] $\overline{B_l((y_1,\ldots,y_l),\epsilon)}\times\overline{U}\subset X_2$,
	\end{enumerate} 
	where, as in Theorem \ref{th:zero solution standard form}, for $v\in\R^l, \delta>0$, we denote by $B_l(v,\delta)$ the euclidean ball in $\R^l$ of radius $\delta$ about $v$. An application of Theorem \ref{th:zero solution} to $\check{P}$ gives the existence of $u\in\mathscr{E}_{\check{P}}(\R^d)$ such that
	\begin{equation}\label{eq:auxiliary2}
		\overline{B_l((y_1,\ldots,y_l),\epsilon/2)}\times\R^{d-l}\subset \supp u\subset \overline{B_l((y_1,\ldots,y_l),\epsilon)}\times\R^{d-l}
	\end{equation}
	Fix $\psi\in\mathscr{D}(U)$ such that $\psi=1$ in a neighborhood of $\overline{V}$. Then
	\[\varphi\colon\R^d\rightarrow\C,x\mapsto u(x)\psi(x_{l+1},\ldots,x_d)\]
	belongs to $\mathscr{D}\left(\overline{B_l((y_1,\ldots,y_l),\epsilon)}\times U\right)$, in particular $\supp\varphi\subset X_2$. By
    \[B(y,\epsilon/2)+W=B_l((y_1\ldots,y_l),\epsilon/2)\times\R^{d-l},\]
    \eqref{eq:auxiliary1} and \eqref{eq:auxiliary2}, it holds $C\subset\supp \varphi$ and by (b), as in the proof of Theorem \ref{th:necessary P-convexity}, we have
	\begin{eqnarray*}
	    \supp\check{P}(D)\varphi&\subset&\supp u\cap\left(\R^l\times (\overline{U}\backslash V)\right)\\
        &\subset& \left(B_l((y_1,\ldots,y_l),\epsilon)\times\overline{U}\right)\cap\left(\R^l\times (\overline{U}\backslash V)\right)\\
        &=&B_l((y_1,\ldots,y_l),\epsilon)\times (\overline{U}\backslash V)\subset X_1.
	\end{eqnarray*}
	Because $C\subset\supp\varphi\cap (\R^d\backslash X_1)$, but $\varphi\in\mathscr{D}(X_2)$ satisfies $\supp\check{P}(D)\varphi\subset X_1$, we derive the desired contradiction to (i). Hence, (i) implies (ii)
	
	Under the additional hypothesis that $X_1$ is $P$-convex for supports, the proof that (ii) implies (i) is done exactly as in \cite[Proof of Theorem 1, (iii) implies (i) and (ii)]{Kalmes2021}, where the special case $\operatorname{dim}(Z)=1$ is studied and where one has to apply Corollary \ref{cor:characterization P-convexity} instead of \cite[Theorem 7]{Kalmes2021}.
\end{proof}

\begin{example}\label{example:wave operator 2}
Keeping the notation from Example \ref{example} (iii), let $P\in\C[\xi_0,\xi_1,\ldots,x_d]$ be a polynomial corresponding to a partial differential operator $P(D)$ whose principal part is the $d$-dimensional wave operator, $d\geq1$. 

(i) By Example \ref{example} (iii) and Theorem \ref{thm: necessary condition for P-Runge for partially elliptic}, if closed subsets $F_1,F_2$ of $\R^{d+1}$ form a $P$-Runge pair for smooth Whitney jets, then for all characteristic hyperplanes $H$ for $P$ and all $\epsilon>0$ the set $F_2$ does not contain a bounded connected component of $(\R^{d+1}\setminus F_1)\cap (H+B(0,\epsilon))$.

(ii) By Example \ref{example} (iii), Lemma \ref{lemma:preparation P-Runge pairs for open sets} and \cite[Theorem 6]{Kalmes2021}, if open $P$-convex for supports sets $X_1,X_2\subset\R^{d+1}$ form a $P$-Runge pair then for all characteristic hyperplanes $H$ for $P$ the set $X_2$ does not contain a bounded connected component of $(\R^{d+1}\setminus X_1)\cap H$.
\end{example}

Our final theorem in this section characterizes the open sets which are $P$-Runge pairs for differential operators which are \veryniceon{their linear characteristic cone} under the additional hypothesis, that both sets are $P$-convex for supports. It should be noted that for an elliptic differential operator all open sets are  $P$-convex for supports. Hence the next theorem may be considered as an analogue to the celebrated Lax-Malgrange Theorem for non-elliptic differential operators which are \veryniceon{their linear characteristic cone}.

\begin{theorem}\label{th:P-Runge pairs for open sets}
	Let $P\in\C[\xi_1,\ldots,\xi_d]$ be \veryniceon{its linear characteristic cone $\charcone(P)\neq \{0\}$}. Moreover, let $X_1, X_2$ be open subsets of $\R^d$ which are $P$-convex for supports, such that $X_1\subset X_2$. Then, the following are equivalent.
	\begin{enumerate}
		\item[\upshape{(i)}] $X_1, X_2$ is a $P$-Runge pair.
		\item[\upshape{(ii)}] $X_1, X_2$ is a $P$-Runge pair for distributions.
		\item[\upshape{(iii)}] There is no $x\in\R^d$ such that $X_2$ contains a bounded connected component of $\left(\R^d\backslash X_1\right)\cap\left(x+\charcone(P)^\perp\right)$.
	\end{enumerate}
\end{theorem}

\begin{proof}
	The theorem follows immediately by combining Lemma \ref{lemma:preparation P-Runge pairs for open sets} with \cite[Theorem 6]{Kalmes2021}.
\end{proof}

\section{Square systems of constant coefficient partial differential equations}\label{sec:systems}

In this section, we shall combine the results from the previous sections with those from \cite{Kalmes2021,CiasKalmes2025} to establish Runge type approximation results for square systems of constant coefficient, linear partial differential equations. In order to do so, we first recall some general notation and facts (see \cite[Section 3.8]{Hormander1976}). For $N\in\N$ we equip $\mathscr{E}(X)^N$, $\mathscr{D}'(X)^N$, and $\mathscr{E}(F)^N$ with the product topology, where $X\subset\R^d$ is open and $F\subset\R^d$ is closed, and we interpret elements from these spaces as column vectors. For $P_{jk}\in \C[\xi_1,\ldots,\xi_d]$, $1\leq j,k\leq N$, the $N\times N$ matrix of constant coefficients partial differential operators $P(D)=(P_{jk}(D))_{1\leq j,k\leq N}$ defines a continuous linear operator on each of the spaces $\mathscr{E}(X)^N$, $\mathscr{D}'(X)^N$, and $\mathscr{E}(F)^N$ in an obvious way.

Denoting for $\xi\in \C^d$ by $\Pad(\xi)=(\Pad_{jk}(\xi))_{1\leq j,k\leq N}$ the adjoint matrix of the matrix $P(\xi)=(P_{jk}(\xi))_{1\leq j,k\leq N}$ consisting of the cofactors of $(P_{kj}(\xi))_{1\leq j,k\leq N}$, and by $\detP(\xi)$ the determinant of $P(\xi)$, it holds $\Pad_{jk}, \detP\in\C[\xi_1,\ldots,\xi_d]$ and
\begin{equation}\label{eq:fundamental for systems 1}
	P(\xi)\Pad(\xi)=\Pad(\xi)P(\xi)=\detP(\xi)I_N,
\end{equation}
where by $I_N=(\delta_{jk})_{1\leq j,k\leq N}$.

\begin{proposition}\label{prop:surjcetivity of systems}
	Let $E\in\{\mathscr{E}(X), \mathscr{D}'(X), \mathscr{E}(F)\}$. Then, the following are equivalent.
	\begin{itemize}
		\item[(i)] $\detP(D)$ is surjective on $E$.
		\item[(ii)] $P(D)$ and $\Pad(D)$ are both surjective on $E^N$. 
	\end{itemize}
\end{proposition}

\begin{proof}
	The assertion follows immediately from equation \eqref{eq:fundamental for systems 1} which implies
	\begin{equation}\label{eq:fundamental for systems 2}
		P(D)\Pad(D)u=\Pad(D)P(D)u=\detP(D)I_N u
	\end{equation}
	for each $u\in E^N$.
\end{proof}

\begin{remark}\label{rem:surjectivity of Whitney}
	Assume that $\detP\not\equiv 0$. Then, combining the previous proposition with \cite[Introduction to Chapter 6]{Frerick-Habilitation}, we obtain the surjectivity of both $P(D)$ and $\Pad(D)$ on $\mathscr{E}(F)^N$.
\end{remark}

\begin{definition}
	For $E\in\{\mathscr{E}(X), \mathscr{D}'(X), \mathscr{E}(F)\}$ we set
	\[E_P=\{u\in E^N\colon P(D)u=0\}\]
	which we also denote by $\mathscr{E}_P(X)$, $\mathscr{D}'_P(X)$, and $\mathscr{E}_P(F)$, respectively, and which we equip with the respective subspace topology induced by $E^N$. Analogously to the case of a single partial differential operator, a pair of open subsets $X_1\subset X_2$ of $\R^d$, respectively, a pair of closed subset $F_1\subset F_2$ of $\R^d$ is called a \emph{$P$-Runge pair (for distributions, or smooth Whitney jets)}, respectively, if the restriction mapping
    \[
    \mathscr{E}_P(X_2)\rightarrow \mathscr{E}_P(X_1), u\mapsto u_{|X_1},\quad \mathscr{D}'_P(X_2)\rightarrow \mathscr{D}'_P(X_1), u\mapsto u_{|X_1},
    \]
    or
    \[\mathscr{E}_P(F_2)\rightarrow \mathscr{E}_P(F_1), u\mapsto u_{|F_1}\]
    has dense range, respectively.
\end{definition}

\begin{proposition}\label{prop:fundamental for Runge for systems}
	Let $E\in\{\mathscr{E}(X), \mathscr{D}'(X), \mathscr{E}(F)\}$. Then
	\[\Pad(D):E_{\detP}^N\rightarrow E_P\]
	is a correctly defined, linear and continuous mapping. Additionally, if $\Pad(D)$ is surjective on $E^N$, the same holds for
	\[\Pad(D):E_{\detP}^N\rightarrow E_P.\]
\end{proposition}

\begin{proof}
	For $u\in E_{\detP}^N$ it follows from \eqref{eq:fundamental for systems 2} that $\Pad(D) u$ belongs to $E_P$. Hence,
	\[\Pad(D):E_{\detP}^N\rightarrow E_P\]
	is a correctly defined, obviously linear and continuous mapping.
	
	Now, we assume that $\Pad(D)$ is surjective on $E^N$. Let $v$ be in $E_P$ and let $u\in E^N$ be such that $\Pad(D)u=v$. It follows from $v\in E_P$ and \eqref{eq:fundamental for systems 2} that
	\[0=P(D)v=P(D)\Pad(D)u=\detP(D) I_N u\]
	so that $u\in E_{\detP}^N$ which proves the surjectivity of $\Pad(D):E_{\detP}^N\rightarrow E_P$.
\end{proof}

\begin{remark}\label{rem:dense subset}
	(i) With the notation from Proposition \ref{prop:fundamental for Runge for systems}, if $E_j$, $j=1,\ldots,N$, are dense subsets of $E_{\detP}$ then $\Pad(D)\left(E_1\times\cdots\times E_N\right)$ is dense in $E_P$ whenever $\Pad(D)$ is surjective on $E^N$.
	
	Thus, for convex open subsets $X$ of $\R^d$ and $P$ such that $\detP\not\equiv 0$, the linear span of the set
	\[\{\Pad(D)u \colon u=(u_1,\ldots,u_N), u_j\text{ exponential solution for }\detP(D), j=1,\ldots, N \}\]
	is dense in both $\mathscr{E}_P(X)$ and $\mathscr{D}'_P(X)$ by \cite[Chapitre Premiere, \S 2, Th\'eor\`eme 2]{Malgrange1955}, Proposition \ref{prop:fundamental for Runge for systems}, and \cite[Theorem 10.6.2, Corollary 10.6.8, Theorem 10.7.2, Corollary 10.7.10]{Hor2}. In particular, $\mathscr{E}_P(X)$ is a dense subspace of $\mathscr{D}'_P(X)$ for convex open subsets $X$ of $\R^d$. 
	
	Recall that for $Q\in\C[\xi_1,\ldots,\xi_d]$ an exponential solution $u$ for $Q(D)$ is a smooth function on $\R^d$ satisfying $Q(D)u=0$ in $\R^d$ and which is of the form
	\[u(x)=f(x)\exp(i\langle x,\zeta\rangle),\]
	where $f$ is a polynomial and $\zeta\in\C^d$ (see \cite[Definition 7.3.5]{Hor1}). It is easily verified that necessarily $Q(\zeta)=0$ and that the degree of $f$ is bounded by $m-1$ if $\zeta$ is a root of $Q$ of multiplicity $m$. Indeed, this follows from the identity
	\begin{eqnarray*}
		Q(D)\left(f(x)\exp\left(i\langle \zeta,x\rangle\right)\right)&=&\sum_{\alpha\in\N_0^d}D^\alpha f(x)\frac{Q^{(\alpha)}(D)\left(\exp(i\langle\zeta,x\rangle)\right)}{\alpha!}\\
		&=&\left(\exp(i\langle\zeta,x\rangle)\right) \sum_{\alpha\in\N_0^d}D^\alpha f(x)\frac{Q^{(\alpha)}(\zeta)}{\alpha!},
	\end{eqnarray*}
	where as usual $Q^{(\alpha)}(\xi)=\partial^\alpha Q(\xi)$, combined with the linear independence of the non-vanishing partial derivatives of $Q$.
	
	Similarly, for convex open subsets $X$ of $\R^d$ and $P$ such that $\detP\not\equiv 0$ and such that every irreducible factor of $\detP$ vanishes at the origin, the linear span of the set
	\begin{equation}\label{eq:exponential solutions}
		\{\Pad(D)u \colon u=(u_1,\ldots,u_N), u_j\in\mathscr{E}_{\detP}(X), u_j\text{ polynomial, } j=1,\ldots, N \}
	\end{equation}
	is dense in both $\mathscr{E}_P(X)$ and $\mathscr{D}'_P(X)$ by \cite[Chapitre Premiere, \S 2, Th\'eor\`eme 2']{Malgrange1955}.
	
	(ii) Assume that $\detP$ is hypoelliptic. Then, by \eqref{eq:fundamental for systems 2} it follows immediately that $\mathscr{D}'_P(X)=\mathscr{E}_P(X)$ as vector spaces for every open subset $X$ of $\R^d$. Let $F$ be a fundamental solution for $\detP(D)$ and $E:=(\Pad_{j,k}(D)F)_{j,k}$. Then, $E\in\mathscr{D}'(\R^d)^{N\times N}$ is a two-sided fundamental matrix for $P(D)$ (see e.g.~\cite[Section 2.1]{OrtnerWagner2015}). Since $F_{|\R^d\backslash\{0\}}\in\mathscr{E}(\R^d\backslash\{0\})$ it holds $E_{|\R^d\backslash\{0\}}\in\mathscr{E}(\R^d\backslash\{0\})^{N\times N}$. By the same reasoning as in \cite[Proof of Theorem 4.4.2]{Hor1} one concludes that the bijection
	\[i:\mathscr{D}'_P(X)\rightarrow\mathscr{E}_P(X), f\mapsto f\]
	is continuous. Trivially, $i^{-1}$ is continuous, too, so that $\mathscr{D}'_P(X)=\mathscr{E}_P(X)$ as locally convex spaces. Similarly, $\mathscr{D}'_{\Pad}(X)=\mathscr{E}_{\Pad}(X)$ as locally convex spaces.
\end{remark}

\begin{theorem}\label{theo:Runge for systems}
	Let $X_1\subset X_2$ be open subsets of $\R^d$, $F_1\subset F_2$ be closed subsets of $\R^d$.
	\begin{itemize}
		\item[(i)] Assume that $\Pad(D)$ is surjective on $\mathscr{E}(X_1)^N$ and that $X_1,X_2$ is a $\detP$-Runge pair. Then, the set of restrictions to $X_1$ of elements from the set $\Pad(D)\left(\mathscr{E}_{\detP}(X_2)^N\right)$ is dense in $\mathscr{E}_P(X_1)$. In particular, $X_1, X_2$ is a $P$-Runge pair.
		\item[(ii)] Assume that $\Pad(D)$ is surjective on $\mathscr{D}'(X_1)^N$ and that $X_1,X_2$ is a $\detP$-Runge pair for distributions. Then, the set of restrictions to $X_1$ of elements from the set $\Pad(D)\left(\mathscr{D}'_{\detP}(X_2)^N\right)$ is dense in $\mathscr{D}'_P(X_1)$. In particular, $X_1, X_2$ is a $P$-Runge pair for distributions.
		\item[(iii)] Assume that $\detP\not\equiv 0$ and that $F_1, F_2$ is a $\detP$-Runge pair for smooth Whitney jets. Then, the set of restrictions to $X_1$ of $\Pad(D)\left(\mathscr{E}_{\detP}(F_2)^N\right)$ is dense in $\mathscr{E}_P(F_1)$. In particular, $F_1, F_2$ is a $P$-Runge pair for smooth Whitney jets.
	\end{itemize}
\end{theorem}

\begin{proof}
	(i) Fix $v\in \mathscr{E}_P(X_1)$ and let $V$ be a neighborhood of $v$ in $\mathscr{E}_P(X_1)$. By Proposition \ref{prop:fundamental for Runge for systems} there is $u=(u_1,\ldots,u_N)\in \mathscr{E}_{\detP}(X_1)^N$ with $\Pad(D) u=v$. For $j=1,\ldots, N$, let $U_j$ be neighborhoods of $u_j$ in $\mathscr{E}_{\detP}(X_1)$ such that
	\[\Pad(D)\left(U_1\times \cdots\times U_N\right)\subset V.\]
	Since $X_1,X_2$ is a $\detP$-Runge pair, for $j=1,\ldots,N$ there are $w_j\in \mathscr{E}_{\det P}(X_2)$ with $w_{j| X_1}\in U_j$. Then, with $w=(w_1,\ldots,w_N)^T$, by Proposition \ref{prop:fundamental for Runge for systems}, $\Pad(D)w\in\mathscr{E}_P(X_2)$. Because we also we have
	\[\left(\Pad(D)w\right)_{|X_1}=\Pad(D)(w_{|X_1})\in V,\]
	assertion (i) follow.
	
	The proofs of (ii) and (iii) are so similar to the proof of (i) that they are left to the reader. We only point out that the proof of (iii) uses Remark \ref{rem:surjectivity of Whitney}.
\end{proof}

As an immediate consequence of the previous theorem and Remark \ref{rem:dense subset} we obtain the following result.

\begin{corollary}\label{cor:density of exponential solutions}
	Let $X_1$ be an open subset of $\R^d$ and $F_1$ be a closed subset of $\R^d$.
	\begin{itemize}
	\item[(a)] Let $G$ be the linear span of the set
	\[\{\Pad(D)u\colon u= (u_1,\ldots,u_N), u_j\text{ exponential solution for }\detP(D), j=1,\ldots, N \}.\]
	\begin{itemize}
		\item[(a-i)] Assume that $\Pad(D)$ is surjective on $\mathscr{E}(X_1)^N$ and that $X_1,\R^d$ is a $\detP$-Runge pair. Then, $G$ is dense in $\mathscr{E}_P(X_1)$ and in $\mathscr{D}'_P(X_1)$.
		\item[(a-ii)] Assume that $\detP\not\equiv 0$ and that $F_1, \R^d$ is a $\detP$-Runge pair for smooth Whitney jets. Then, $G$ is dense in $\mathscr{E}_P(F_1)$.
	\end{itemize}
	\item[(b)] Let
	\[\mathcal{P}=\{u=(u_1,\ldots,u_N)\in \mathscr{E}_{P}(\R^d)\colon u_j\text{ polynomial}\}.\]
	Assume that every irreducible factor of $\detP$ vanishes at the origin. 
	\begin{itemize}
		\item[(b-i)] Assume that $\Pad(D)$ is surjective on $\mathscr{E}(X_1)^N$ and that $X_1,\R^d$ is a $\detP$-Runge pair. Then, $\mathcal{P}$ is dense in $\mathscr{E}_P(X_1)$ and in $\mathscr{D}'_P(X_1)$.
		\item[(b-ii)] Assume that $\detP\not\equiv 0$ and that $F_1, \R^d$ is a $\detP$-Runge pair for smooth Whitney jets. Then, $\mathcal{P}$ is dense in $\mathscr{E}_P(F_1)$.
	\end{itemize}
	\end{itemize}
\end{corollary}

\begin{remark}\label{rem:Hoelder continuous boundary}
	Let $X\subset\R^d$ be open. As usual, we write
	\[C^\infty(\overline{X})=\{f\in \mathscr{E}(X)\colon \partial^\alpha f \text{ has a continuous extension to }\overline{X} \text{ for all }\alpha\in\N_0^d \}\] 
	and we equip $C^\infty(\overline{X})$ with the family of seminorms $\{\|\cdot\|_{l,K}\colon K\subset\overline{X}\text{ compact }, l\in\N_0\}$ defined by
	\[\forall\,f\in C^\infty(\overline{X})\colon \|f\|_{l,K}=\max\{|\partial^\alpha f(x)|\colon x\in K, |\alpha|\leq l\},\]
	turning $C^\infty(\overline{X})$ into a Fr\'echet space. Consequently,
	\[C^\infty_P(\overline{X}):=\{f=(f_1,\ldots,f_N)^T\in C^\infty(\overline{X})^N\colon P(D)f=0 \text{ in }X\}\]
	is a closed subspace of $C^\infty(\overline{X})^N$, hence a Fr\'echet space. We point out that the topology of  $C^\infty(\overline{X})^N$, and thus also the topology of $C^\infty_P(\overline{X})$, is the topology of uniform convergence of partial derivatives of arbitrary order on compact subsets of $\overline{X}$. In particular, for bounded $X$, the topology is the one of uniform convergence of partial derivatives of arbitrary order on $\overline{X}$.
	
	Under the additional hypothesis that $\overline{X}$ is path connected and satisfies the strong regularity condition, see below, $\mathscr{E}(\overline{X})^N$ and $C^\infty(\overline{X})^N$, respectively $\mathscr{E}_P(\overline{X})$ and $C^\infty_P(\overline{X})$, are isomorphic as Fr\'echet spaces, the isomorphism being given by
	\[(f^{(\alpha)})_{\alpha\in\N_0^d}\mapsto f^{(0)}.\]
	This is a straightforward modification of \cite[Corollary 11]{CiasKalmes2025} where the case $N=1$ is considered. Recall that a path connected closed set $F\subset\R^d$  satisfies the \emph{strong regularity condition} if $F=\overline{\operatorname{int}F}$ and if for all $x_0\in C$ there are $\epsilon>0, M>0$, and $\theta\in (0,1]$ such that any two points $x,y\in F\cap B(x_0,\epsilon)$ can be joint by a rectifiable arc $\gamma\subset F$ with length $|\gamma|$ such that
	\begin{enumerate}
		\item $\gamma$ meets the boundary of $F$ at most finitely many times,
		\item $|\gamma|\leq M|x-y|^\theta$.
	\end{enumerate}
	It can be shown that $\overline{X}$ satisfies the strong regularity condition whenever $X$ has a H\"older continuous boundary, \cite[Appendix]{CiasKalmes2025}. 
\end{remark}

We recall some abstract criterion that is a direct consequence of \cite[Proposition 13.1 and Corollary 3 on p.~54]{{Treves1967}}.

\begin{proposition}\label{abstract-criterion}
	Let $E_1,E_2$ be Fr\'echet spaces and let $T_1\colon E_1\to E_1$, $T_2\colon E_2\to E_2$ and $r\colon E_2\to E_1$ be continuous linear maps. Assume moreover that $T_2$ is surjective and $r\circ T_2=T_1\circ r$. Then the following assertions are equivalent:
	\begin{enumerate}[{\upshape(i)}]
		\item $T_1$ is surjective and $r(\ker T_2)$ is dense in $\ker T_1$;
		\item $r$ has dense range and for all $y\in E_2'$ it holds $y\in\operatorname{im}{}^t r$ whenever ${}^tT_2(y)\in\operatorname{im}{}^t r$.
	\end{enumerate}
\end{proposition}

As a consequence of the previous criterion, we have the following result which is proved similarly as \cite[Corollary 9]{CiasKalmes2025}. We include its proof for the reader's convenience.

\begin{proposition}\label{prop:necessary for P-Runge for Whitney}
    Let $F_1\subset F_2$ be closed subsets of $\R^d$ which are a $P$-Runge pair for smooth Whitney jets for some $P$ with $\detP\not\equiv 0$. Then, $F_2$ does not contain a bounded connected component of $\R^d\backslash F_1$. 
\end{proposition}

\begin{proof}
    It is  well-known fact, that for a closed subset $F$ of $\R^d$, the dual space of $\mathscr{E}(F)$ coincides with the subspace $\mathscr{E}'(F)$ of the space of compactly supported distributions on $\R^d$, $\mathscr{E}'(\R^d)$, which have their support contained in $F$ (see e.g.~\cite[Proposition 4]{CiasKalmes2025}).
    Moreover, let us note that the transpose ${}^tP(D)$ of $P(D)\colon \mathscr{E}(F)^N\to \mathscr{E}(F)^N$ is the operator
    \[{}^tP(D)\colon \mathscr{E}'(F)^N\to \mathscr{E}'(F)^N,\quad{}^tP(D)=(P_{kj}(-D))_{1\leq j,k\leq N}=\check P^T(D).\]
    Since $\detP^T=\detP\not\equiv 0$, there is $\zeta\in\C^d$ with $\detP^T(-\zeta)=0$. We set $h_0(x)=\exp\left(i\sum_{j=1}^d\zeta_j x_i\right)$, $x\in\R^d$, as well as
    \[h:=\left(\check{P}^T\right)^{ad}(D)(h_0,\ldots,h_0)^T\in\mathscr{E}(\R^d)^N\]
    so that $\check{P}^T(D)h=\det\check{P}^T(D)h=0$.
    Assume that $G$ is a bounded (open) connected component of $\R^d\backslash F_1$ which is contained in $F_2$. Then, it is straight forward to prove that
    \[g:\R^d\rightarrow\C, x\mapsto \begin{cases}
        h(x),&\text{if }x\in G,\\ 0, &\text{otherwise}
    \end{cases}\]
    satisfies $g\in\mathscr{E}'(F_2)^N$ and
    \[\supp\check{P}^T(D)g\subset\partial G\subset F_1,\]
    i.e.~$\check{P}^T(D)g\in\mathscr{E}(F_1)^N$. But, since $F_1, F_2$ is a $P$-Runge pair for smooth Whitney jets, Proposition \ref{abstract-criterion} applied to $E_j=\mathscr{E}(F_j)^N$, $T_1=T_2=P(D)$, and $r$ the restriction mapping from $\mathscr{E}(F_2)^N$ to $\mathscr{E}(F_1)^N$, gives that every $u\in \mathscr{E}'(F_2)^N$ with $\check P^T(D)u\in \mathscr{E}'(F_1)^N$ satisfies $u\in\mathscr{E}'(F_1)^N$ which gives the desired contradiction. 
\end{proof}

Recall that a square matrix $P=(P_{jk})_{1\leq j,k\leq N}$ of polynomials (and the corresponding square matrix of partial differential operators) is \emph{elliptic} if the polynomial ${\det P}$ is elliptic. The following theorem is the analogue for the setting of smooth Whitney jets of \cite[\enquote{Cas d'un syst\'eme pour lequel $p=q$}, p.~337]{Malgrange1955} which considers $P$-Runge pairs for elliptic square systems.

\begin{theorem}\label{th:lax-malgrange for system of pde whitney}
	Let $F_1\subset F_2$ be closed subsets of $\R^d$ and let $P(D)$ be an elliptic square matrix of partial differential operators. Then, the following conditions are equivalent.
\begin{enumerate}[\upshape{(i)}] 
		\item[\upshape{(i)}] $F_1, F_2$ is a $P$-Runge pair for smooth Whitney jets.
		\item[\upshape{(ii)}] No bounded connected component of $\R^d\setminus F_1$ is contained in $F_2$.
	\end{enumerate}
\end{theorem}

\begin{proof}
    By \cite[Theorem 12]{CiasKalmes2025} and Theorem \ref{theo:Runge for systems}, condition (ii) implies (i). Conversely, (i) implies (ii) by Proposition \ref{prop:necessary for P-Runge for Whitney}.
\end{proof}

\subsection{The \texorpdfstring{$(\lambda-\curl)$}{lambda-curl}-system}\label{subsec:curl}

In this subsection, we have $d=N=3$. Recall that a \emph{Beltrami field} on a open subset $X$ of $\R^3$ is a vector field $u$ which is an eigenfunction to a non-zero eigenvalue of the curl operator, i.e.~which satisfies
\[\curl u = \lambda u\]
for $\lambda\in\C\backslash \{0\}$. For such $\lambda$ we consider the $(\lambda-\curl)$-system, i.e.~the system given by the matrix
\begin{equation}\label{eq:curl matrix}
	P_\lambda(D)=\begin{pmatrix}
			\lambda & \partial_3 & -\partial_2\\
			-\partial_3 & \lambda & \partial_1\\
			\partial_2 & -\partial_1 & \lambda
		\end{pmatrix},
\end{equation}
where $\partial_j=\frac{\partial}{\partial x_j}$, $j=1,2,3$. A straightforward calculation gives
\begin{equation*}
	\Pad_\lambda(D)=\begin{pmatrix}
		\lambda^2+\partial_1^2 & \partial_1\partial_2 - \lambda\partial_3 & \partial_1\partial_3+\lambda\partial_2\\
		\partial_1\partial_2+\lambda\partial_3 & \lambda^2+\partial_2^2& \partial_2\partial_3 - \lambda\partial_1\\
		\partial_1\partial_3-\lambda\partial_2 & \partial_2\partial_3 + \lambda\partial_1 & \lambda^2+\partial_3^2
	\end{pmatrix}
\end{equation*}
and
\begin{equation}\label{eq:det curl}
	\detP_\lambda(D)=\lambda(\lambda^2+\Delta).
\end{equation}
Since $\detP_\lambda$ is (hypo)elliptic, by Remark \ref{rem:dense subset} (ii), for every open subset $X$ of $\R^3$ we have  $\mathscr{D}'_{P_\lambda}(X)=\mathscr{E}_{P_\lambda}(X)$ and $\mathscr{D}'_{\Pad_\lambda}(X)=\mathscr{E}_{\Pad_\lambda}(X)$ as locally convex spaces. In particular, the set of all Beltrami fields on $X$ equals $\bigcup_{\lambda\in\C\backslash\{0\}}\mathscr{E}_{P_\lambda}(X)$. 

Therefore, by the classical Lax-Malgrange Theorem for elliptic square systems (cf.~\cite[\enquote{Cas d'un syst\'eme pour lequel $p=q$}, p.~337]{Malgrange1955}), open subset $X_1\subset X_2$ of $\R^3$ are a $P_\lambda$-Runge pair (for distributions) if and only if $X_2$ does not contain a bounded connected component of $\R^3\backslash X_1$. The following version of this result for smooth Whitney jets is an immediate consequence of Theorem \ref{th:lax-malgrange for system of pde whitney}.

\begin{theorem}\label{theo:Runge for curl}
	Let $F_1\subset F_2$ be closed subsets of $\R^3$. Moreover, for $\lambda\in\C\backslash\{0\}$, let $P_\lambda(D)$ be as in \eqref{eq:curl matrix}. Then, $F_1,F_2$ is a $P_\lambda$-Runge pair for smooth Whitney jets if and only if $F_2$ does not contain a bounded connected component of $\R^3\backslash F_1$.
\end{theorem}

\begin{remark}\label{rem:Bessel and spherical harmonics}
	It is well known that every solution $u$ of the Helmholtz equation with (complex) wave number $\lambda^2$, i.e.~every $u$ satisfying $\Delta u +\lambda^2u=0$ in $\R^d$, $d\geq 2$, can be written as
	\begin{equation}\label{eq:Helmholtz expansion}
		u(x)=\sum_{k=0}^\infty\sum_{m=1}^{d_k}c_{k,m}|x|^{(2-d)/2}J_{k+(d-2)/2}(\lambda|x|)Y_{k,m}(x/|x|),
	\end{equation}
	where the series converges in $\mathscr{E}(\R^d)$ and where $c_{k,m}\in\C$, $J_\mu$ denotes the Bessel function of the first kind of order $\mu$ and where $\{Y_{k,m}\colon 1\leq m\leq d_k\}$ is an orthonormal basis of the finite dimensional vector space of spherical harmonics of degree $k$, i.e.~the subspace of (restrictions to $S^{d-1}$ of) homogeneous harmonic polynomials of degree $k$ in $L^2(S^{d-1})$ (see \cite[the discussion preceding Theorem 2.2.66]{Folland1995} or \cite[Section V.5.5 and \S VII.2]{CourantHilbert1953}; see also Remark \ref{rem:ODEs for coefficients in spherical harmonics expansion} below.)

	Since the summands in the series \eqref{eq:Helmholtz expansion} satisfy the Helmholtz equation with wave number $\lambda^2$, it follows that the linear span $E_0$ of the set of functions
	\[\{|x|^{(2-d)/2}J_{k+(d-2)/2}(\lambda |x|)Y_{k,m}(x/|x|)\colon k\in\N_0, m=1,\ldots,d_k\},\]
	for $d=3$, is dense in $\mathscr{E}_{\detP_\lambda}(\R^3)$.
	
	Using the well-known recursion relation for Bessel functions
	\[J'_\mu(z)=\frac{1}{2}(J_{\mu-1}(z)-J_{\mu+1}(z)), z\in\C, \mu\in\R,\]
	a straightforward induction on the order of differentiation shows (see Proposition \ref{prop:elementary} below which we include for the reader's convenience) that for each $\alpha\in\N_0^d$ there are functions $\{g_{\alpha,j,l}\colon -|\alpha|\leq j\leq |\alpha|, 0\leq l\leq \min\{1,|\alpha|\}\}\subset \mathscr{E}(\R^d\backslash\{0\})$ such that
	\begin{eqnarray*}
			&&\forall\,x\in\R^d\backslash\{0\}\colon \partial^\alpha_x\left(|x|^{(2-d)/2}J_{(k-1)/2}(\lambda |x|)Y_{k,m}(x/|x|)\right)\\
			&=&\sum_{j=-|\alpha|}^{|\alpha|}\sum_{l=0}^{\min\{1,|\alpha|\}} |x|^{-|\alpha|-l+(2-d)/2}J_{(k-1)/2-j}(\lambda|x|)g_{\alpha,j,l}\left(\frac{x}{|x|}\right)\\
			&=&|x|^{-|\alpha|+(1-d)/2} \sum_{j=-|\alpha|}^{|\alpha|}\sum_{l=0}^{\min\{1,|\alpha|\}} |x|^{-l} |x|^{1/2} J_{(k-1)/2-j}(\lambda|x|)g_{\alpha,j,l}\left(\frac{x}{|x|}\right).
	\end{eqnarray*}
	For real $\lambda\neq 0$ it follows from 
	\[\sup_{r>0} r^{1/2}|J_{n+1/2}(r)|<\infty\]
	(see \cite[\S 7.21]{Watson1944}) that the right hand side of the above equation is bounded by $|x|^{-|\alpha|+(1-d)/2} C_\alpha$ for a suitable constant $C_\alpha>0$. In particular, for $d=3$ and real $\lambda\neq 0$, for every function $f$ from the dense subspace  $\Pad_\lambda(D)(E_0^3)$ of $\mathscr{E}_{P_\lambda(D)}(\R^3)$, the following decay property at infinity holds for the components of $f=(f_1,f_2,f_3)$
	\begin{equation}\label{eq:growth bound for curl}
		\forall\alpha\in\N_0^3\,\exists C_\alpha>0\forall x\in\R^3\colon |\partial^\alpha f_j(x)|\leq C_\alpha |x|^{-|\alpha|-1}, j=1,2,3.
	\end{equation}
	Combining Remark \ref{rem:dense subset} (i) with the classical Lax-Malgrange Theorem for elliptic square systems (cf.~\cite[\enquote{Cas d'un syst\'eme pour lequel $p=q$}, p.~337]{Malgrange1955}) and with Theorem \ref{theo:Runge for curl} (for $F_2=\R^3$), respectively, implies the next theorem. Part (i) is due to Enciso and Peralta-Salas \cite[Theorem 8.3]{EnPe15}.
\end{remark}

\begin{theorem}\label{th:EnPe}
	Let $\lambda\in\R\backslash\{0\}$ and let $P_\lambda(D)$ be as in \eqref{eq:curl matrix}. Moreover, let $X_1\subset\R^3$ be open and $F_1\subset\R^3$ be closed.
	\begin{itemize}
		\item[(i)] The subspace of functions of $\mathscr{E}_{P_\lambda(D)}(\R^3)$ whose components satisfy the decay property \eqref{eq:growth bound for curl} is dense in $\mathscr{E}_{P_\lambda(D)}(X_1)$ if and only if $\R^3\backslash X_1$ does not have a bounded connected component.
		\item[(ii)] The subspace of functions of $\mathscr{E}_{P_\lambda(D)}(\R^3)$ whose components satisfy the decay property \eqref{eq:growth bound for curl} is dense in $\mathscr{E}_{P_\lambda(D)}(F_1)$ if and only if $\R^3\backslash F_1$ does not have a bounded connected component.
	\end{itemize} 
\end{theorem}

\begin{remark}\label{rem:Peralta-Salas for path connected sets with Hoelder continuous boundary}
    We point out that for a path connected, open set $X$ in $\R^3$ with H\"older continuous boundary, Theorem \ref{th:EnPe} (ii) above combined with Remark \ref{rem:Hoelder continuous boundary} yields that the subspace of functions of $\mathscr{E}_{P_\lambda(D)}(\R^3)$ whose components satisfy the decay property \eqref{eq:growth bound for curl} is dense in $C^\infty_{P_{\lambda}(D)}(\overline{X})$. 
\end{remark}

\begin{proposition}\label{prop:elementary}
	Let $\mu\in\R$, $\lambda\in\C$, and $g\in \mathscr{E}(\R^d\backslash\{0\})$. For the smooth function
	\[f\colon\R^d\backslash\{0\}\rightarrow\C, x\mapsto |x|^{(2-d)/2}J_\mu(\lambda|x|) g\left(\frac{x}{|x|}\right)\]
	the following holds. For every $\alpha\in\N_0^d$ there are $\{g_{\alpha,j,l}\colon -|\alpha|\leq j\leq |\alpha|, 0\leq l\leq \min\{1,|\alpha|\}\}\subset \mathscr{E}(\R^d\backslash\{0\})$ such that
	\[\forall x\in\R^d\backslash\{0\}\colon \partial^\alpha f(x)=\sum_{j=-|\alpha|}^{|\alpha|}\sum_{l=0}^{\min\{1,|\alpha|\}} |x|^{-|\alpha|-l}|x|^{(2-d)/2}J_{\mu-j}(\lambda|x|)g_{\alpha,j,l}\left(\frac{x}{|x|}\right).\]
\end{proposition}

\begin{proof}
	Let us denote
	\[r:\R^d\backslash\{0\}\rightarrow\R, x\mapsto |x|\]
	and
	\[\omega:\R^d\backslash\{0\}\rightarrow\R^d, x\mapsto\frac{x}{|x|}.\]
	Then, for $1\leq i,k\leq d$ it holds
	\[\partial_i r(x)=\frac{x_i}{r(x)}\text{ and }\partial_i\omega_k=\frac{1}{r(x)}\left(\delta_{i,k}-\omega_i(x)\omega_k(x)\right).\]
	We prove the claim by induction on $|\alpha|$.
	
	For $|\alpha|=0$ there is nothing to prove. Assume that the claim holds for each $\alpha\in \N_0^d$ with $|\alpha|\leq k$.  Let $\alpha\in\N_0^d$ with $|\alpha|=k$ and let $1\leq i\leq d$. By hypothesis, there are $\{g_{\alpha,j,l}\colon -|\alpha|\leq j\leq |\alpha|, 0\leq l\leq \min\{1,|\alpha|\}\}\subset \mathscr{E}(\R^d\backslash\{0\})$ such that
	\begin{equation}\label{eq:induction hypothesis}
		\forall x\in\R^d\backslash\{0\}\colon \partial^\alpha f(x)=\sum_{j=-|\alpha|}^{|\alpha|}\sum_{l=0}^{\min\{1,|\alpha|\}} |x|^{-|\alpha|-l+(2-d)/2}J_{\mu-j}(\lambda|x|)g_{\alpha,j,l}\left(\frac{x}{|x|}\right).
	\end{equation}
	Using the recursion relation for Bessel functions 
	\[J'_\nu(z)=\frac{1}{2}(J_{\nu-1}(z)-J_{\nu+1}(z)), z\in\C, \nu\in\R,\]
	we calculate
	\begin{eqnarray*}
		&&\partial_i\left(|x|^{-|\alpha|-l+(2-d)/2}J_{\mu-j}(\lambda |x|)g_{\alpha,j,l}\left(\frac{x}{|x|}\right)\right)\\
		&=&((2-d)/2-|\alpha|-l)|x|^{-|\alpha+e_i|-l+(2-d)/2}\,\partial_i r(x) J_{\mu-j}(\lambda|x|)g_{\alpha,j,l}\left(\frac{x}{|x|}\right)\\
		&&+|x|^{-|\alpha|-l+(2-d)/2}\,\frac{\lambda}{2}\left(J_{\mu-j-1}(\lambda|x|)-J_{\mu-j+1}(\lambda|x|)\right)\partial_i r(x) g_{\alpha,j,l}\left(\frac{x}{|x|}\right)\\
		&&+|x|^{-|\alpha|-l+(2-d)/2}\,J_{\mu-j}(\lambda|x|)\sum_{n=1}^d\partial_n g_{\alpha,j,l}\left(\frac{x}{|x|}\right)\partial_i\omega_n(x)\\
		&=&((2-d)/2-|\alpha|-l)|x|^{-|\alpha+e_i|-l+(2-d)/2}\, J_{\mu-j}(\lambda|x|) \frac{x_i}{|x|}g_{\alpha,j,l}\left(\frac{x}{|x|}\right)\\
		&&+|x|^{-|\alpha|-l+(2-d)/2}\,\frac{\lambda}{2}\left(J_{\mu-j-1}(\lambda|x|)-J_{\mu-j+1}(\lambda|x|)\right)\partial_i r(x) g_{\alpha,j,l}\left(\frac{x}{|x|}\right)\\
		&&+|x|^{-|\alpha+e_i|-l+(2-d)/2}\,J_{\mu-j}(\lambda|x|)\sum_{n=1}^d\partial_n g_{\alpha,j,l}\left(\frac{x}{|x|}\right)\left(\delta_{i,n}-\frac{x_i}{|x|}\frac{x_n}{|x|}\right).\\ 
	\end{eqnarray*}
	Now, in case $l=0$, it holds
	\[|x|^{-|\alpha|-l+(2-d)/2}=|x|^{-|\alpha|+(2-d)/2}=|x|^{-|\alpha+e_i|-1+(2-d)/2}\]
	while in case $l=1$, we have
	\[|x|^{-|\alpha|-l+(2-d)/2}=|x|^{-|\alpha|-1+(2-d)/2}=|x|^{-|\alpha+e_i|+(2-d)/2}.\]
	Thus, the $\partial_i$ derivative of each summand in the right hand side of \eqref{eq:induction hypothesis} can be written as
	\[\sum_{s\in\{-1,0,1\}}\sum_{l=0}^{\min\{1,|\alpha|\}}|x|^{-|\alpha+e_i|-l+(2-d)/2}J_{\mu-j-s}\tilde{g}_{\alpha+e_i,j+s,l}\left(\frac{x}{|x|}\right)\]
	which finishes the induction step.
\end{proof}

\subsection{The 3D unsteady Stokes system}\label{subsec:Unsteady Stokes}

In this subsection let $d=N=4$. Instead of $x=(x_1,x_2,x_3,x_4)\in\R^4$ we write $(x_0,x_1,x_2,x_3)\in\R^4$, we denote $x_0$ by $t$ and $(x_1,x_2,x_3)$ by $x$. Moreover, we abbreviate $\Delta_x=\frac{\partial^2}{\partial x_1^2}+ \frac{\partial^2}{\partial x_2^2}+\frac{\partial^2}{\partial x_3^2}$, $\partial_j=\frac{\partial}{\partial x_j}$, $j=1,2,3$, $\nabla_x=(\partial_1,\partial_2,\partial_3)^T$, as well as $\partial_t=\frac{\partial}{\partial t}(=\frac{\partial}{\partial x_0})$. We consider the 3D unsteady Stokes system, i.e.~the system
\begin{eqnarray*}
	\left(\partial_t-\Delta_x\right)u_j + \partial_j u_4&=&0,\quad j=1,2,3,\\
	\sum_{k=1}^3\partial_k u_k&=&0,
\end{eqnarray*}
which, with the usual notation $u=(u_1,u_2,u_3)^T$ and $p=u_4$, reads
\begin{equation}\label{eq:Stokes system}
	\begin{cases}
 		\left(\partial_t-\Delta_x\right)u + \nabla_x p=0,&\\
 		\nabla_x\cdot u=0.&
	\end{cases}
\end{equation}
Thus, we have
\begin{equation}\label{eq:Stokes matrix}
	P(D)=\begin{pmatrix}
	\partial_t-\Delta_x & 0 & 0 & \partial_1\\
	0 & \partial_t-\Delta_x & 0 & \partial_2\\
	0 & 0 & \partial_t-\Delta_x & \partial_3\\
	\partial_1 & \partial_2 & \partial_3 & 0
\end{pmatrix}
\end{equation}
which implies
\[\Pad(D)=(\partial_t-\Delta_x)\begin{pmatrix}
	-(\partial_2^2+\partial_3^2) & \partial_1\partial_2 & \partial_1\partial_3 & (\Delta_x-\partial_t)\partial_1 \\ 
	\partial_2\partial_1 &-(\partial_1^2+\partial_3^2) &  \partial_2\partial_3 & (\Delta_x-\partial_t)\partial_2\\
 	\partial_3\partial_1 &  \partial_3\partial_2 &-(\partial_1^2+\partial_2^2) & (\Delta_x-\partial_t)\partial_3\\
 	(\Delta_x-\partial_t)\partial_1 & (\Delta_x-\partial_t)\partial_2  & (\Delta_x-\partial_t)\partial_3 & (\partial_t-\Delta_x)^2
\end{pmatrix}\]
and
\begin{equation}\label{eq:det Stokes}
	\detP(D)=-(\partial_t-\Delta_x)^2\Delta_x.
\end{equation}
 
\begin{proposition}\label{prop:surjectivity for stokes}
	Let $X\subset\R^4$ be open and let $\detP(D)$ be as in \eqref{eq:det Stokes}. Then, the following are equivalent.
	\begin{itemize}
		\item[(i)] $\detP(D)$ is surjective on $\mathscr{E}(X)$.
		\item[(ii)] $\detP(D)$ is surjective on $\mathscr{D}'(X)$.
		\item[(iii)] For every $t\in\R$, the boundary distance $d_X$ satisfies the minimum principle in the hyperplane $\{t\}\times\R^3$.
	\end{itemize}
\end{proposition}
 
\begin{proof}
	Clearly, $\detP(D)$ is surjective on $\mathscr{E}(X)$ or on $\mathscr{D}'(X)$, respectively, precisely when this is true for both $\partial_t-\Delta_x$ and $\Delta_x$. Because $\partial_t-\Delta_x$ is hypoelliptic (\cite[Theorem 11.1.11]{Hor2}), $\partial_t-\Delta_x$ is surjective on $\mathscr{D}'(X)$ if and only if it is surjective on $\mathscr{E}(X)$ (\cite[Corollary 10.7.10 and Theorem 11.1.1]{Hor2}). By Corollary \ref{cor:characterization P-convexity}, $\partial_t-\Delta_x$ is surjective on $\mathscr{E}(X)$ if and only if (iii) holds. Moreover, by \cite[Theorem 9]{Kalmes19}, $\Delta_x$ is surjective on $\mathscr{E}(X)$ if and only if it is surjective on $\mathscr{D}'(X)$ which holds precisely when (iii) is true. This proves the proposition.
\end{proof}
 
 \begin{remark}
	For $\detP(D)$ as in \eqref{eq:det Stokes}, it follows from Proposition \ref{prop:surjectivity for stokes} (with \cite[Lemma 4]{Kalmes19}) that $\detP(D)$ is surjective on $\mathscr{E}\left(I\times \Omega\right)$ and $\mathscr{D}'\left(I\times \Omega\right)$ for every open sets $\Omega\subset\R^3, I\subset\R$.
\end{remark}
 
\begin{theorem}\label{theo:Runge pairs for det-Stokes}
	Let $\detP(D)$ be as in \eqref{eq:det Stokes}.
	\begin{itemize}
		\item[(a)] Let $X_1\subset X_2$ be open subsets of $\R^4$. Moreover, assume that for every $t\in\R$ the boundary distance $d_{X_2}$ satisfies the minimum principle in the hyperplane $\{t\}\times \R^3$. Then, the following are equivalent. 
		\begin{itemize}
 			\item[(a-i)] For every $t\in\R$, the boundary distance $d_{X_1}$ satisfies the minimum principle in the hyperplane $\{t\}\times \R^3$ and $X_1, X_2$ is a $\detP$-Runge pair.
 			\item[(a-ii)] For every $t\in\R$, the boundary distance $d_{X_1}$ satisfies the minimum principle in the hyperplane $\{t\}\times \R^3$ and $X_1, X_2$ is a $\detP$-Runge pair for distributions.
 			\item[(a-iii)] There is no $t\in\R$ such that $X_2$ contains a bounded connected component of $\left(\{t\}\times\R^3\right)\cap \left(\R^4\backslash X_1\right)$.
 		\end{itemize}
 		\item[(b)] Let $F_1\subset F_2$ be closed subsets of $\R^4$. Consider the following conditions.
 		\begin{itemize}
 			\item[(b-i)] There is no $t\in\R$ for which $F_2$ contains a bounded connected component of $\left(\{t\}\times\R^3\right) \cap \left(\R^4\backslash F_1\right)$.
 			\item[(b-ii)] $F_1,F_2$ is a $\detP$-Runge pair for smooth Whitney jets.
 			\item[(b-iii)] There are no $t\in\R$, $\epsilon>0$ such that $F_2$ contains a bounded connected component of $\left([t-\epsilon,t+\epsilon]\times\R^3\right)\cap\left(\R^4\backslash F_1\right)$.
 		\end{itemize}
 		Then {\rm (b-i)} implies {\rm (b-ii)} and {\rm (b-ii)} implies {\rm (b-iii)}.
 		\item[(c)] Let $F_1\subset \R^4$ be closed with $C^1$-boundary such that the normal space of every $\xi\in\partial F_1$ is not spanned by $(1,0,0,0)^T$. Then, the following are equivalent.
 		\begin{itemize}
 			\item[(c-i)] $F_1,\R^4$ is a $\detP$-Runge pair for smooth Whitney jets.
 			\item[(c-ii)] There is not $t\in\R$ such that $\left(\{t\}\times\R^3\right)\cap\left(\R^4\backslash F_1\right)$ has a bounded connected component.
 		\end{itemize}
 	\end{itemize}
 \end{theorem}
 
\begin{proof}
	We first prove part (a). By Proposition \ref{prop:surjectivity for stokes}, $\detP(D)$ is surjective on $\mathscr{E}(X_2)$. The equivalence of (a-i) and (a-ii) follows from Proposition \ref{prop:surjectivity for stokes} applied to $X_1$ combined with \cite[Theorem 6]{Kalmes2021}. We point out that the $t$ variable plays the role of the $x_1$ variable in \cite{Kalmes2021}. Moreover, by \cite[Theorem Theorem 1]{Kalmes2021}, (a-ii) implies (a-iii). Finally, assume that (a-iii) holds. It follows from \cite[Lemma 3.2]{DebKalmes2024} that $d_{X_1}$ satisfies the minimum principle in $\{t\}\times\R^3$ for each $t\in\R$ so that $\detP(D)$ is surjective on $\mathscr{E}(X_1)$ by Proposition \ref{prop:surjectivity for stokes}. Another application of \cite[Theorem 1]{Kalmes2021} yields (a-i). 
 	
	By \cite[Theorem 16]{CiasKalmes2025}, (b-i) implies (b-ii). That (b-ii) implies (b-iii) is an immediate consequence of \cite[Corollary 8]{CiasKalmes2025} and Theorem \ref{thm: necessary condition for P-Runge for partially elliptic} applied to the factor $\partial_t-\Delta_x$ of $\detP(D)$.
	
	Part (c) is a direct consequence of \cite[Theorem B]{CiasKalmes2025} where the role of the variable $x_1$ there is played by the $t$ variable in the current context.
\end{proof}

Combining the previous theorem with Theorem \ref{theo:Runge for systems} one can easily derive Runge approximation results for the 3D unsteady Stokes system \eqref{eq:Stokes system}. We only state a global Runge result explicitly. Part (a) of the following theorem should be compared with a recent result due to Higaki and Sueur \cite{HiSu25}, where it was proved that approximation of smooth local solutions in the $L^\infty$-norm on a compact subset is possible by global solutions whose velocity field grows at most exponentially at spatial infinity while its pressure component grows polynomially.

\begin{corollary}\label{cor:global Runge for 3D unsteady Stokes}
	Let $P(D)$ be as in \eqref{eq:Stokes matrix} and let $X_1\subset\R^4$ be open, $F_1\subset \R^4$ be closed. Moreover, let
	\[\mathcal{P}=\left\{u=(u_1,u_2,u_3,u_4)^T\in\mathscr{E}_P(\R^4)\colon u_j\mbox{ polynomial}, 1\leq j\leq 4\right\}.\]
	\begin{itemize}
		\item[(a)] Assume that there is no $t\in\R$ such that $(\{t\}\times\R^3)\cap (\R^4\backslash X_1)$ has a compact connected component. Then, the restriction mapping
		\[\mathcal{P}\rightarrow \mathscr{E}_P(X_1), u\mapsto u_{|X_1}\]
		has dense range.
		\item[(b)] Assume that there is no $t\in\R$ such that $(\{t\}\times\R^3)\cap (\R^4\backslash F_1)$ has a bounded connected component. Then, the restriction mapping
		\[\mathcal{P}\rightarrow \mathscr{E}_P(F_1), u\mapsto \left(\partial^\alpha u_{|F_1}\right)_{\alpha\in\N_0^4}\]
		has dense range.
	\end{itemize}
\end{corollary}

\begin{proof}
	Let us prove (a). By Theorem \ref{theo:Runge pairs for det-Stokes} (a), $X_1, \R^4$ is a $\detP$-Runge pair, and by Proposition \ref{prop:surjectivity for stokes} $\detP(D)$ is surjective on $\mathscr{E}(X_1)$. The latter implies the surjectivity of $\Pad(D)$ on $\mathscr{E}(X_1)^4$ so that the claim follows from Corollary \ref{cor:density of exponential solutions} (b-i). 
	
	(b) follows from Theorem \ref{theo:Runge pairs for det-Stokes} (b) and Corollary \ref{cor:density of exponential solutions} (b-ii).
\end{proof}

\begin{remark}\label{rem:Stokes for path connected sets with Hoelder continuous boundary}
    Again, we point out that for a path connected, open set $X$ in $\R^4$ with H\"older continuous boundary, Corollary \ref{cor:global Runge for 3D unsteady Stokes} (b) above combined with Remark \ref{rem:Hoelder continuous boundary} yields that the subspace $\mathcal{P}$ of $\mathscr{E}_P(\R^4)$ is dense in $C^\infty_{P}(\overline{X})$. 
\end{remark}

\appendix
\renewcommand{\thesection}{\arabic{section}}
\section{Characterization of some class of polynomials}\label{appendix: characterization of some polynomials}

In this section we denote by $(p_1,\ldots,p_k)$ the ideal in $\C[\xi_1,\ldots,\xi_d]$ generated by polynomials $p_1,\ldots,p_k$.
The following lemma is a standard consequence of basic algebraic geometry; we include a proof for completeness. 
\begin{lemma}\label{lem:hilbert nullstellensatz}
Let $P\in\C[\xi_1,\ldots,\xi_d]$ and let $1\leq l\leq d-1$. Then $P$ vanishes on the set 
\[\{\xi\in\R^d\colon \xi_{l+1}=\ldots=\xi_d=0\},\]
if and only if $P\in (\xi_{l+1},\ldots,\xi_d)$.

\end{lemma}

\begin{proof}
Suppose that $P$ vanishes on $\{\xi\in\R^d\colon \xi_{l+1}=\ldots=\xi_d=0\}$. 
Define the polynomial $Q\in\C[\xi_1,\ldots,\xi_l]$ by
\[
Q(\xi_1,\ldots,\xi_l):=P(\xi_1,\ldots,\xi_l,0,\ldots,0).
\]
Then $Q$ vanishes on $\R^l$, hence $Q\equiv 0$. It follows that $P$ vanishes on 
\[\{\xi\in\C^d\colon \xi_{l+1}=\ldots=\xi_d=0\}.\]
	
For a subset $A\subset \C[\xi_1,\ldots,\xi_d]$, let
\[
Z(A):=\{\xi\in\C^d\colon P(\xi)=0 \text{ for all } P\in A\}.
\]
Then, by Hilbert's Nullstellensatz, any polynomial $P$ vanishing on
\[Z((\xi_{l+1},\ldots,\xi_d))=\{\xi\in\C^d\colon \xi_{l+1}=\ldots=\xi_d=0\}\]
belongs to the radical $\sqrt{(\xi_{l+1},\ldots,\xi_d)}$, i.e.~$P^r\in (\xi_{l+1},\ldots,\xi_d)$ for some $r\in\N$.
Moreover,
\[\C[\xi_1,\ldots,\xi_d]/(\xi_{l+1},\ldots,\xi_d)\cong \C[\xi_1,\ldots,\xi_l]\]
is an integral domain, hence the ideal $(\xi_{l+1},\ldots,\xi_d)$ is prime, and therefore radical. Thus $P\in (\xi_{l+1},\ldots,\xi_d)$.
	
Conversely, if $P\in (\xi_{l+1},\ldots,\xi_d)$, then
\[
P(\xi)=\sum_{j=l+1}^d \xi_j Q_j(\xi)
\]
for some $Q_j\in\C[\xi_1,\ldots,\xi_d]$, so $P(\xi)=0$ whenever $\xi_{l+1}=\ldots=\xi_d=0$.
\end{proof}

\begin{remark}\label{rem:homogeneous polynomials}
It is an easily verifiable fact that for every $m$-homogeneous polynomial $Q(\xi)=\sum_{|\alpha|=m}q_\alpha\xi^\alpha$ the following assertions are equivalent for $1\leq j\leq d$
		\begin{enumerate}
			\item[(i-a)] $Q(e_j)=0$,
			\item[(i-b)] $q_{me_j}=0$,
			\item[(i-c)] there is no $c\neq 0$ such that $c\xi_j^m$ appears in the Taylor expansion of $Q$.
		\end{enumerate}
\end{remark}

We can now prove Lemma \ref{lem:nice polynomial2} from Section \ref{section:zero solutions}. For the reader's convenience, we repeat the lemma here again.

\begin{crucial lemma}
	Let $P\in\C[\xi_1,\ldots,\xi_d]$ be a non-constant polynomial of degree $m$ with principal part $P_m$. Moreover, let $Z$ be a subspace of $\R^d$ of dimension $l\geq 1$ and let $w\in \R^d\backslash Z$. Then the following statements are equivalent.
    \begin{enumerate}
        \item[\upshape{(i)}] 
	    $P_m(w)\neq0$ and $P_m$ together with the directional derivatives $D^k_w P_m$, $1\leq k\leq m-1$, vanish on $Z$. 
	   \item[\upshape{(ii)}] There is a linear isomorphism $A:\R^d\rightarrow\R^d$,  $c\in\C\backslash\{0\}$, and polynomials $p_{l+1},\ldots,p_{d-1}\in\C[\xi_1,\ldots,\xi_d]$ such that $A^{-1}w=e_d$,
	\[A^{-1}(Z)=\{\xi\in\R^d\colon \xi_{l+1}=\ldots=\xi_d=0\},\]
    and the principal part $(P\circ A)_m$ of $P\circ A$ is given by
	\[(P\circ A)_m(\xi)=c\xi_d^m + \sum_{j=l+1}^{d-1}p_j(\xi_1,\ldots,\xi_d)\xi_j.\] 
    \end{enumerate}
    Additionally, for $w\in Z^\perp\setminus\{0\}$, the above equivalence is true with an orthogonal linear map $A$.
\end{crucial lemma}

\begin{proof} (i)$\Rightarrow$(ii):
	let $a_1,\ldots,a_l$ be an orthonormal basis of $Z$. Moreover, let $w\in\R^d$ be a unit vector with $P_m(w)\neq 0$ and let $W$ be a subspace complementary to $Z$ in $\R^d$ which contains $w$. Let $b_1,\ldots,b_{d-l}$ be an orthonormal basis of $W$ with $b_{d-l}=w$. If $w$ is orthogonal to $Z$, we can choose $W$ as the orthogonal complement of $Z$. We define the linear isomorphism $A:\R^d\rightarrow\R^d$ via
	\[A e_j=\begin{cases}
		a_j,& \text{ for }1\leq j\leq l,\\ b_{j-l},& \text{ for }l+1\leq j\leq d,
	\end{cases}\]
	where $e_j=(\delta_{j,n})_{(1\leq n\leq d)}$ (Kronecker's $\delta$) denotes the $j$-th standard basis vector of $\R^d$. If $W$ is the orthogonal complement of $Z$, $A$ is orthogonal. Then $P\circ A$ is a polynomial of degree $m$ with the principal part $(P\circ A)_m=P_m\circ A$ such that
	\[(P\circ A)_m(e_j)=(P_m\circ A)(e_j)=\begin{cases}
		P_m(a_j)=0,& \text{ for }1\leq j\leq l,\\ P_m(w)\neq 0,& \text{ for }j=d.
	\end{cases}\]
	Hence, by Remark \ref{rem:homogeneous polynomials}, 
	\begin{equation}\label{eq:decomposition 4}
		(P\circ A)_m(\xi)=c\xi_d^m+\sum_{j=0}^{m-1}q_j(\xi_1,\ldots,\xi_{d-1})\xi_d^j,
	\end{equation}
	where $c\in\C\setminus\{0\}$ and each $q_j(\xi_1,\ldots,\xi_{d-1})\in\C[\xi_1,\ldots,\xi_{d-1}]$ is of degree $m-j$ or zero polynomials.
	Additionally, $P_m$ vanishes on
	\[Z=\operatorname{span}\{a_1,\ldots,a_l\}=\operatorname{span}\{A e_1,\ldots,A e_l\}=A\left(\operatorname{span}\{e_1,\ldots,e_l\}\right)\]
	so that $(P\circ A)_m=P_m\circ A$ vanishes on
	\[\operatorname{span}\{e_1,\ldots,e_l\}=\{\xi\in\R^d\colon \xi_{l+1}=\ldots=\xi_d=0\}.\]
Moreover, by hypothesis, $D^k_w P_m$ vanishes on $Z$ for all $1\leq k\leq m-1$ and the chain rule gives
\begin{equation}\label{eq:D_d^k(P_m circ A}
D_d^k(P_m\circ A)=(D^k_w P_m)\circ A,
\end{equation}
so $D_d^k(P_m\circ A)$ vanishes on $\operatorname{span}\{e_1,\ldots,e_l\}$ for all $1\leq k\leq m-1$, as well. Consequently, according to Lemma \ref{lem:hilbert nullstellensatz}, $D_d^k(P_m\circ A)$ belongs to the ideal $(\xi_{l+1},\ldots,\xi_d)$ for $0\leq k\leq m-1$.

We now claim that $q_j\in (\xi_{l+1},\ldots,\xi_d)$ for all $0\leq j\leq m-1$. This will be shown by induction on $j$.
By (\ref{eq:decomposition 4}),
\begin{equation}\label{eq:derivatives of P circ A}
D_d^k(P\circ A)_m(\xi)=c\frac{m!}{(m-k)!}\xi_d^{m-k}+\sum_{j=k}^{m-1}\frac{j!}{(j-k)!}q_j(\xi_1,\ldots,\xi_{d-1})\xi_d^{j-k}.
\end{equation}
In particular, 
\[D_d^{m-1}(P_m\circ A)(\xi)=cm!\xi_d+(m-1)!q_{m-1}(\xi_1,\ldots,\xi_{d-1}).\]
But the polynomial $D_d^{m-1}(P_m\circ A)$ and the monomial $cm!\xi_d$ belong to the ideal $(\xi_{l+1},\ldots,\xi_d)$, so does the polynomial $q_{m-1}$. 
Assume that $q_k\in (\xi_{l+1},\ldots,\xi_{d-1})$ for all $k\leq j\leq m-1$. By (\ref{eq:derivatives of P circ A}),
\begin{align*}
(k-1)!q_{k-1}(\xi_1,\ldots,\xi_{d-1})=&D_d^{k-1}(P_m\circ A)(\xi)-c\frac{m!}{(m-k+1)!}\xi_d^{m-k+1}&\\
&-\sum_{j=k}^{m-1}\frac{j!}{(j-k+1)!}q_j(\xi_1,\ldots,\xi_{d-1})\xi_d^{j-k+1}.
\end{align*}
Consequently, by the inductive assumption and the fact that $D_d^{k-1}(P_m\circ A)$ and the monomial $\xi_d^{m-k+1}$ belong to the ideal $(\xi_{l+1},\ldots,\xi_d)$, we have $q_{k-1}\in (\xi_{l+1},\ldots,\xi_d)$, and our claim is proved. 

In fact, since $q_j$ does not depend on $\xi_d$, we concluded that $q_j\in (\xi_{l+1},\ldots,\xi_{d-1})$ for $0\leq j\leq m-1$. Hence
\[(P_m\circ A)(\xi)-c\xi_d^m=\sum_{j=0}^{m-1}q_j(\xi_1,\ldots,\xi_{d-1})\xi_d^j\in (\xi_{l+1},\ldots,\xi_{d-1}),\]
which completes the proof of implication (i)$\Rightarrow$(ii).

(ii)$\Rightarrow$(i): Since  
\[D_d^{k}(P_m\circ A)(\xi)=c\frac{m!}{(m-k)!}\xi_d^{m-k}+\sum_{j=l+1}^{d-1}(D_d^kp_j)(\xi)\xi_j,\]
the polynomial $D_d^{k}(P_m\circ A)$ vanishes on $\operatorname{span}\{e_1,\ldots,e_l\}$, for all $0\leq k\leq m-1$. Hence,  identity gives
(\ref{eq:D_d^k(P_m circ A}), 
\[(D_w^{k} P_m)(Z)=[D_d^{k}(P_m\circ A)](A^{-1}(Z))=[D_d^{k}(P_m\circ A)](\operatorname{span}\{e_1,\ldots,e_l\})=\{0\}\]
for all $0\leq k\leq m-1$.
Clearly $P_m(w)=(P_m\circ A)(e_d)=c\neq 0$, which completes the proof.
\end{proof}

\section{Spherical harmonics expansion in \texorpdfstring{$\mathscr{E}$}{E}}\label{appendix: spherical harmonics}

The purpose of this appendix is to prove Theorem \ref{theo:spherical harmonics expansion for smooth functions} below. We think that this theorem is known. However, as we were unable to find a reference, we include its proof for the reader's convenience. As usual, let $\mathcal{P}_k$ be the vector space of homogeneous polynomials of degree $k$ on $\R^d$, and let
\begin{eqnarray*}
	\mathcal{H}_k&=&\{P\in\mathcal{P}_k\colon \Delta P=0\},\\
	H_k&=&\{P_{|S^{d-1}}\colon P\in\mathcal{H}_k\},
\end{eqnarray*}
i.e.~$H_k$ is the space of spherical harmonics of degree $k$. By definition, for $Y\in H_k$ the function
\[\R^d\backslash\{0\}\rightarrow \C, x\mapsto |x|^kY\left(\frac{x}{|x|}\right)\]
can be extended (in a unique way) to a smooth function on $\R^d$ which belongs to $\mathcal{H}_k$. As is well known (see \cite[Corollary 2.55]{Folland1995})
\[\text{dim}\mathcal{H}_k=\text{dim}H_k=(2k+d-2)\frac{(k+d-3)!}{k!(d-2)!}=:d_k\]
and for $k\neq l$ the (closed) subspaces $H_k$ and $H_l$ of $L^2(S^{d-1})=L^2(S^{d-1},\sigma)$ are orthogonal, where $\sigma$ is the surface measure on $S^{d-1}$ (see \cite[Theorem 2.53]{Folland1995}). For $k\in\N_0$ we fix an arbitrary orthonormal basis $\{Y_{k,m}\colon m=1,\ldots, d_k\}$ of $H_k$.

Additionally, in order to simplify the notation, whenever this is convenient, for a function $f:S^{d-1}\rightarrow \C$ we set
\[f^*\colon\R^d\backslash\{0\}\rightarrow\C, x\mapsto f\left(\frac{x}{|x|}\right),\]
i.e.~we extend $f$ as a positively homogeneous function of degree 0 to $\R^d\backslash\{0\}$.

For $r\in(0,\infty]$ let $B_r=\{x\in\R^d\colon |x|<r\}$. Note that $B_\infty=\R^d$. Additionally, let $U\subset\R^q$ be open. We denote the elements from $U$ by $y$ and the elements from $\R^d$ by $x$. For $v\in \mathscr{E}(U\times B_r)$, we define
\begin{equation}\label{eq:definition coefficients of v}
	v_{k,m}:U\times(0,r)\rightarrow\C,(y,\rho)\mapsto\int_{S^{d-1}}v(y,\rho\omega)\overline{Y_{k,m}(\omega)}d\sigma(\omega),
\end{equation}
where $k\in\N_0, m\in\{1,\ldots,d_k\}$, and we sometimes write $a_{k,m}(v)$ instead of $v_{k,m}$. Then, $v_{k,m}\in\mathscr{E}(U\times(0,r))$, for $\beta\in\N_0^q$ we have $\partial_y^\beta a_{k,m}(v)=a_{k,m}(\partial_y^\beta v)$, and for every $(y,\rho)\in U\times (0,r)$ it holds (cf.~\cite[Theorem (2.53)]{Folland1995})
\[v(y,\rho\, \cdot)=\sum_{k=0}^\infty\sum_{m=1}^{d_k} v_{k,m}(y,\rho)Y_{k,m}\text{ in }L^2(S^{d-1})\]
and therefore
\[\sum_{k=0}^\infty\sum_{m=1}^{d_k} |v_{k,m}(y,\rho)|^2=\int_{S^{d-1}} |v(y,\rho\omega)|^2 d\sigma(\omega).\]
For $K\subset U$ compact and $r_0\in (0,r)$ the latter implies
\begin{eqnarray*}
	\int_{K\times B_{r_0}} |v(y,x)|^2 d(y,x)&=&\int_K\int_0^{r_0}\rho^{d-1} \int_{S^{d-1}} |v(y,\rho\omega)|^2 d\sigma(\omega) d\rho dy\\
	&=&\int_K\int_0^{r_0}\rho^{d-1} \sum_{k=0}^\infty\sum_{m=1}^{d_k} |v_{k,m}(y,\rho)|^2 d\rho dy
\end{eqnarray*}
so that $(y,\rho)\mapsto \rho^{d-1} \sum_{k=0}^\infty\sum_{m=1}^{d_k} |v_{k,m}(y,\rho)|^2$ belongs to  $L^1(K\times (0,r_0))$. By the Dominated Convergence Theorem, we conclude
\begin{eqnarray*}
	&&\lim_{N\rightarrow\infty}\int_{K\times B_{r_0}}\left| v(y,x)-\sum_{k=0}^N\sum_{m=1}^{d_k} v_{k,m}(y,|x|) Y_{k,m}^*(x)\right|^2 d(y,x)\\
	&=&\lim_{N\rightarrow\infty}\int_{K}\int_0^{r_0}\rho^{d-1} \int_{S^{d-1}}\left| v(y,\rho\omega)-\sum_{k=0}^N\sum_{m=1}^{d_k} v_{k,m}(y,\rho) Y_{k,m}(\omega)\right|^2d\sigma(\omega)d\rho dy\\
	&=&\lim_{N\rightarrow\infty}\int_{K}\int_0^{r_0}\rho^{d-1}  \sum_{k=N+1}^\infty\sum_{m=1}^{d_k} |v_{k,m}(y,\rho)|^2  d\rho dy\\
	&=&\int_{K}\int_0^{r_0}\rho^{d-1}  \lim_{N\rightarrow\infty}\sum_{k=N+1}^\infty\sum_{m=1}^{d_k} |v_{k,m}(y,\rho)|^2  d\rho dy=0.
\end{eqnarray*}
Thus,
\begin{equation}\label{eq:convergence in local L2}
	v=\sum_{k=0}^\infty \sum_{m=1}^{d_k} v_{k,m} Y^*_{k,m}\text{ in }L^2(K\times \overline{B_{r_0}}).
\end{equation}
We point out that a more appropriate notation is
\[(y,x)\mapsto v_{k,m}\otimes Y_{k,m}(y,|x|,x)\]
instead of $v_{k,m}Y_{k,m}^*$but we want to avoid exurberant notation, so we stick to a notation that is close to the one used in the literature. Our first objective is to prove the following lemma.

\begin{lemma}\label{lemma:spherical harmonics expansion}
	Let $r\in(0,\infty]$ and let $v\in\mathscr{E}(U\times B_r)$. Then, with the notation introduced above, for every $k\in N_0, m\in {1,\ldots, d_k}$ there are $\tilde{v}_{k,m}\in\mathscr{E}(U\times (-r,r))$ such that the following conditions hold.
	\begin{itemize}
		\item[(i)] For every $(y,\rho)\in U\times (0,r)$ we have $v_{k,m}(y,\rho)=\tilde{v}_{k,m}(y,\rho)$.
		\item[(ii)] The functions
		\[U\times (B_r\backslash\{0\})\rightarrow\C,(y,x)\mapsto \tilde{v}_{k,m}(y,|x|) |x|^{-k}\]
		extend to smooth functions on $U\times B_r$. 
 		\item[(iii)] The functions
		\[U\times (B_r\backslash\{0\})\rightarrow\C,(y,x)\mapsto \tilde{v}_{k,m}(y,|x|)Y_{k,m}^*(x)\]
		extend to smooth functions on $U\times B_r$. 
	\end{itemize}
\end{lemma}

By the above lemma, the functions
\[U\times B_r\backslash\{0\}\rightarrow\C,(y,x)\mapsto  v_{k,m}(y,|x|)Y^*_{k,m}(x)\]
extend to smooth functions on $U\times B_r$; we denote these extensions again by $v_{k,m}Y^*_{k,m}$.

In the proof Lemma \ref{lemma:spherical harmonics expansion} we will apply the following elementary result. We include its proof for the reader's convenience.

\begin{proposition}\label{prop:regularity of factor}
	Let $f\in \mathscr{E}(U\times [0,r))$ and let $m\in\N$. Then, the function
	\[h:U\times B_r\rightarrow\C,(y,x)\mapsto f(y,|x|)|x|^{2m}\]
	is $(2m-1)$-times continuously differentiable.
\end{proposition}

\begin{proof}
	Clearly, $h\in C^{\infty}(U\times B_r\backslash\{0\})$.
	
	By induction on $k=|\alpha|$, $\alpha\in\N_0$, one shows that for every $\alpha\in\N_0^d$ there are $Q_{\alpha,1}, \ldots, Q_{\alpha,|\alpha|}\in \C[x_1,\ldots x_d]$ homogeneous of degree $|\alpha|$ such that for every $\phi\in C^{\infty}([0,r))$
	\[\partial_x^\alpha \left(\phi(|x|)\right)
	=\sum_{j=0}^{|\alpha|}\phi^{(j)}(|x|)\frac{Q_{\alpha,j}(x)}{|x|^{|\alpha|}},\quad x\in B_r\backslash\{0\}.
	\]  
	Therefore, for $\alpha\in\N_0^d$ there is $C_\alpha>0$ independent of $\phi$ such that
	\[\left|\partial_x^\alpha \left(\phi(|x|)\right)\right|
	\leq C_\alpha\sum_{j=0}^{|\alpha|}\left|\phi^{(j)}(|x|)\right|,\quad x\in B_r\backslash\{0\}.
	\]  
	Additionally, since $|x|^{2m}$ is a homogeneous polynomial of degree $2m$, for $\alpha\in\N_0^d$ with $|\alpha|\leq 2m$, $\partial_x^\alpha\left(|x|^{2m}\right)$ is a homogeneous polynomial of degree $2m-|\alpha|$. In particular, for every such $\alpha\in\N_0^d$ there is $D_\alpha>0$ with
	\[\forall\,x\in\R^d\colon \partial_x^\alpha\left(|x|^{2m}\right)\leq D_\alpha|x|^{2m-|\alpha|}.\]
	By the above, we conclude for $\gamma\in\N_0^q$ and $\alpha\in\N_0^d$, $|\alpha|\leq 2m-1$ that for every $(y,x)\in U\times B_r\backslash\{0\}$
	\begin{eqnarray}\label{eq:regularity of factor}
		\left|\partial_y^\gamma\partial_x^\alpha h(y,x)\right|&\leq&\sum_{\beta\leq \alpha}\binom{\alpha}{\beta}C_{\alpha-\beta}D_{\beta}\sum_{j=0}^{|\alpha|-|\beta|}\left|\partial_y^\gamma\partial_\rho^j f(y,|x|)\right| |x|^{2m-|\beta|},
	\end{eqnarray}
	where by $\rho$ we denote the last variable from $[0,r)$ of $f$. In particular, by defining $\partial_y^\gamma\partial_x^\alpha h(y,0)=0$, we obtain a continuous extension of $\partial_y^\gamma\partial_x^\alpha h$ from $U\times B_r\backslash\{0\}$ to $U\times B_r$. Moreover, whenever $|\alpha|\leq 2m-2$ it follows from \eqref{eq:regularity of factor} that $\partial_y^\gamma \partial_x^\alpha h$ is continuously differentiable with respect to $x$. This proves the claim.
\end{proof}

\begin{proof}[Proof of Lemma \ref{lemma:spherical harmonics expansion}]
	By Taylor's formula, for $N\in\N$ we have
	\[v(y,x)=\sum_{l=0}^N \sum_{\alpha\in\N_0^d, |\alpha|= l}\frac{1}{\alpha!}\partial_x^\alpha v(y,0)x^\alpha + R_N(v)(y,x)\]
	for the smooth function
	\begin{equation}\label{eq:Taylor rest}
		R_N(v)(y,x)=\sum_{|\alpha|=N+1} \frac{N+1}{\alpha!}\int_0^1 (1-t)^N \partial_x^\alpha v(y,tx)dt\, x^\alpha.
	\end{equation}
	By \cite[Proof of Proposition 2.49 and Corollary 2.50]{Folland1995}, there are smooth functions $P_{l-2j}$ on $U\times\R^d$, $l=1,\ldots, N$, $j=0,\ldots,\lfloor l/2\rfloor$, such that $P_{l-2j}(y,\cdot)$ is a harmonic homogeneous polynomial of degree $l-2j$ on $\R^d$ such that
	\[\sum_{\alpha\in\N_0^d, |\alpha|= l}\frac{1}{\alpha!}\partial_x^\alpha v(y,0)x^\alpha =\sum_{j=0}^{\lfloor l/2\rfloor} |x|^{2j} P_{l-2j}(y,x).\]
	By the choice of the $Y_{l,m}$, for $(y,x)\in U\times(B_r\backslash\{0\})$ it holds
	\[P_{l-2j}(y,x)=\sum_{m=1}^{d_{l-2j}} c_{l-2j,m}(y)|x|^{l-2j} Y_{l-2j,m}\left(\frac{x}{|x|}\right),\]
	where
	\[c_{l-2j,m}(y)=\int_{S^{d-1}} P_{l-2j,m}(y,\omega)\overline{Y_{l-2j,m}(\omega)}d\sigma(\omega)\]
	is easily seen to be a smooth function on $U$. Hence, for $(y,x)\in U\times(B_r\backslash\{0\})$ we obtain
	\[v(y,x)=\sum_{l=0}^N \sum_{j=0}^{\lfloor l/2\rfloor} \sum_{m=1}^{d_{l-2j}} c_{l-2j,m}(y)|x|^{l} Y_{l-2j,m}\left(\frac{x}{|x|}\right) + R_N(v)(y,x).\]
	Relabeling $k=l-2j\Leftrightarrow k+2j=l$, from $k+2j\in\{0,\ldots, N\}$ it follows $j\in\{0,\ldots,\lfloor (N-k)/2\rfloor\}$, so that for $(y,x)\in U\times(B_r\backslash\{0\})$ we have
	\begin{eqnarray}\label{eq:spherical harmonics expansion 1}
		v(y,x)&=&\sum_{l=0}^N \sum_{j=0}^{\lfloor l/2\rfloor} \sum_{m=1}^{d_{l-2j}} c_{l-2j,m}(y)|x|^{l} Y_{l-2j,m}\left(\frac{x}{|x|}\right) + R_N(v)(y,x)\nonumber\\
		&=&\sum_{k=0}^N\sum_{j=0}^{\lfloor(N-k)/2\rfloor}\sum_{m=1}^{d_k} c_{k,m}(y)|x|^{k+2j}Y_{k,m}\left(\frac{x}{|x|}\right)+R_N(v)(y,x)\nonumber\\
		&=&\sum_{k=0}^N\sum_{m=1}^{d_k}\left(\sum_{j=0}^{\lfloor(N-k)/2\rfloor} c_{k,m}(y)|x|^{2j}\right) |x|^k Y_{k,m}\left(\frac{x}{|x|}\right)+R_N(v)(y,x)\nonumber\\
		&=&\sum_{k=0}^N\sum_{m=1}^{d_k} b_{N,k,m}(y,|x|)|x|^k Y_{k,m}\left(\frac{x}{|x|}\right)+R_N(v)(y,x),
	\end{eqnarray}
	where we set for $N\in\N_0, k\in\{0,\ldots,N\}, m\in\{1,\ldots,d_k\}$
	\begin{equation}\label{eq:spherical harmonics expansion 0}
		b_{N,k,m}:U\times\R\rightarrow\C, (y,r)\mapsto \sum_{j=0}^{\lfloor(N-k)/2\rfloor} c_{k,m}(y) r^{2j}
	\end{equation}
	which is a smooth function. By the properties of spherical harmonics, the functions
	\begin{equation}\label{eq:definition of v_N}
		v_N:U\times(\R^d\backslash\{0\})\rightarrow\C,(y,x)\mapsto \sum_{k=0}^N\sum_{m=1}^{d_k} b_{N,k,m}(y,|x|)|x|^k Y_{k,m}\left(\frac{x}{|x|}\right)
	\end{equation}
	extend (in a unique way) to a smooth function on $U\times\R^d$ which we denote again by $v_N$. 
	
	Additionally, by the discussion preceding Lemma \ref{lemma:spherical harmonics expansion} applied to $R_N(v)$ instead of $v$, for every compact subset $K$ of $Y$ and each $r_0\in (0,r)$ it holds
	\begin{equation}\label{eq:expansion of Taylor rest}
		R_N(v)(y,x)=\sum_{k=0}^\infty\sum_{m=1}^{d_k}r_{N,k,m}(y,|x|) Y_{k,m}\left(\frac{x}{|x|}\right)\text{ in }L^2(K\times \overline{B_{r_0}})
	\end{equation}
	where
	\begin{eqnarray}\label{eq:coefficients of Taylor rest}
		\begin{gathered}
		r_{N,k,m}: U\times (-r,r)\rightarrow\C,(y,\rho)\mapsto \int_{S^{d-1}}R_N(v)(y,\rho\omega)\overline{Y_{k,m}(\omega)}d\sigma(\omega)\\
		=\sum_{|\alpha|=N+1}\frac{N+1}{\alpha!} \int_{S^{d-1}} \int_0^1 (1-t)^N \partial_x^\alpha v(y,t\rho\omega)dt\, \omega^\alpha \overline{Y_{k,m}(\omega)} d\sigma(\omega)\rho^{N+1}
		\end{gathered}
	\end{eqnarray}
	due to \eqref{eq:Taylor rest}.
    
	Trivially, by \eqref{eq:definition coefficients of v}, \eqref{eq:spherical harmonics expansion 1}, \eqref{eq:expansion of Taylor rest}, and \eqref{eq:coefficients of Taylor rest}, for $N\geq k$ and each $(y,\rho)\in Y\times(0,r)$ it holds
	\begin{eqnarray}\label{eq:decomposition of a_{l,m}}
		v_{k,m}(y,\rho)&=&b_{N,k,m}(y,\rho)\rho^k+r_{N,k,m}(y,\rho)\nonumber\\
		&=&\left(b_{N,k,m}(y,\rho)+r_{N,k,m}(y,\rho)\rho^{-k}\right)\rho^k.
	\end{eqnarray}
    Since $c_{k,m}\in \mathscr{E}(U)$, the function
	\[U\times B_r\rightarrow\C, (y,x)\mapsto b_{N,k,m}(y,x)=\sum_{j=0}^{\lfloor (N-k)/2\rfloor} c_{k,m}(y)|x|^{2j} \]
	is smooth. Because the function $|x|^k Y_{k,m}(x/|x|)$ extends to a harmonic polynomial, in order to complete the proof, it suffices to show that, given $l$, for sufficiently large $N$
	\[U\times B_r\times\C, (y,x)\mapsto r_{N,k,m}(y,|x|)|x|^{-k}\]
	is a $C^l$-function. Because in \eqref{eq:decomposition of a_{l,m}} we can choose $N$ as large as we like, this will prove the existence of functions $\tilde{v}_{k,m}\in\mathscr{E}(U\times B_r)$ with the desired properties.
	
	For $(y,x)\in U\times B_r$ and $N>k$ we have
	\begin{gather*}
		r_{N,k,m}(y,|x|)|x|^{-k}=\sum_{|\alpha|=N+1}\frac{N+1}{\alpha!} \int_{S^{d-1}} \int_0^1 (1-t)^N \partial_x^\alpha v(y,t|x|\omega)dt\times \\ \times \omega^\alpha \overline{Y_{k,m}(\omega)} d\sigma(\omega)|x|^{N+1-k}.
	\end{gather*}
	It follows from Proposition \ref{prop:regularity of factor} that for $N=2l+k-1$ the above function is a $C^l$-function. As already explained, by \eqref{eq:decomposition of a_{l,m}} this entails the existence of $\tilde{v}_{k,m}\in \mathscr{E}(U\times (-r,r))$ with the desired properties.	
\end{proof}

\begin{proposition}\label{prop:cofficients commute with Laplace}
	Let $r\in(0,\infty]$ and let $v\in\mathscr{E}(U\times B_r)$.  Then, for $k\in\N_0$ and $m\in\{1,\ldots,d_k\}$ it holds
	\[\forall\,(y,x)\in U\times B_r\colon \Delta_x\left(a_{k,m}(v)(y,|x|)Y^*_{k,m}(x)\right)=a_{k,m}\left(\Delta_x v\right)(y,|x|)Y^*_{k,m}(x),\]
	where $\Delta_x=\sum_{j=1}^d\frac{\partial^2}{\partial x_j^2}$ denotes the Laplacian in the $x$-variables.
\end{proposition}

\begin{proof}
	For $(y,x)\in U\times B_r\backslash\{0\}$, by \cite[Lemma 2.63 - note that there is typo]{Folland1995} it holds
	\begin{eqnarray*}
		&&\Delta_x\left(a_{k,m}(y,|x|)Y^*_{k,m}(x)\right)\\
		&=&\left.\left(\partial_\rho^2 a_{k,m}(y,\rho)+\frac{d-1}{\rho}\partial_\rho a_{k,m}(y,\rho)-\frac{k(k+d-2)}{\rho^2}a_{k,m}(y,\rho)\right)\right|_{\rho=|x|}Y_{k,m}^*(x).
	\end{eqnarray*}
	Moreover, denoting the Laplace-Beltrami operator on $S^{d-1}$ by $\Delta_{S^{d-1}}$, it is well know that $\Delta_{S^{d-1}}Y_{k,m}=-k(k+d-2)Y_{k,m}$ (see \cite[Equation (\S15.2)]{Mueller1998}). Then,
	\begin{eqnarray*}
		&&\partial_\rho^2 a_{k,m}(y,\rho)+\frac{d-1}{\rho}\partial_\rho a_{k,m}(y,\rho)-\frac{k(k+d-2)}{\rho^2}a_{k,m}(y,\rho)\\
		&=&\int_{S^{d-1}} \left(\partial_\rho^2 v(y,\rho\omega)+\frac{d-1}{\rho}\partial_\rho v(y,\rho\omega)-\frac{k(k+d-2)}{\rho^2}v(y,\rho\omega)\right)\overline{Y_{k,m}(\omega)}d\sigma(\omega)\\
		&=&\int_{S^{d-1}} \left(\partial_\rho^2 v(y,\rho\omega)+\frac{d-1}{\rho}\partial_\rho v(y,\rho\omega)\right)\overline{Y_{k,m}(\omega)}d\sigma(\omega)\\
		&&+\int_{S^{d-1}} \frac{1}{\rho^2}v(y,\rho\omega)\Delta_{S^{d-1}}\overline{Y_{k,m}(\omega)}d\sigma(\omega)\\
		&=&\int_{S^{d-1}} \left(\partial_\rho^2 v(y,\rho\omega)+\frac{d-1}{\rho}\partial_\rho v(y,\rho\omega)+\frac{1}{\rho^2}\Delta_{S^{d-1}}v(y,\rho\omega)\right)\overline{Y_{k,m}(\omega)}d\sigma(\omega)\\
		&=&\int_{S^{d-1}}\Delta_x v(y,\rho\omega)\overline{Y_{k,m}(\omega)}d\sigma(\omega)\\
		&=&a_{k,m}(\Delta_x v)(y,\rho),
	\end{eqnarray*}
	where we have used an immediate consequence of \cite[\S 14, Lemma 1]{Mueller1998} and the well-known expression for the Laplace operator in spherical coordinates (\cite[Equation (\S 14.17)]{Mueller1998}). Combining these two equalities shows that the claimed equality holds on the dense subset  $U\times B_r\backslash\{0\}$ of $U\times B_r$ and therefore, everywhere in $U\times B_r$.
\end{proof}

\begin{theorem}\label{theo:spherical harmonics expansion for smooth functions}
	Let $r\in (0,\infty]$, $B_r=\{x\in\R^d\colon |x|<r\}$ and let $U\subset \R^q$ be open. Then, for every $v\in\mathscr{E}(U\times B_r)$ we have
	\[v=\sum_{k=0}^\infty\sum_{m=1}^{d_k} v_{k,m} Y_{k,m}^*\text{ in }\mathscr{E}(U\times B_r).\]
\end{theorem}

\begin{proof}
	By Proposition \ref{lemma:spherical harmonics expansion}, for each $v\in\mathscr{E}(U\times B_r)$ and every $N\in \N$, the function
	\[\sum_{k=0}^N\sum_{m=1}^{d_k} v_{k,m} Y_{k,m}^*:U\times B_r\rightarrow\C, (y,x)\mapsto \sum_{k=0}^N\sum_{m=1}^{d_k} v_{k,m}(y,|x|) Y_{k,m}^*(x)\]
	is smooth.
	
	Let $K$ be a compact subset of $U$ and let $r_0\in (0,r)$. We fix $\psi\in \mathscr{D}(U)$ with $\psi=1$ in a neighborhood of $K$. Moreover, let $r_1\in(r_0,r)$ and we fix $\psi\in \mathscr{D}(\R)$ with $\supp\varphi\subset [-r_1^2,r_1^2]$ and $\varphi=1$ in a neighborhood of $[0,r_0^2]$. For $v\in\mathscr{E}(U\times B_r)$ the function
	\[v_{\psi,\varphi}\colon \R^q\times\R^d\rightarrow\C, (y,x)\mapsto \psi(y)\varphi(|x|^2)v(y,x)\]
	is smooth and compactly supported, and $v=v_{\psi,\varphi}$ on $K\times \overline{B_{r_0}}$. For $k\in \N_0$, $m=1,\ldots,d_k$ we have
	\[\forall\,(y,x)\in U\times B_r\colon a_{k,m}\left(v_{\psi,\varphi}\right)(y,x)=\psi(y)\varphi(|x|^2)a_{k,m}(v)(y,x),\]
	in particular
	\[\forall\,(y,x)\in K\times \overline{B_{r_0}}\colon a_{k,m}\left(v_{\psi,\varphi}\right)(y,x)=a_{k,m}(v)(y,x).\]
	Additionally, it holds
	\begin{equation}\label{eq:commutation of psi, varphi}
		v_{\psi,\varphi}-\sum_{k=0}^N\sum_{m=1}^{d_k}a_{k,m}(v_{\psi,\varphi})Y_{k,m}^*=\left(v-\sum_{k=0}^N\sum_{m=1}^{d_k}a_{k,m}(v)Y_{k,m}^*\right)_{\psi,\varphi}.
	\end{equation}
	Abbreviating the support of $\psi$ by $\tilde{K}$, $\tilde{K}\times\overline{B_{r_1}}$ is a compact subset of $U\times B_r$, so that by \eqref{eq:convergence in local L2} and Proposition \ref{prop:cofficients commute with Laplace} we conclude for every $M_x,M_y\in \N_0$
	\begin{equation}\label{eq:local L2-convergence}
		\lim_{N\rightarrow\infty}(1-\Delta_y)^{M_y}(1-\Delta_x)^{M_x}\left(v_{\psi,\varphi}-\sum_{k=0}^N\sum_{m=1}^{d_k}a_{k,m}(v_{\psi,\varphi})Y_{k,m}^*\right)=0
	\end{equation}
	in $L^2(\tilde{K}\times\overline{B_{r_1}})$. 
	
	For arbitrary $f\in\mathscr{E}(U\times B_r)$, $\alpha\in\N_0^d$, $\beta\in\N_0^q$, and $(y,x)\in K\times \overline{B_{r_0}}$, by the Fourier inversion formula, the Cauchy-Schwarz inequality, and Plancherel's Theorem we obtain
	\begin{eqnarray*}
		|\partial_y^\beta\partial_x^\alpha f(y,x)|&=&	|\partial_y^\beta\partial_x^\alpha \left(f_{\psi,\varphi}(y,x)\right)|\\
		&=&(2\pi)^{d+q}\left|\int_{\R^q\times\R^d}e^{i\langle (y,x),(\eta,\xi)\rangle}\mathscr{F}\left(\partial_y^\beta\partial_x^\alpha\left(f_{\psi,\varphi}\right)\right)(\eta,\xi)d(\eta,\xi)\right|\\
		&\leq&(2\pi)^{d+q}\left( \int_{\R^q\times\R^d} \frac{|\eta|^{2|\beta|}}{(1+|\eta|^2)^{M_y}}\frac{|\xi|^{2|\alpha|}}{(1+|\xi|^2)^{M_x}}d(\eta,\xi)\right)^{1/2}\times\\
		&&\times \left( \int_{\R^q\times\R^d} \left| (1+|\eta|^2)^{M_y}(1+|\xi|^2)^{M_x}\mathscr{F}(f_{\psi,\varphi})(\eta,\xi)\right|^2d(\eta,\xi)\right)^{1/2}\\
		&=&(2\pi)^{d+q}\left( \int_{\R^q\times\R^d} \frac{|\eta|^{2|\beta|}}{(1+|\eta|^2)^{M_y}}\frac{|\xi|^{2|\alpha|}}{(1+|\xi|^2)^{M_x}}d(\eta,\xi)\right)^{1/2}\times\\
		&&\times\left\| (1-\Delta_y)^{M_y}(1-\Delta_x)^{M_x}f_{\psi,\varphi}\right\|_{L^2(\R^q\times\R^d)}\\
		&=&(2\pi)^{d+q}\left( \int_{\R^q\times\R^d} \frac{|\eta|^{2|\beta|}}{(1+|\eta|^2)^{M_y}}\frac{|\xi|^{2|\alpha|}}{(1+|\xi|^2)^{M_x}}d(\eta,\xi)\right)^{1/2}\times\\
		&&\times\left\| (1-\Delta_y)^{M_y}(1-\Delta_x)^{M_x}f_{\psi,\varphi}\right\|_{L^2(\tilde{K}\times\overline{B_{r_1}})}.
	\end{eqnarray*}
	Since the integral is finite whenever $M_x>|\alpha|+d/2$ and $M_y>|\beta|+q/2$, by the previous inequality, for $l\in \N_0$ there is $C_l$ such that for every $f\in\mathscr{E}(U\times B_r)$
	\begin{eqnarray*}
		\|f\|_{l,K\times\overline{B_{r_0}}}&=&\max_{(y,x)\in K\times\overline{B_{r_0}}, |\alpha|+|\beta|\leq l}|\partial_y^\beta\partial_x^\alpha f(y,x)|\\
		&\leq& C_l \left\| (1-\Delta_y)^{l+1+q/2}(1-\Delta_x)^{l+1+d/2}f_{\psi,\varphi}\right\|_{L^2(\tilde{K}\times\overline{B_{r_1}})}.
	\end{eqnarray*}
	Applying this inequality to $f=(v-\sum_{k=0}^N\sum_{m=1}^{d_k}a_{k,m}(v)Y_{k,m}^*)$ the theorem follows from \eqref{eq:commutation of psi, varphi}, \eqref{eq:local L2-convergence} and the arbitrariness of $K, r_0$ and $l$.
\end{proof}

\begin{remark}\label{rem:ODEs for coefficients in spherical harmonics expansion}
	Let $u\in\mathscr{D}'(\R^d)$ satisfy $\Delta u+\lambda^2 u=0$ in $\R^d$, $\lambda\in\C\backslash\{0\}$. Due to the hypoellipticity of $\Delta +\lambda^2$, $u\in\mathscr{E}(\R^d)$ and it follows from Theorem \ref{theo:spherical harmonics expansion for smooth functions} that
	\begin{eqnarray*}
		0&=&\sum_{k=0}^\infty\sum_{m=1}^{d_k}(\Delta+\lambda^2)\left(u_{k,m}Y^*_{k,m}\right)\\
		&=&\sum_{k=0}^\infty\sum_{m=1}^{d_k} \left( \frac{d}{d\rho^2}u_{k,m}+\frac{d-1}{\rho}\frac{d}{d\rho} u_{k,m}+\left(\lambda^2-\frac{k(k+d-2)}{\rho^2}\right)u_{k,m}\right) Y_{k,m}^*
	\end{eqnarray*}
	in $\mathscr{E}(\R^d)$. In particular, by the orthonormality of the $Y_{k.m}$'s in $L^2(S^{d-1})$ and Lemma \ref{lemma:spherical harmonics expansion} (i), for every $k,m$, the smooth function $u_{k,m}$ solves the ordinary differential equation
	\[f''(\rho)+\frac{d-1}{\rho}f'(\rho)+\left(\lambda^2-\frac{k(k+d-2)}{\rho^2}\right) f(\rho)=0\text{ in }[0,\infty)\]
	which is essentially a Bessel equation. As in \cite[p.~106]{Folland1995} it follows that
	\[u_{k,m}(\rho)=c_{k,m}\rho^{(2-d)/2}J_{k+(d-2)/2}(\lambda\rho),\]
	where $c_{k,m}\in\C$ and where $J_\alpha$ is the Bessel function of the first kind of order $\alpha$.
\end{remark}

\section*{Acknowledgements}

The research on Runge pairs for systems of partial differential equations presented in Section \ref{sec:systems} was initiated during a visit of the second author to the Instituto de Ciencias Matem\'aticas - ICMAT, Madrid, Spain. The second author would like to thank Daniel Peralta-Salas for the invitation and ICMAT for its warm hospitality.

\bibliographystyle{plain}
\bibliography{references}

\end{document}